\documentclass{daj}

\dajAUTHORdetails{%
  title = {Hypergraph Removal Lemmas via Robust Sharp Threshold Theorems}, 
  author = {Noam Lifshitz},
  plaintextauthor = {Noam Lifshitz},
}   

\dajEDITORdetails{%
   year={2020},
   number={11},
   received={21 September 2018},   
   published={7 August 2020},  
   doi={10.19086/da.14165},       
}   

\usepackage{mathtools}
\usepackage{amsfonts}
\usepackage{amsmath}
\usepackage{amssymb}
\usepackage{amsthm}
\usepackage{bbm}
\usepackage{amsbsy}
\usepackage{amstext}
\usepackage{mathdots}

\newtheorem{thm}{Theorem}
\newtheorem{problem}{Problem}
\newtheorem{defn}{Definition}
\newtheorem{example}{Example}
\newtheorem{prop}{Proposition}
\newtheorem{rem}{Remark}
\newtheorem{fact}{Fact}
\newtheorem{lem}{Lemma}
\newtheorem{claim}{Claim}
\newtheorem{cor}{Corollary}
\newtheorem*{thm*}{Theorem}
\newtheorem*{prop*}{Proposition}

\begin{document}
\global\long\def\f{\mathcal{F}}
\global\long\def\pn{\mathcal{P}\left(\left[n\right]\right)}
\global\long\def\g{\mathcal{G}}
\global\long\def\s{\mathcal{S}}
\global\long\def\j{\mathcal{J}}
\global\long\def\d{\mathcal{D}}
\global\long\def\Inf{}
\global\long\def\p{\mathcal{P}}
\global\long\def\h{\mathcal{H}}
\global\long\def\n{\mathbb{N}}
\global\long\def\a{\mathcal{A}}
\global\long\def\b{\mathcal{B}}
\global\long\def\c{\mathcal{C}}
\global\long\def\e{\mathbb{E}}
\global\long\def\ffou{\hat{f_{p}}}
\global\long\def\gfouq{\hat{g}_{q}}
\global\long\def\cps{\chi_{S}^{p}}
\global\long\def\cqs{\chi_{S}^{q}}
\global\long\def\tra{\mathsf{T}_{p\to q}}
\global\long\def\art{\mathsf{T}^{q\to p}}
\global\long\def\var{\mathrm{Var}}
\global\long\def\x{\mathbf{x}}
\global\long\def\y{\mathbf{y}}
\global\long\def\z{\mathbf{z}}
\global\long\def\c{\mathbf{c}}
\global\long\def\av{\mathsf{A}}
\global\long\def\chop{\mathrm{Chop}}
\global\long\def\stab{\mathrm{Stab}}
\global\long\def\t{\mathsf{T}}
\global\long\def\binom#1#2{{#1 \choose #2}}%

\begin{frontmatter}[classification=text]


\author[noam]{Noam Lifshitz}

\begin{abstract}
The classical sharp threshold theorem of Friedgut and Kalai (1996)
asserts that any symmetric monotone function $f\colon\left\{ 0,1\right\} ^{n}\to\left\{ 0,1\right\} $
exhibits a sharp threshold phenomenon. This means that the expectation
of $f$ with respect to the biased measure $\mu_{p}$ increases rapidly
from 0 to 1 as $p$ increases. 

In this paper we present `robust' versions of the theorem, which assert
that it holds also if the function is `almost' monotone, and admits
a much weaker notion of symmetry. Unlike the original proof of the
theorem which relies on hypercontractivity, our proof relies on a
`regularity' lemma (of the class of Szemer̩di's regularity lemma and
its generalizations) and on the `invariance principle' of Mossel,
O'Donnell, and Oleszkiewicz which allows (under certain conditions)
replacing functions on the cube $\{0,1\}^{n}$ with functions on Gaussian
random variables.

The hypergraph removal lemma of Gowers (2007) and independently of
Nagle, R\"{o}dl, Schacht, and Skokan (2006) says that if a $k$-uniform
hypergraph on $n$ vertices contains few copies of a fixed hypergraph
$H$, then it can be made $H$-free by removing few of its edges.
While this settles the `hypergraph removal problem' in the case where
$k$ and $H$ are fixed, the result is meaningless when $k$ is large
(e.g. $k>\log\log\log n$).

Using our robust version of the Friedgut--Kalai Theorem, we obtain
a hypergraph removal lemma that holds for $k$ up to linear in $n$
for a large class of hypergraphs. These contain all the hypergraphs
such that both their number of edges and the sizes of the intersections
of pairs of their edges are upper bounded by some constant.

\end{abstract}
\end{frontmatter}

\section{Introduction}

\subsection{Problems on $\mathcal{H}$-free families }

For any set $V$ we use $\binom{V}{k}$ to denote the family of all
subsets of $V$ of size $k$. Any $\mathcal{H}\subseteq\binom{V}{k}$
is called a \emph{$k$-uniform hypergraph }or a $k$\emph{-uniform
family} on the \emph{vertex set} $V$, and the elements of $\mathcal{H}$
are its \emph{edges}. We write $\left[n\right]$ for the set $\left\{ 1,\ldots,n\right\} .$ 

The celebrated Mantel's Theorem \cite{mantel1907problem} from 1907
says that the largest triangle free graph $G\subseteq\binom{\left[n\right]}{2}$
is the balanced complete bipartite graph. In 1941, Tur\'{a}n \cite{turan1941basic}
generalized Mantel's Theorem from triangles to cliques. He raised
the following problem known as the Tur\'{a}n problem for hypergraphs. 
\begin{problem}
Given a hypergraph $\mathcal{H}\subseteq\binom{V}{k},$ determine
the value $\mathrm{ex}\left(n,\mathcal{H}\right)$ of the largest
$\h$-free family in $\binom{\left[n\right]}{k}.$
\end{problem}

An $\mathcal{H}$-free family of size $\mathrm{ex}\left(n,\mathcal{H}\right)$
is called an \emph{extremal $\h$-free family}. The Tur\'{a}n problem
is one of the most prominent problems in extremal combinatorics, and
it includes many of the well studied problems in this area. (See the
excellent survey of Keevash \cite{keevash2011hypergraph}, and the
more recent survey of Mubayi and Verstra\"{e}te \cite{mubayi2016survey}.)
One example, is the classical Erd\H{o}s-Ko-Rado Theorem (EKR Theorem)
\cite{erdos1961intersection} from 1961, which determines the largest
size of a family $\f\subseteq\binom{\left[n\right]}{k}$ such that
each two of its edges have nonempty intersection. It can be rephrased
by saying that $\mathrm{ex}\left(n,\mathcal{M}_{2}\right)=\binom{n-1}{k-1}$
for any $n\ge2k$, where $\mathcal{M}_{2}$ is the $k$-uniform hypergraph
that consists of two disjoint edges. 

Another well-studied set of Tur\'{a}n problems is the Forbidden intersection
problem of Erd\H{o}s and SÌ?s \cite{erdHos1975problems} from 1975.
It concerns determining $\mathrm{ex}\left(n,\h\right)$ in the case
where the forbidden hypergraph $\h$ consists of two edges with a
given intersection, (see \cite{ellis2016stabilityfor,frankl1977families,frankl1985forbidding,frankl1987forbidden,keevash2006set}). 

Besides the Tur\'{a}n problem for hypergraphs of determining $\mathrm{ex}\left(n,\h\right)$
for various hypergraphs $\h$, the research of $\h$-free families
in recent years concentrated on the following types of problems. 
\begin{enumerate}
\item \textbf{$0.99$}-\textbf{type stability results. }Suppose that $\f$
is a nearly extremal $\h$-free family in the sense that its size
is close to $\mathrm{ex}\left(n,\mathcal{H}\right)$. Can we say that
$\f$ is close to an $\mathcal{H}$-free family? (See e.g. \cite{frankl1987erdos,simonovits1968method}).
\item $0.01$\textbf{-type stability results.} What is the structure of
an $\mathcal{H}$-free family whose size is within a constant of $\mathrm{ex}\left(n,\mathcal{H}\right)$?
(See \cite{dinur2009intersecting,friedgut2017kneser}).
\item \textbf{The removal problem. }Suppose that $\f$ is a family that
is almost $\h$-free, in the sense that it contains few copies of
$\h$. Is it true that $\f$ is close to an $\h$-free family? (See
e.g. \cite{gowers2007hypergraph,nagle2006counting,rodl2004regularity,ruzsa1978triple}).
\item \textbf{The counting problem. }How many $\h$-free families are there?
Particularly, is it true that almost all of them are contained in
an extremal $\h$-free family? (See e.g. \cite{samotij2014stability,saxton2015hypergraph}.
\item \textbf{The random problem. }Let $p\in\left(0,1\right)$, and let
$\binom{\left[n\right]}{k}_{p}$ be the random family that contains
each set in $\binom{\left[n\right]}{k}$ independently with probability
$p$. What is the size of the largest $\h$-free subfamily of $\binom{\left[n\right]}{k}_{p}$?
(See e.g. \cite{conlon2016combinatorial,schacht2016extremal}).
\end{enumerate}
In this paper our main focus will be on solving the 0.01-type stability
problem and the removal problem for a large class of hypergraphs called
\emph{expanded hypergraphs}. These are the hypergraphs in which both
the number of edges and the intersections of pairs of the edges are
bounded by a constant. While we shall not address the counting problem
and the random problem in this paper, we would like to note that the
container method of Balogh, Morris, and Samotij \cite{balogh2015independent},
and independently of Saxton and Thomason \cite{saxton2015hypergraph},
essentially reduces the solutions of the counting problem and the
random problem to the solutions of the 0.99-type stability problem
and the removal problem. Therefore, our work should be viewed as progress
towards all of the above problems.

Our results are based on a novel theorem about the sharp threshold
of `almost monotone' Boolean functions in the discrete cube, whose
proof uses the invariance principle of Mossel, O'Donnell, and Oleszkiewicz
\cite{mossel2010noise}. We believe that the connection we establish
between sharp threshold phenomena of Boolean functions and the removal
problem is the main contribution of this paper. 

\subsection{The structure of large families that are free from an expanded hypergraphs}
\begin{defn}
A hypergraph is said to be \emph{$\left(h,d\right)$-expanded} if
it has at most $h$ edges and the intersection of each two of its
edges is of size at most $d$. 
\end{defn}

The hypergraph $\mathcal{M}_{2}$ is an example of a $\left(2,0\right)$-expanded
hypergraph, and hypergraphs that are $\left(h,1\right)$-expanded
for some $h$ are known as \emph{linear hypergraph}s. Generally speaking,
we shall be concerned with $k$-uniform $\left(h,d\right)$-expanded
hypergraphs, where $h$ and $d$ are fixed, and where $k$ is significantly
larger. 

Our terminology stems from the following standard definition, (see
Mubayi and Verstra\"{e}te \cite{mubayi2016survey}).
\begin{defn}
Let $\mathcal{H}$ be a hypergraph. The $k$-expansion of $\mathcal{H}$
is the $k$-uniform hypergraph $\mathcal{H}^{+}$ obtained from $\mathcal{H}$
by enlarging each of its edges with distinct new vertices. We denote
by $\mathrm{ex}_{k}\left(n,\mathcal{H}^{+}\right)$ the problem of
determining the largest size of a family $\mathcal{F}\subseteq\binom{\left[n\right]}{k}$
free of the $k$-expansion of $\mathcal{H}.$ 
\end{defn}

Note that the $k$-expansion of a $d$-uniform hypergraph with $h$
edges is $\left(h,d\right)$-expanded. Conversely, any $\left(h,d\right)$-expanded
hypergraph can be easily seen to be the $k$-expansion of some $d\left(h-1\right)$-uniform
hypergraph.

Many problems in extremal combinatorics can be expressed as determining
$\mathrm{ex}_{k}\left(n,\mathcal{H}^{+}\right)$ for a fixed hypergraph
$\mathcal{H}$ (see e.g. \cite{frankl1987exact,furedi2014exact,kostochka2015turan,frankl2013improved},
and the survey of Mubayi and Verstra\"{e}te \cite{mubayi2016survey}
for the case where $\mathcal{H}$ is a graph). The methods used for
attacking such problems are varied. One of the most successful methods
is the delta-system method of Erd\H{o}s, Deza, and Frankl \cite{deza1978intersection}.
This method was applied by Frankl and F\"{u}redi \cite{frankl1977families,frankl1985forbidding,frankl1987exact}
to solve various Tur\'{a}n problems for expanded hypergraphs (including
the case where $\mathcal{H}$ is a special simplex, a sunflower, or
the hypergraph that consist of two edges with some intersection of
a fixed size). This allowed them to make significant progress on several
longstanding open problems in extremal combinatorics. 

Another notable technique is the shifting technique of Erd\H{o}s,
Ko, and Rado \cite{erdos1961intersection}. This technique was applied,
e.g., in a recent breakthrough of Frankl \cite{frankl2013improved}.
He gave the best bound for the Erd\H{o}s Matching Conjecture \cite{erdHos1965problem},
which asks to determine $\mathrm{ex}_{k}\left(n,M_{s}^{+}\right),$
where $M_{s}\subseteq\binom{\left[n\right]}{2}$ is a matching of
size $s$. Other methods include the Erd\H{o}s-Simonovits stability
method \cite{simonovits1968method}, and the random sampling from
the shadow method of Kostochka, Mubayi, and Verstra\"{e}te (see \cite{kostochka2015turan,Kostochka2015,kostochka2017turan}). 

Recently, a new approach towards the Tur\'{a}n problem for expansion
was initiated by Keller and the author \cite{keller2017junta} and
further developed by Ellis, Keller, and the author \cite{ellis2016stabilityfor}. 
\begin{defn}
A family $\mathcal{F}\subseteq\binom{\left[n\right]}{k}$ is said
to depend on the set of coordinates $J$ if for each sets $A,B\in\binom{\left[n\right]}{k}$
that satisfy $A\cap J=B\cap J$ we have $A\in\mathcal{F}\iff B\in\mathcal{F}.$
A family $\mathcal{F}$ is said to be a $j$-junta if it depends on
a set $J$ of size at most $j$. We say that a family $\mathcal{F}_{1}$
is \emph{$\epsilon$-essentially contained} in $\mathcal{F}_{2}$
if 
\[
\left|\mathcal{F}_{1}\backslash\mathcal{F}_{2}\right|\le\epsilon\binom{n}{k}.
\]
\end{defn}

The notion of a junta was introduced by Friedgut \cite{friedgut1998boolean}
while studying the isoperimetric problem in discrete cube. Dinur and
Friedgut \cite{dinur2009intersecting} were the first to use this
notion in the study of $k$-uniform set-systems. They showed the following.
\begin{thm}[Dinur--Friedgut \cite{dinur2009intersecting}]
\label{thm:Dinur-Friedgut} For each $r\in\mathbb{N},$ there exist
$C>0,j\in\mathbb{N},$ such that any intersecting family $\mathcal{F}\subseteq\binom{\left[n\right]}{k}$
is $C\left(\frac{k}{n}\right)^{r}$-essentially contained in an intersecting
$j$-junta. 
\end{thm}

Note that the theorem is trivial for $\frac{k}{n}=\Theta\left(1\right),$
while it is meaningful once $\frac{k}{n}$ is sufficiently small. 

Inspired by \cite{dinur2009intersecting}, Keller and the author \cite{keller2017junta}
extended Theorem \ref{thm:Dinur-Friedgut} to show that for each $h,r$
there exist $C>0,j\in\mathbb{N}$, such that any $M_{h}^{+}$-free
family $\mathcal{F}\subseteq\binom{\left[n\right]}{k}$ is $C\left(\frac{k}{n}\right)^{r}$-essentially
contained in an $M_{h}^{+}$-free $j$-junta, and obtained the following
result for general expanded hypergraphs.
\begin{thm}[\cite{keller2017junta}]
\label{thm:Keller Lifshitz} For each $\epsilon>0,h,d\in\mathbb{N},$
there exist $C>0,j\in\mathbb{N}$, such that the following holds.
Let $C<k<\frac{n}{C}$, and let $\mathcal{H}$ be a $k$-uniform $\left(h,d\right)$-expanded
hypergraph. Then any $\mathcal{H}$-free family $\mathcal{F}\subseteq\binom{\left[n\right]}{k}$
is $\epsilon$-essentially contained in an $\mathcal{H}$-free $j$-junta.
\end{thm}

Theorem \ref{thm:Keller Lifshitz} serves as the first step in the
following strategy for determining $\mathrm{ex}\left(n,\h\right).$ 
\begin{enumerate}
\item Show that any $\h$-free family is essentially contained in an $\h$-free
junta $\j$.
\item Find the extremal junta $\mathcal{J}_{\mathrm{ex}}$ that is free
of $\mathcal{H}.$ 
\item Show that if an $\mathcal{H}$-free junta has size that is close to
$\left|\mathcal{J}_{\mathrm{ex}}\right|$, then it must be a small
perturbation of $\mathcal{J}_{\mathrm{ex}}.$ 
\item Show that any $\mathcal{H}$-free small perturbation of $\mathcal{J}_{\mathrm{ex}}$
must have smaller size than it.
\end{enumerate}
These four steps together suffice in order to show that the extremal
junta is the family $\mathcal{J}_{\mathrm{ex}}.$ Indeed, if $\mathcal{F}$
is the extremal $\mathcal{H}$-free family, then Step 1 implies that
$\mathcal{F}$ is essentially contained in an $\mathcal{H}$-free
junta $\mathcal{J}$. The fact that $\mathcal{F}$ is of the extremal
size implies that the size of $\mathcal{J}$ cannot be much smaller
than the size of $\mathcal{J}_{\mathrm{ex}}.$ Step 3 implies that
$\mathcal{F}$ is essentially contained in $\mathcal{J}_{\mathrm{ex}},$
and Step 4 implies that $\mathcal{F}$ is actually equal to $\mathcal{J}_{\mathrm{ex}}.$ 

This strategy was successfully carried out in \cite{keller2017junta}
to solve the Tur\'{a}n problem for various $\left(h,d\right)$-expanded
hypergraphs, in the regime where $C<k<\frac{n}{C}$ for some $C=C\left(h,d\right).$

Later, \cite{ellis2016stabilityfor} showed that this strategy can
be carried out also for some hypergraphs in the regime where $\epsilon n<k<\left(\frac{1}{2}-\epsilon\right)n$
for an arbitrarily small constant $\epsilon,$ and a sufficiently
large $n$. Specifically, they considered the case where the forbidden
hypergraph $\h$ is $\mathcal{I}_{2,d}$ that consists of two edges
that intersect in $d$ elements. 

Their basic observation was that any junta that does not contain a
copy of $\mathcal{I}_{2,d}$ must be free of $\mathcal{I}_{2,d'}$
for any $d'<d$ as well. In other words, any two sets in an $\mathcal{I}_{2,d}$-free
junta have intersection of size at least $d+1$. This essentially
reduces the problem to the well known problem on the size of $\left(d+1\right)$-intersecting families, which was solved decades ago using the shifting
technique (see Ahlswede--Khachatrian \cite{ahlswede1996complete},
Filmus \cite{filmus2016ahlswede}, and Frankl \cite{frankl18erdos}).

It is our belief that this strategy may be carried out for various
other $\left(h,d\right)$-expanded hypergraphs, and that the following
result we prove in this paper will serve as the first step in the
solution of the Tur\'{a}n problem for various other hypergraphs in
the regime where $\epsilon n<k\le\left(\frac{1}{h}-\epsilon\right)n.$ 
\begin{thm}
\label{thm:general approximation by junta theorem} For each $\epsilon>0,d,h\in\mathbb{N},$
there exists $j>0,$ such that the following holds. Let $\epsilon n\le k\le\left(\frac{1}{h}-\epsilon\right)n$,
and let $\mathcal{H}$ be an $\left(h,d\right)$-expanded hypergraph.
Then any $\mathcal{H}$-free $\mathcal{F\subseteq}\binom{\left[n\right]}{k}$
is $\epsilon$-essentially contained in an $\mathcal{H}$-free $j$-junta.
\end{thm}

The special case of Theorem \ref{thm:general approximation by junta theorem}
where $\mathcal{H}=M_{2}^{+}=\mathcal{M}_{2}$ was already proved
recently by Friedgut and Regev \cite{friedgut2017kneser} who built
upon the work of Dinur and Friedgut \cite{dinur2009intersecting}.
Other special cases of Theorem \ref{thm:general approximation by junta theorem}
were proved in \cite{ellis2016stabilityfor}, which settles the case
$h=2$ of the theorem. 

Theorem \ref{thm:general approximation by junta theorem} is actually
a special case of our main Theorem \ref{thm:Genneral removal Lemma}
below, which deals also with families that contain few copies of $\mathcal{H}$,
rather than dealing only with $\h$-free families. 

Similarly to the case where $\h=\mathcal{I}_{2,d}$, it turns out
that it is a general phenomenon that $\h$-free juntas are automatically
free of some other hypergraphs. 
\begin{defn}
Let $\mathcal{H}$ be a hypergraph and let $v$ be a vertex of $\mathcal{H}$.
The \emph{resolution of $\mathcal{H}$ at $v$, }denoted by $\mathrm{res}\left(\mathcal{H},v\right)$,
is the hypergraph obtained from $\mathcal{H}$ by taking $v$ out
of each edge of $\mathcal{H}$ that contains $v$, and by replacing
it with a new vertex that belongs only to this edge. The resolution
of $\mathcal{H}$ at a set of vertices $S$, denoted by $\mathrm{res}\left(\mathcal{H},S\right)$,
is the hypergraph obtained by resolving $\mathcal{H}$ at the vertices
of $S$ one after the other. Any hypergraph of the form $\mathrm{res}\left(\mathcal{H},S\right)$
will be called a \emph{resolution }of\emph{ $\mathcal{H}$}.
\end{defn}

\begin{example}
Any hypergraph $\h$ is a resolution of itself since $\mathrm{res}\left(\h,\varnothing\right)=\h$.
Defining the \emph{center} of a hypergraph $\h$ to be the set of
its vertices that belong to at least two of its edges, the $k$-uniform
$h$-matching $\mathcal{M}_{h}:=M_{h}^{+}$ is the resolution of any
$k$-uniform hypergraph with $h$ edges at its center. Another simple
example is the hypergraph $\mathcal{I}_{2,d}$: its resolutions are
the hypergraphs of the form $\mathcal{I}_{2,d'}$ for  $d'\le d.$ 
\end{example}

It is easy to show that any $j$-junta $\mathcal{G}\subseteq\binom{\left[n\right]}{k}$
that is free of a hypergraph $\mathcal{H}$ with $h$ edges is also
free of every resolution of $\mathcal{H}$, provided that $C<k\le\left(\frac{1}{h}-\epsilon\right)n$
and $n$ is large enough. Hence, in order to show that a given junta
$\j$ is an extremal $\h$-free family, it would essentially be enough
to show that it is the extremal family that is free of a copy of $\h$
as well as of all of its resolutions. 

\subsection{Removal lemma for expanded hypergraphs}

While Theorem \ref{thm:general approximation by junta theorem} tells
us the structure of $\mathcal{H}$-free families it tells us nothing
on families that are `almost $\mathcal{H}$-free', a notion that may
be defined more precisely as follows. 
\begin{defn}
Let $\delta>0$ and let $\mathcal{H}$ be a $k$-uniform hypergraph.
We say that a family $\mathcal{F}\subseteq\binom{\left[n\right]}{k}$
is \emph{$\delta$-almost $\mathcal{H}$-free} if a random copy of
$\mathcal{H}$ in $\binom{\left[n\right]}{k}$ lies within $\mathcal{F}$
with probability at most $\delta.$ 
\end{defn}

The celebrated triangle removal lemma says that for any $\epsilon>0$,
there exists $\delta>0$, such that any $\delta$-almost triangle-free
graph is $\epsilon$-essentially contained in a triangle-free graph.
This was generalized by Gowers \cite{gowers2006quasirandomness,gowers2007hypergraph},
and independently by Nagle, R\"{o}dl, Schacht, and Skokan \cite{nagle2006counting,rodl2004regularity}
to show that for each fixed $k$-uniform hypergraph $\mathcal{H}$,
there exists $\epsilon>0$, such that if a family $\mathcal{F}\subseteq\binom{\left[n\right]}{k}$
is $\delta$-almost $\mathcal{H}$-free, then $\mathcal{F}$ is $\epsilon$-essentially
contained in an $\mathcal{H}$-free family. This result is known as
the \emph{hypergraph removal lemma. }(See the survey of Conlon and
Fox \cite{conlon2013graph} for a more thorough history, and for quantitative
aspects of removal lemmas.) 

While the hypergraph removal lemma settles the case where $k,$ $\mathcal{H}$,
and $\epsilon$ are fixed, it becomes quite useless for $k$ that
tends to infinity with $n$. Indeed, the initial dependence of $\delta$
on $\epsilon$ in the graph case where $k=2$ was 
\[
\delta=\left(\mathrm{tower}\left(\epsilon^{-O_{\h}\left(1\right)}\right)\right)^{-1}=\underset{\epsilon^{-O_{\mathcal{H}}\left(1\right)}\text{ times}}{\underbrace{2^{2^{\iddots^{2}}}}},
\]
 and this was improved by Fox \cite{fox2011new} to $\mathrm{tower}\left(O_{\h}\left(\log\frac{1}{\epsilon}\right)\right)^{-1}.$
For $k=3$ the best known bound is $\delta=\underset{\epsilon^{-O_{\h}\left(1\right)}\text{ times}}{\underbrace{\mathrm{tower}\left(\mathrm{tower}\cdots\left(2\right)\right)^{-1}}}$,
and the bounds similarly worsen as $k$ increases (see \cite[Remark 2.11]{tao2006variant}). 

Friedgut and Regev \cite{friedgut2017kneser} were the first to prove
a removal lemma in the case where $k$ is linear in $n$. They showed
that for each $\epsilon>0$ there exists $\delta>0$, such that if
$\epsilon n\le k\le\left(\frac{1}{2}-\epsilon\right)n,$ and if $\mathcal{F}\subseteq\binom{\left[n\right]}{k}$
is a $\delta$-almost $\mathcal{M}_{2}$-free family, then $\mathcal{F}$
is $\epsilon$-essentially cxontained in an $\mathcal{M}_{2}$-free
family. Later, Das and Tran \cite{Das2016Removal} proved a quantitatively
stronger removal result for $\delta$-almost $M_{2}^{+}$-free families
whose size is close to $\binom{n-1}{k-1}.$ 

At first glance it may seem that the Friedgut--Regev Theorem follows
from the hypergraph removal lemma, but it actually does not. While
the hypergraph removal lemma deals with the case where $k$ and the
hypergraph $\mathcal{H}$ are fixed, the Friedgut--Regev theorem
deals with the case where $k$ is linear in $n$. Our goal in this
paper it to prove removal lemmas for other expanded hypergraphs in
the regime where $k$ is up to linear in $n$. 

In the light of Theorem \ref{thm:general approximation by junta theorem},
it may seem as though the Friedgut--Regev Theorem can be generalized to all
$\left(h,d\right)$-expanded hypergraphs. However, we show that the
following surprising statement holds.
\begin{thm}
\label{thm:Removal for matchings} For each $h,d\in\mathbb{N},\epsilon>0$
there exist $C,\delta>0$ such that if $C\le k\le\left(\frac{1}{h}-\epsilon\right)n$,
and $\mathcal{H}$ is a $k$-uniform $\left(h,d\right)$-expanded
hypergraph, then the following statements hold.
\begin{enumerate}
\item If the family $\mathcal{F}$ is $\delta$-almost $\mathcal{H}$-free,
then $\f$ is $\epsilon$-essentially contained in an $\mathcal{M}_{h}$-free
family. 
\item Conversely, if the family $\mathcal{F}$ is $\delta$-essentially
contained in an $\mathcal{M}_{h}$-free family, then $\f$ is $\epsilon$-almost
$\h$-free. 
\end{enumerate}
\end{thm}

So suppose that $\mathcal{F}$ is a family and we want to check whether
it is $\mathcal{H}$-free or not. One natural way to check if $\mathcal{F}$
is $\mathcal{H}$-free is to choose uniformly at random copies of
$\mathcal{H},$ and to check that none of them are contained in $\mathcal{F}$.
While we could expect that this would tell us that $\mathcal{F}$
is close to some $\mathcal{H}$-free family, we instead obtain from
Theorem \ref{thm:Removal for matchings} that this implies that $\mathcal{F}$
is close to a family that is free of the hypergraph $\mathcal{M}_{h}.$
Even more surprisingly, the converse also holds. Any family that is
close to an $\mathcal{M}_{h}$-free family contains few copies of
$\h.$ This phenomenon becomes clearer by inspecting the following
example. 
\begin{example}
\label{exa:Star is almost free of i21}The \emph{star} $\left\{ A\in\binom{\left[n\right]}{n/3}:\,1\in A\right\} $
is $o\left(1\right)$-almost free of the hypergraph $\mathcal{I}_{2,1}$,
which consists of two edges that intersect in a singleton $\left\{ i\right\} $.
Indeed, the probability that a random copy of this hypergraph lies
in the star is $\frac{1}{n}$, as it is the probability that a random
injection from the vertices of $\mathcal{I}_{2,1}$ to $\left[n\right]$
sends the vertex $i$ to $1$. As Theorem \ref{thm:Removal for matchings}
guarantees, the star is $o\left(1\right)$-essentially contained in
an $\mathcal{M}_{2}$-free family as it is in itself $\mathcal{M}_{2}$-free.
However, the star is not $o\left(1\right)$-essentially contained
in any family free of the hypergraph $\mathcal{I}_{2,1}$. 

More generally, suppose that $\mathcal{\g}$ is a $j$-junta depending
on a set $J$ and that $\h$ is an $\left(h,d\right)$-expanded hypergraph.
Then the center of a random copy of $\h$ most likely does not intersect
$J$. So from the `point of view' of the junta $\g$, a random copy
of $\h$ and a random copy of $\mathcal{M}_{h}$ look the same. It
is therefore easy to see that 
\[
\Pr\left[\text{a random copy of \ensuremath{\h}}\text{ lies in }\g\right]=\Pr\left[\text{a random copy of \ensuremath{\mathcal{M}_{h}}}\text{ lies in }\g\right]+o\left(1\right).
\]
\end{example}

Let $\mathcal{F}\subseteq\binom{\left[n\right]}{k}$, and let $\mathcal{H}$
be a hypergraph. We say that $\mathcal{F}$ is \emph{$\left(\mathcal{H},s\right)$-free}
if it is free of any resolution of $\mathcal{H}$ whose center is
of size at most $s.$ While Example \ref{exa:Star is almost free of i21}
shows that being $o\left(1\right)$-almost free of $\h$ is not sufficient
for guaranteeing closeness to an $\h$-free family, the following
theorem shows that a stronger assumption is sufficient.
\begin{thm}
\label{thm:Genneral removal Lemma} For each $h,d,s\in\mathbb{N},\epsilon>0$
there exist $\delta>0,j\in\mathbb{N}$, such that the following holds.
Let $\h$ be an $\left(h,d\right)$-expanded hypergraph. Let $\frac{1}{\delta}\le k\le\left(\frac{1}{h}-\epsilon\right)n,$
and let $\mathcal{F}$ be a $\frac{\delta}{n^{s}}$-almost $\mathcal{H}$-free
family. Then $\mathcal{F}$ is $\epsilon$-essentially contained in
an $\left(\h,s\right)$-free $j$-junta.
\end{thm}

Note that Theorem \ref{thm:general approximation by junta theorem}
is a special case of Theorem \ref{thm:Genneral removal Lemma}. Indeed,
Theorem \ref{thm:Genneral removal Lemma} implies that if $\h$ is
a hypergraph whose center is of size $c$, then any $\frac{\delta}{n^{c}}$-almost
$\h$-free family is $\epsilon$-essentially contained in an $\left(\h,c\right)$-free
family, i.e. to a family free of $\h$ and of any resolution of it.
On the other hand, Theorem \ref{thm:general approximation by junta theorem}
yields the same conclusion under the stronger hypothesis that $\f$
is $\h$-free. 

The following proposition is a converse to Theorem \ref{thm:Genneral removal Lemma}.
It shows that any $\left(\mathcal{H},s\right)$-free $j$-juntas is
$O\left(\frac{1}{n^{s+1}}\right)$-almost $\mathcal{H}$-free. So
in particular, $j$-juntas are $\frac{\delta}{n^{s}}$-almost $\mathcal{H}$-free,
provided that $n$ is sufficiently large as a function of $\delta$.
\begin{prop}
\label{Prop:converse to removal lemma}For each $h,c,j,s\in\mathbb{N},$
there exists a constant $C>0$, such that the following holds. Let
$\h$ be a hypergraph with $h$ edges whose center is of size $c$.
Let $C\le k\le\left(\frac{1}{h}-\epsilon\right)n,$ and let $\j$
be an $\left(\mathcal{H},s\right)$-free $j$-junta. Then $\mathcal{J}$
is $\frac{C}{n^{s+1}}$-almost $\mathcal{H}$-free.
\end{prop}

\subsection{Sketch of Proof of Theorem \ref{thm:Genneral removal Lemma} for
matching }

We shall now sketch the proof of Theorem \ref{thm:Genneral removal Lemma}
in the case where the forbidden hypergraph is $\mathcal{M}_{h}.$
The proof relies on the regularity method and on a novel sharp threshold
result for `almost monotone' Boolean functions that will be presented
in Section \ref{sec:Sharp-threshold-theorems}. 

Let $\epsilon>0,h\in\mathbb{N}$ be fixed constants, and let $\f$
be a family which is not $\epsilon$-contained in any family free
of the matching $\mathcal{M}_{h}$. Our goal is to show that a random
matching lies in $\f$ with probability $\Theta\left(1\right)$.

Note that any set $J$ decomposes the sets in $\f$ into $2^{\left|J\right|}$
parts according to their intersection with $J$. Following Friedgut
and Regev \cite{friedgut2017kneser} and \cite{ellis2016stabilityfor}
we apply a regularity lemma which says that we may find a set $J$,
such that in the decomposition of $\f$ induced by $J$, almost all
of the parts are either `random-like' or sufficiently small that we
can ignore them. We may then take as our approximating junta, the
family 
\[
\g=\left\{ A\in\binom{\left[n\right]}{k}:\,A\cap J\text{ corresponds to a random part of }\f\right\} .
\]
The fact that $\f$ is not $\epsilon$-essentially contained in an
$\mathcal{M}_{h}$-free family will allow us to show that $\g$ is
not $\mathcal{M}_{h}$-free. This in turn will imply that there exist
pairwise disjoints sets $A_{1},\ldots,A_{h}\subseteq J$ that correspond
to random parts of $\f.$ Now note that a random matching $\left\{ \boldsymbol{B}_{1},\ldots,\boldsymbol{B}_{h}\right\} $
intersects the set $J$ in the sets $A_{1},\ldots,A_{h}$ with probability
$\Theta\left(1\right)$. So the remaining task is to show that if
$\f_{1},\ldots,\f_{h}$ are `random-like' parts, then a random matching
$\boldsymbol{A}_{1},\ldots,\boldsymbol{A}_{h}$ satisfies $\boldsymbol{A}_{i}\in\f_{i}$
with probability $\Theta\left(1\right)$. We will accomplish this
task using an enhancement of the `sharp threshold technology' presented
by Dinur and Friedgut \cite{dinur2009intersecting}. Let us recall
first the method in \cite{dinur2009intersecting}. We say that families
$\f_{1},\ldots,\f_{h}$ are cross free of a matching if there exist
no pairwise disjoint sets $A_{1},\ldots,A_{h}$ such that $A_{i}\in\f_{i}$
for each $i$, otherwise they cross-contain a matching. 

The \emph{$p$-biased distribution} on $\p\left(\left[n\right]\right)$
is a probability distribution on sets $\boldsymbol{A}\subseteq\left[n\right]$,
where each element is chosen to be in $\boldsymbol{A}$ independently
with probability $p.$ For a family $\g$, write $\mu_{p}\left(\g\right)$
for $\Pr_{\boldsymbol{A}\sim\mu_{p}}\left[\boldsymbol{A}\in\g\right].$
A family $\f\subseteq\p\left(\left[n\right]\right)$ is monotone if
$B\in\f$ whenever $B\supseteq A$ for some $A\in\f.$ The `sharp
threshold principle' essentially says that for a random-like monotone
family $\f$ the $p$-biased measure of $\f$ jumps from being near
$0$ to being near $1$ in a short interval.

Roughly speaking, the analogue of the strategy in \cite{ellis2016stabilityfor}
for the hypergraph $\mathcal{M}_{h}$ goes as follows.
\begin{enumerate}
\item Observe that if $\left\{ \f_{i}\right\} _{i=1}^{h}$ are cross-free
of a matching, then their up-closures 
\[
\left\{ \f_{i}^{\uparrow}:=\left\{ B:\,\exists A\subseteq B\text{ such that \ensuremath{A\in\f}}\right\} \right\} _{i=1}^{h}
\]
 are also cross-free of a matching (in the sense that there are no
pairwise disjoint sets $A_{1},\ldots,A_{h}$ with $A_{i}\in\f_{i}^{\uparrow}$).
\item Use a simple coupling argument to show that if $\mu_{\frac{1}{h}}\left(\f_{i}^{\uparrow}\right)>1-\frac{1}{h}$
for each $i$, then the families $\f_{1}^{\uparrow},\ldots,\f_{h}^{\uparrow}$
cross-contain a matching. So in particular, the families $\f_{1},\ldots,\f_{h}$
cross contain a matching. 
\item Show that the families $\f_{i}^{\uparrow}$ are random-like monotone
families whose $\mu_{\frac{k}{n}}$ measure is bounded away from 0.
The sharp threshold principle will allow us to deduce that $\text{\ensuremath{\mu}}_{\frac{1}{h}}\left(\f_{i}^{\uparrow}\right)$
is close to 1, so by Step 2 the families $\left\{ \f_{i}\right\} _{i=1}^{h}$
cannot be cross free of a matching. 
\end{enumerate}
This plan fails completely when we try to show the desired statement
that random-like families contain \emph{many} matchings. The step
which stops working is the first one. While it is true that if $\left\{ \f_{i}\right\} $
are cross-free of a matching, then their up closures $\left\{ \f_{i}^{\uparrow}\right\} $
are cross-free of a matching, it is not true that if $\left\{ \f_{i}\right\} $-are
almost cross-free of a matching (in the sense that they cross-contain
few matchings), then the families $\f_{i}^{\uparrow}$ are also almost
cross free of a matching. We resolve this issue by replacing the up-closure
of $\f$ by the family 
\[
\left\{ A\in\p\left(\left[n\right]\right):\,\left|A\right|\ge k\text{ and a random }k\text{-}\text{subset of }A\text{ lies in }\f\text{ with probability }\Theta\left(1\right)\right\} .
\]
However, this new family is not monotone, and instead it satisfies
a weaker hypothesis that may be called `almost monotonicity'. 

So to make the above plan work, we shall need to generalize the sharp
threshold principle from monotone families to `almost monotone' families.
This statement is made more precise in Section \ref{sec:Sharp-threshold-theorems}.
It is accomplished with the help of the invariance principle of Mossel,
O'Donnell, and Oleszkiewicz \cite{mossel2010noise}. 

In our view, the main contribution of this paper comes from the fact
that we relate sharp threshold results to hypergraph removal
problems. We believe that further exploration of the relation between
these two well studied problems will improve the understanding of
each of them.  

In the following section we give a more thorough introduction of the
sharp threshold principle of monotone Boolean functions, and state
our sharp threshold result for almost monotone Boolean functions.
(Note that Boolean functions $f\colon\left\{ 0,1\right\} ^{n}\to\left\{ 0,1\right\} $,
and families $\f\subseteq\p\left(\left[n\right]\right)$ can be identified).

\section{\label{sec:Sharp-threshold-theorems}Sharp threshold theorems for
almost monotone functions}

We use bold letters to denote random variables, and we write $\left[n\right]$
for the set $\left\{ 1,\ldots,n\right\} .$ We shall use the convention
that the $i$th coordinate of an $x\in\left\{ 0,1\right\} ^{n}$ is
denoted by $x_{i}$.\textbf{ }A function $f\colon\left\{ 0,1\right\} ^{n}\to\mathbb{R}$
is said to be \emph{monotone} if $f\left(x\right)\le f\left(y\right)$
whenever $x,y$ are elements of $\left\{ 0,1\right\} ^{n}$ that satisfy
$\forall i\in\left[n\right]:\,x_{i}\le y_{i}$. The $p$-biased distribution
$\mu_{p}$ is the distribution on the set $\left\{ 0,1\right\} ^{n},$
where a random element $\boldsymbol{x}\sim\mu_{p}$ is chosen by letting
its coordinates $\mathbf{x}_{i}$ to be independent random variables
that take the value $1$ with probability $p$. For a function $f\colon\left\{ 0,1\right\} ^{n}\to\mathbb{R}$,
we write $\mu_{p}\left(f\right)$ for $\mathbb{E}_{\mathbf{x}\sim\mu_{p}}\left[f\left(\mathbf{x}\right)\right].$

It is easy to see that for any monotone function $f\colon\left\{ 0,1\right\} ^{n}\to\left\{ 0,1\right\} $,
the function $p\mapsto\mu_{p}\left(f\right)$ is a monotone increasing
function of $p$. Roughly speaking, a Boolean function $f\colon\left\{ 0,1\right\} ^{n}\to\left\{ 0,1\right\} $
is said to have a \emph{sharp threshold} if there exists a `short'
interval $\left[q,p\right]$, such that $\mu_{q}\left(f\right)$ is
`close' to 0, and $\mu_{p}\left(f\right)$ is `close' to 1. Otherwise,
it is said to have a \emph{coarse threshold}. 

A central problem in the area of analysis of Boolean functions is
the following (see e.g. \cite{BourgainKalai19,friedgut1999sharp,friedgut1996every,hatami2012structure}). 
\begin{problem}
\label{Problem: Characterization Boolean functions that exhibit a coarse threshold}
Which monotone Boolean functions $f\colon\left\{ 0,1\right\} ^{n}\to\left\{ 0,1\right\} $
exhibit a coarse threshold?
\end{problem}

We shall now make the above discussion more formal. For a non-constant
monotone $f\colon\left\{ 0,1\right\} ^{n}\to\left\{ 0,1\right\} $,
the \emph{critical probability} of $f$ (denoted by $p_{c}\left(f\right)$)
is the unique number in the interval $\left(0,1\right)$\emph{, such
that $\mu_{p_{c}}\left(f\right)=\frac{1}{2}.$ }Bollob\'{a}s and
Thomason \cite{bollobas1987threshold} showed that for any fixed $\epsilon>0,$
and each Boolean function $f$ there exists an interval $\left[q,p\right]$
with $q,p=\Theta\left(p_{c}\left(f\right)\right),$ such that $\mu_{q}\left(f\right)<\epsilon,$
and $\mu_{p}\left(f\right)>1-\epsilon$. Therefore, $f$ should be
considered to have a sharp threshold if there exists an interval $\left[q,p\right]$
of length significantly smaller than $p_{c}\left(f\right)$, such
that $\mu_{q}\left(f\right)<\epsilon$ and $\mu_{p}\left(f\right)>1-\epsilon$.

Formally, we say that a Boolean function $f\colon\left\{ 0,1\right\} ^{n}\to\left\{ 0,1\right\} $
exhibits an $\epsilon$\emph{-sharp threshold} if there exists an
interval $\left[q,p\right]$ of length $\epsilon p_{c}\left(f_{n}\right)$,
such that $\mu_{q}\left(f\right)<\epsilon$ and $\mu_{p}\left(f\right)>1-\epsilon$.
We say that $f$ exhibits an $\epsilon$\emph{-coarse threshold} if
there exist an $\epsilon>0,$ and an interval $\left[q,p\right]$
of length at least $\epsilon p_{c}\left(f\right),$ such that $\mu_{q}\left(f_{n}\right)>\epsilon,$
and $\mu_{p}\left(f_{n}\right)<1-\epsilon.$ 

A function $f\colon\left\{ 0,1\right\} ^{n}\to\left\{ 0,1\right\} $
is said to be \emph{transitive symmetric }if the group of all permutations
$\sigma\in S_{n},$ such that 
\[
\forall x\in\left\{ 0,1\right\} ^{n}:\,f\left(x_{\sigma\left(1\right)},\ldots,x_{\sigma\left(n\right)}\right)\equiv f\left(x_{1},\ldots,x_{n}\right)
\]
 acts transitively on $\left\{ 1,\ldots,n\right\} .$ 

The Friedgut--Kalai Theorem \cite{friedgut1996every} says that if
$f$ is transitive symmetric and $p_{c}\left(f\right)$ is bounded
away from 0 and 1, then $f$ exhibits a sharp threshold.
\begin{thm}[Friedgut--Kalai]
\label{thm:Friedgut-Kalai} For each $\epsilon>0$ there exists $n_{0}=n_{0}\left(\epsilon\right)$
such that the following holds. Let $n>n_{0}$ and let $f\colon\left\{ 0,1\right\} ^{n}\to\left\{ 0,1\right\} $
be a monotone transitive symmetric function satisfying $\epsilon<p_{c}\left(f\right)<1-\epsilon$.
Then $f$ exhibits an $\epsilon$-sharp threshold. 
\end{thm}

On the other hand, $f$ need not exhibit a coarse threshold if $f$
is no longer assumed to be transitive symmetric. Let $j$ be a constant.
A function $f$ is said to be a \emph{$j$-junta} if it depends on
at most $j$ coordinates. It is easy to see that any non-constant
monotone $j$-junta exhibits an $\epsilon$-coarse threshold for some
constant $\epsilon=\epsilon\left(j\right)>0$. A well known corollary
of the celebrated Friedgut's Junta Theorem \cite{friedgut1998boolean}
is a partial converse to this statement. We shall say that $f$ is
\emph{$\left(\mu_{r},\epsilon\right)$-close} to $g$ if 
\[
\Pr_{\mathbf{x}\sim\mu_{r}}\left[f\left(\mathbf{x}\right)\ne g\left(\mathbf{x}\right)\right]<\epsilon.
\]

\begin{thm}[Corollary of Friedgut's Junta Theorem]
\label{thm:Friedgut's junta theorem} For each $\epsilon>0$, there
exists $j\in\mathbb{N}$, such that the following holds. Let $f\colon\left\{ 0,1\right\} ^{n}\to\left\{ 0,1\right\} $
be a Boolean function, and let $q,p$ be numbers in the interval $\left(0,1\right)$
that satisfy $p>q+\epsilon.$ Then there exists some $r$ in the interval
$\left[q,p\right],$ such that $f$ is $\left(\mu_{r},\epsilon\right)$-close
to a $j$-junta.
\end{thm}

Note that Friedgut's Junta Theorem becomes trivial if $\mu_{q}\left(f\right)<\epsilon$
or if $\mu_{p}\left(f\right)>1-\epsilon$ as in which case we may
take the junta to be the corresponding constant function. For that
reason, Friedgut's Junta Theorem can be interpreted by saying that
non-junta-like functions exhibit a sharp threshold behavior. 

\subsection{Structural results on monotone families}

We extend Theorems \ref{thm:Friedgut-Kalai} and \ref{thm:Friedgut's junta theorem}
in the following directions.
\begin{itemize}
\item We replace the condition that $f$ is monotone with the weaker condition
that $f$ satisfies a notion we call \emph{$\left(q,p,\delta\right)$-almost
monotonicity.}
\item We strengthen the Friedgut--Kalai theorem by relaxing the condition
that $f$ is transitive symmetric to the weaker condition that $f$
satisfies a notion called \emph{$\left(r,\delta,\mu_{q}\right)$-regularity}.
\item Bearing in mind our applications to the removal problem, we modify
Theorem \ref{thm:Friedgut's junta theorem} by replacing the condition
that $f$ is `close' to a junta with respect to the $\mu_{r}$ measure
with a condition that says that $f$ is `close' to a junta in a sense
that involves only the measures $\mu_{p}$ and $\mu_{q}$, i.e. the
measures at the ends of the interval.
\end{itemize}
We shall now define the above notions more precisely, starting with
$\left(q,p,\delta\right)$-almost monotonicity. Intuitively, a function
$f$ should be called `almost monotone' if $f\left(x\right)\le f\left(y\right)$
for almost all values of $x$ and $y$ that satisfy $\forall i\in\left[n\right]:\,x_{i}\le y_{i}.$
However, there are many ways to interpret the notion `almost all values
of $x$ and $y$'. For instance, the following definitions all seem
to fit equally well.
\begin{itemize}
\item Choose $\mathbf{x}$ uniformly out of $\left\{ 0,1\right\} ^{n}$
and then choose $\mathbf{y}$ uniformly among the set of all the elements
$y\in\left\{ 0,1\right\} ^{n}$ that satisfy $\forall i:\,y_{i}\ge\mathbf{x}_{i}$.
Say that $f$ is `almost monotone' if $\Pr\left[f\left(\mathbf{x}\right)>f\left(\mathbf{y}\right)\right]$
is `small'.
\item First choose $\mathbf{y}$ uniformly out of $\left\{ 0,1\right\} ^{n}$,
then choose $\mathbf{x}$ among the set of all $x\in\left\{ 0,1\right\} ^{n}$
that satisfy $\forall i:\,x_{i}\le\mathbf{y}_{i},$ and say that $f$
is `almost monotone' if $\Pr\left[f\left(\mathbf{x}\right)>f\left(\mathbf{y}\right)\right]$
is `small'. 
\item Choose a uniformly random pair of elements $\mathbf{x},\mathbf{y}\in\left\{ 0,1\right\} ^{n}$
among the $x,y\in\left\{ 0,1\right\} ^{n}$ that satisfy $\forall i:\,x_{i}\le y_{i},$
and say that $f$ is `almost monotone' if $\Pr\left[f\left(\mathbf{x}\right)>f\left(\mathbf{y}\right)\right]$
is `small'. 
\end{itemize}
Note that these notions are different. In the first we have 
\[
\Pr\left[\mathbf{x}_{i}=1\right]=\frac{1}{2}\text{ and }\Pr\left[\mathbf{y}_{i}=1\right]=\frac{3}{4},
\]
 in the second we have 
\[
\Pr\left[\mathbf{x}_{i}=1\right]=\frac{1}{4}\text{ and }\Pr\left[\mathbf{y}_{i}=1\right]=\frac{1}{2},
\]
 and in the last we have 
\[
\Pr\left[\mathbf{x}_{i}=1\right]=\frac{1}{3}\text{ and }\Pr\left[\mathbf{y}_{i}=1\right]=\frac{2}{3}.
\]
 All these notions are captured by the following framework.
\begin{defn}
Let $q<p.$ The $\left(q,p\right)$-biased distribution, denoted by
$D\left(q,p\right)$, is the unique probability distribution on elements
$\left(\mathbf{x},\mathbf{y}\right)\in\left\{ 0,1\right\} ^{n}\times\left\{ 0,1\right\} ^{n}$
that satisfies the following.

\begin{enumerate}
\item The pairs $\left(\mathbf{x}_{i},\mathbf{y}_{i}\right)$ are independent
random variables. 
\item We have $\mathbf{x}_{i}\le\mathbf{y}_{i}$ with probability 1.
\item We have $\Pr\left[\mathbf{x}_{i}=1\right]=q$ and $\Pr\left[\mathbf{y}_{i}=1\right]=p.$
\end{enumerate}
We write $\mathbf{x,y}\sim D\left(q,p\right)$ to denote that they
are chosen according to this distribution. We say that $f\colon\left\{ 0,1\right\} ^{n}\to\left\{ 0,1\right\} $
is $\left(q,p,\delta\right)$-almost monotone if 
\[
\Pr_{\mathbf{x,y}\sim D\left(q,p\right)}\left[f\left(\mathbf{x}\right)>f\left(\mathbf{y}\right)\right]<\delta.
\]
\end{defn}

We give the following variant of Friedgut's junta theorem. It implies
that if $\epsilon,q<p$ are fixed numbers in the interval $\left(0,1\right)$,
if $\delta>0$ is small enough, and if $j\in\mathbb{N}$ is sufficiently
large, then for any $\left(q,p,\delta\right)$-almost monotone function
$f$, there exists a monotone $j$-junta $g$, such that with high
probabiliity
$f(\mathbf{x})\le g(\mathbf{x})$ with respect to
$\mathbf{x}\sim \mu_q$, while $g(\mathbf{y})\le f(\mathbf{y})$ with
respect to $\mathbf{y}\sim \mu_p$.
\begin{thm}
\label{thm:Monotone approximation}For each $\epsilon>0$, there exists
$j\in\mathbb{N},\delta>0$, such that the following holds. Let $p,q$
be  numbers in the interval $\left(\epsilon,1-\epsilon\right)$ that
satisfy $p-q>\epsilon$ and let $f\colon\left\{ 0,1\right\} ^{n}\to\left\{ 0,1\right\} $
be a $\left(q,p,\delta\right)$-almost monotone function. Then there
exists a monotone $j$-junta $g$, such that 
\[
\Pr_{\x\sim\mu_{q}}\left[f\left(\x\right)>g\left(\x\right)\right]<\epsilon\text{ and }\Pr_{\x\sim\mu_{p}}\left[f\left(\x\right)<g\left(\x\right)\right]<\epsilon.
\]
\end{thm}

Note that Theorem \ref{thm:Monotone approximation} is really a theorem
about functions that have a coarse threshold. Indeed, if we have either
$\mu_{p}\left(f\right)>1-\epsilon$ or $\mu_{q}\left(f\right)<\epsilon$,
then the theorem becomes trivial by taking $g$ to be a suitable constant
function. 

The conclusion of Theorem \ref{thm:Monotone approximation} says that
$f$ can be `approximated' by the junta $g$, where our approximation
notion is the `two-sided' notion of closeness.
It is natural to ask whether $f$ can also be approximated by a junta
according to a `one-sided' notion, such as the notions of $\left(\mu_{p},\epsilon\right)$-closeness
and $\left(\mu_{q},\epsilon\right)$-closeness. The following example
demonstrates that the two-sided approximation is actually necessary.
\begin{example}
Fix some numbers $q,p$ in the interval $\left(0,1\right)$ that satisfy
$q<p$. Let $f\colon\left\{ 0,1\right\} ^{n}\to\left\{ 0,1\right\} $
be the function defined by 
\[
f\left(x\right)=\begin{cases}
1 & x_{1}=1,\mbox{ and }\sum_{i=2}^{n}x_{i}>qn\\
1 & x_{1}=0,\mbox{ and }\sum_{i=2}^{n}x_{i}>pn\\
0 & \mbox{Otherwise}
\end{cases}.
\]
The Central Limit Theorem implies that 
\[
\mu_{q}\left(f\right)=\frac{q}{2}+o\left(1\right)\text{ and }\mu_{p}\left(f\right)=\frac{\left(1+p\right)}{2}+o\left(1\right).
\]
Since both $\mu_{q}\left(f\right)$ and $\mu_{p}\left(f\right)$ are
bounded away from 0 and 1, we obtain that $f$ has an $\epsilon$-coarse
threshold for some constant $\epsilon$ independent of $n$. On the
other hand, it is easy to see that $f$ is not $\left(\mu_{p},\frac{\left(1-p\right)}{4}\right)$-close
to an $O\left(1\right)$-junta and is not $\left(\mu_{q},q\left(1-q\right)\right)$-close
to an $O\left(1\right)$-junta, provided that $n$ is sufficiently
large. However, if we take $g$ to be the \emph{dictator} function
defined by $g\left(x\right)=x_{1},$ then we have 
\[
\Pr_{\x\sim\mu_{q}}\left[f\left(\x\right)>g\left(\x\right)\right]=o\left(1\right)\text{ and }\Pr_{\x\sim\mu_{p}}\left[f\left(\x\right)<g\left(\x\right)\right]=o\left(1\right),
\]
 as Theorem \ref{thm:Monotone approximation} guarantees.
\end{example}

The proof of Theorem \ref{thm:Monotone approximation} is based on
the invariance principle of Mossel, O'Donnell, and Oleszkiewicz \cite{mossel2010noise}
and on a recent unpublished regularity lemma of O\textquoteright Donnell,
Servedio, Tan, and Wan. A presentation of their proof was recently
given by Jones \cite{jones2016noisy}. 

For our next extension of the Friedgut--Kalai Theorem, we need the
notion of $\left(r,\epsilon,\mu_{p}\right)$\emph{-regularity}, (see
O'Donnell \cite[Chapter 7]{o2014analysis} for more about this notion).
Let $R$ be a subset of $\left[n\right],$ and let $y\in\left\{ 0,1\right\} ^{R}$.
We write $f_{R\to y}$ for the Boolean function on the domain $\left\{ 0,1\right\} ^{\left[n\right]\backslash R}$
defined by $f_{R\to y}\left(x\right)=f\left(z\right),$ where $z$
is the vector whose projection to $\left\{ 0,1\right\} ^{R}$ is $y$
and whose projection to $\left\{ 0,1\right\} ^{\left[n\right]\backslash R}$
is $x.$ 

Note that a function $f$ is a $j$-junta if there exists a set $J$
of size $j$, such that all the restrictions $f_{J\to x}$ are constant
functions. On the other extreme, we have the following notion of regularity
which could be thought of as the complete opposite of being a junta.
It says that for each set $J$ of constant size $r$, the $\mu_{p}$
measures of $f$ and of $f_{J\to y}$ are not far apart. 
\begin{defn}
A function $f\colon\left\{ 0,1\right\} ^{n}\to\left[0,1\right]$ is
said to be $\left(r,\epsilon,\mu_{p}\right)$-regular if 
\[
\left|\mu_{p}\left(f_{J\to y}\right)-\mu_{p}\left(f\right)\right|<\epsilon
\]
 for each set $J\subseteq\left[n\right]$ of size at most $r$ and
each $y\in\left\{ 0,1\right\} ^{J}$.
\end{defn}

As we explain below the following is a robust version of the Friedgut--Kalai
Theorem. 
\begin{thm}
\label{thm:Robust version of Friedgut-Kalai Theorem} For each $\epsilon>0$,
there exists $\delta>0$, such that the following holds. Let $q,p\in\left(\epsilon,1-\epsilon\right)$
and suppose that $p>q+\epsilon$. Let $f,g\colon\left\{ 0,1\right\} ^{n}\to\left[0,1\right]$.
Suppose that 
\[
\e_{\left(\x,\y\right)\sim D\left(q,p\right)}\left[\left(1-g\left(\y\right)\right)f\left(\x\right)\right]<\delta,
\]
 and that the function $f$ is $\left(\left\lceil \frac{1}{\delta}\right\rceil ,\delta,\mu_{q}\right)$-regular.
Then either $\mu_{q}\left(f\right)<\epsilon$, or $\mu_{p}\left(g\right)>1-\epsilon.$
\end{thm}

Theorem \ref{thm:Robust version of Friedgut-Kalai Theorem} is a robust
version of Theorem \ref{thm:Friedgut-Kalai}. Indeed, note that one
can equivalently restate Theorem \ref{thm:Friedgut-Kalai} as follows.
Let $q$ and $p$ be numbers in the interval $\left(\epsilon,1-\epsilon\right)$,
and suppose that $p>q+\epsilon$. Let $f\colon\left\{ 0,1\right\} ^{n}\to\left\{ 0,1\right\} $
be a monotone transitive symmetric Boolean function. Then we either
have $\mu_{q}\left(f\right)<\epsilon$, or we have $\mu_{p}\left(g\right)>1-\epsilon$,
provided that $n$ is sufficiently large.

Applying Theorem \ref{thm:Robust version of Friedgut-Kalai Theorem}
(with $f=g$), we see that it strengthens Theorem \ref{thm:Friedgut-Kalai}
in the following ways. It shows that we may replace the hypothesis
that $f$ is monotone by the weaker hypothesis that $f$ is $\left(q,p,\delta\right)$-almost
monotone, and that we may replace the hypothesis that $f$ is transitive
symmetric, with the weaker hypothesis that the $f$ is $\left(\left\lceil \frac{1}{\delta}\right\rceil ,\delta,\mu_{q}\right)$-regular
for some $\delta=\delta\left(n\right),$ where $\lim_{n\to\infty}\delta\left(n\right)=0.$
Example \ref{ex:transitive symmetric are Fourier regular} below shows
that the latter hypothesis is indeed weaker.
\begin{rem}
While Theorem \ref{thm:Robust version of Friedgut-Kalai Theorem}
is more general than the Friedgut--Kalai Theorem, we remark that
the Friedgut--Kalai Theorem is better in the quantitative aspects
that we have not addressed. We would also like to remark that the
proof of Theorem \ref{thm:Monotone approximation} is very different
than the standard proofs of Theorems \ref{thm:Friedgut-Kalai} and
\ref{thm:Friedgut's junta theorem}. While the traditional proofs
are based on the hypercontractivity theorem of Bonami, Gross, and
Beckner \cite{bonami1970etude,gross1975logarithmic,beckner1975inequalities}
and on Russo's Lemma \cite{russo1982approximate}, our proof of Theorem
\ref{thm:Robust version of Friedgut-Kalai Theorem} is based instead
on the invariance principle of O'Donnell, Mossel, and Oleszkiewicz
\cite{mossel2010noise}. 
\end{rem}

\subsection{Sketch of the proof of Theorems \ref{thm:Monotone approximation}
and \ref{thm:Robust version of Friedgut-Kalai Theorem}}

Our proof of Theorem \ref{thm:Monotone approximation} is based on
the \emph{regularity method}. In the setting of the regularity method
we are given a space $\s$, and our goal is to show a `removal lemma'
asserting that any subset $A\subseteq\s$ that contain few copies
of a given `forbidden' configuration may be approximated by a family
that contains no copies of that configuration. The proof contains
two ingredients.
\begin{enumerate}
\item A\emph{ regularity lemma} showing that for any set $A$, we may decompose
$B$ into some parts, such that the intersections of $A$ with `almost
all' of the parts are either `quasirandom' or `small'. 
\item A \emph{counting lemma} showing that if we take the quasirandom parts,
then together they contain many forbidden configurations. 
\end{enumerate}
These two ingredients are put together by approximating $A$ by the
set $J$ defined to be the union of all the quasirandom parts of $B$.
The task is then to use the counting lemma to show that any forbidden
configuration that appears in $J$ results in many forbidden configurations
back in $A.$

The invariance principle of Mossel O'Donnell and Oleszkiewicz \cite{mossel2010noise}
considers a notion of smoothness called \emph{small noisy influences}.
It roughly says that we may replace the variables of a smooth function
$f\colon\left\{ 0,1\right\} ^{n}\to\left[0,1\right]$ by Gaussian
random variables and obtain a function that behaves similarly. We
call this function the Gaussian analogue of $f$. The proof of Theorem
\ref{thm:Monotone approximation} goes through the following steps. 
\begin{enumerate}
\item We apply a regularity lemma presented by Jones \cite{jones2016noisy},
which shows that we may find a set $J$ of constant size that decomposes
$f$ into the parts $\left\{ f_{J\to y}\right\} _{y\in\left\{ 0,1\right\} ^{J}}$,
such that almost all of the parts either have expectation very close
to 0, or have small noisy influences. 
\item We give a counting lemma that shows that if two functions $f_{1},f_{2}\colon\left\{ 0,1\right\} ^{n}\to\left[0,1\right]$
have small noisy influences and satisfy $\e_{\x\sim\mu_{q}}\left[f_{1}\left(\x\right)\right],\e_{\boldsymbol{y}\sim\mu_{p}}\left(1-f_{2}\left(\y\right)\right)=\Theta\left(1\right)$,
then 
\[
\e_{\x,\y\sim D\left(q,p\right)}\left[f_{1}\left(\x\right)\left(1-f_{2}\left(\y\right)\right)\right]=\Theta\left(1\right).
\]
 
\end{enumerate}
The proof of the second part follows \cite{mossel2010noise}. We express
$\e_{\x,\y\sim D\left(q,p\right)}\left[f_{1}\left(\x\right)\left(1-f_{2}\left(\y\right)\right)\right]$
in terms of the Fourier expansions of $f_{1}$ and $f_{2}$, and we
show that this expression can be approximated by a similar expression
involving the Gaussian analogues of $f_{1}$ and $f_{2}$. We then
apply a classical theorem by Borell \cite{borell1985geometric} to
lower bound the value of the corresponding expression. 

The proof of Theorem \ref{thm:Robust version of Friedgut-Kalai Theorem}
is similar. Suppose that $f$ is a $\left(\left\lceil \frac{1}{\delta}\right\rceil ,\delta,\mu_{q}\right)$-regular
function with $\mu_{q}\left(f\right)>\epsilon$.
\begin{enumerate}
\item We apply the regularity lemma of \cite{jones2016noisy} to find a
set $J$ of constant size that decomposes $f$ into the parts $\left\{ f_{J\to y}\right\} _{y\in\left\{ 0,1\right\} ^{J}}$,
such that most of the parts either have expectations very close to
0, or have small noisy influences. The $\left(\left\lceil \frac{1}{\delta}\right\rceil ,\delta,\mu_{q}\right)$-regularity
of $f$ implies that there are no parts with expectations close to
0, so only the latter option is available. 
\item We note that the term 
\[
\e_{\mathbf{z},\boldsymbol{w}\sim\left(\left\{ 0,1\right\} ^{n-\left|J\right|},D\left(q,p\right)\right)}\left[f_{J\to x}\left(\boldsymbol{\mathbf{z}}\right)\left(1-g_{J\to y}\left(\boldsymbol{w}\right)\right)\right]
\]
is small for each $x,y\in\left\{ 0,1\right\} ^{J},$ such that $x_{i}\le y_{i}$
for each $i$. 
\item We deduce from the above counting lemma (applied with $f_{1}=f_{J\to x},f_{2}=g_{J\to y}$)
that for such $x,y$, if $f_{J\to x}$ has small noisy influences,
then the function $g_{J\to y}$ has expectation close to 1. 
\item It is easy that for `almost all' $y\in\left\{ 0,1\right\} ^{J}$ we
may find $x\in\left\{ 0,1\right\} ^{J}$ with $x_{i}\le y_{i}$ for
each $i$, such that $f_{J\to x}$ has small noisy influences. So
Step 3 implies that for almost all $y$ the expectation of $g_{J\to y}$
is close to $1$. Therefore, the expectation of $g$ is close to $1$. 
\end{enumerate}

\section{Prior results and notions that we make use of}

In this section we review some facts on the Fourier analysis of the
$p$-biased cube. Many of them are standard results that can be found
e.g. in O'Donnell \cite[Chapters 2,8, and 11]{o2014analysis}. 

\subsection{Fourier analysis on the $p$-biased cube}

Given a distribution $D$ on a space $\Omega$, we write $\boldsymbol{x}\sim\left(\Omega,D\right)$
or $\boldsymbol{x}\sim D$ to denote that$\boldsymbol{x}$ is chosen
out of $\Omega$ according to the distribution $D$. We shall use
bold letters to denote random variables. 

We denote by $L^{2}\left(\left\{ 0,1\right\} ^{n},\mu_{p}\right)$
the Hilbert space of function $f\colon\left\{ 0,1\right\} ^{n}\to\mathbb{R}$
equipped with the $p$-biased inner product 
\[
\left\langle f,g\right\rangle =\e_{\mathbf{x}\sim\left(\left\{ 0,1\right\} ^{n},\mu_{p}\right)}\left[f\left(\mathbf{x}\right)g\left(\mathbf{x}\right)\right].
\]
The \emph{$p$-biased norm} is defined by setting $\|f\|=\sqrt{\left\langle f,f\right\rangle }$.
In any time that we write that $f$ is an element of the space $L^{2}\left(\left\{ 0,1\right\} ^{n},\mu_{p}\right),$
we shall use the shorthand $\e\left[f\right]$ for $\e_{\mathbf{x}\sim\left(\left\{ 0,1\right\} ^{n},\mu_{p}\right)}\left[f\right]$. 

The $p$-biased Fourier characters are an orthonormal basis of $L^{2}\left(\left\{ 0,1\right\} ^{n},\mu_{p}\right)$
defined as follows.
\begin{defn}
Let $i\in\left[n\right]$. The \emph{Fourier character} corresponding
to the singleton $\left\{ i\right\} $ is the function $\chi_{i}^{p}\in L^{2}\left(\left\{ 0,1\right\} ^{n},\mu_{p}\right)$
defined by the formula 
\[
\chi_{i}^{p}\left(x\right):=\begin{cases}
-\sqrt{\frac{1-p}{p}} & x_{i}=1\\
\sqrt{\frac{p}{1-p}} & x_{i}=0
\end{cases}.
\]
More generally, let $S$ be a subset of $\left[n\right]$. The \emph{Fourier
character} corresponding to the set $S\subseteq\left[n\right]$ is
the function $\chi_{S}^{p}:=\prod_{i\in S}\chi_{i}^{p}.$
\end{defn}

The Fourier characters are known to be an orthonormal basis for $L^{2}\left(\left\{ 0,1\right\} ^{n},\mu_{p}\right)$.
Thus, each function has a unique expansion of the form $f=\sum_{S\subseteq\left[n\right]}\hat{f}\left(S\right)\cps,$
where $\hat{f}\left(S\right)=\left\langle f,\cps\right\rangle .$
This expansion is called the \emph{$p$-biased Fourier expansion}
of $f$, or just the Fourier expansion of $f$, where $p$ is clear
from context. We also have the following identities known as the \emph{Parseval
identities}.

\[
\e\left[f^{2}\right]=\left\langle f,f\right\rangle =\sum_{S\subseteq\left[n\right]}\hat{f}\left(S\right)^{2}
\]
\[
\e\left[fg\right]=\left\langle f,g\right\rangle =\sum_{S\subseteq\left[n\right]}\hat{f}\left(S\right)\hat{g}\left(S\right).
\]

For any $T\subseteq\left[n\right]$, the \emph{averaging }operator
\[
\av_{T}\colon L^{2}\left(\left\{ 0,1\right\} ^{n},\mu_{p}\right)\to L^{2}\left(\left\{ 0,1\right\} ^{n-\left|T\right|},\mu_{p}\right)
\]
 is defined by setting 
\[
\av_{T}\left[f\right]\left(x\right)=\e\left[f_{\left[n\right]\backslash T\to x}\right].
\]
 The operator $\av_{T}\left[f\right]$ has the following nice Fourier
analytical interpretation. It is the operator that annihilates all
the Fourier coefficients that correspond to sets that have nonempty
intersection with $T$. 
\begin{fact}
\label{Fact average} Let $f\in L^{2}\left(\left\{ 0,1\right\} ^{n},\mu_{p}\right)$
be a function that has the Fourier expansion 
\[
f=\sum_{S\subseteq\left[n\right]}\hat{f}\left(S\right)\cps.
\]
 Then 
\[
\av_{T}\left[f\right]=\sum_{S\subseteq\left[n\right]\backslash T}\hat{f}\left(S\right)\cps
\]
 
\end{fact}

Another notion of importance for us is the notion of \emph{influence}
by Ben-Or and Linial \cite{ben1990collective}. 
\begin{defn}
The \emph{$p$-biased $i$th influence} of a function $f\colon\left\{ 0,1\right\} ^{n}\to\mathbb{R}$
whose Fourier expansion is $\sum_{S\subseteq\left[n\right]}\hat{f}\left(S\right)\cps$
is defined by setting 
\begin{equation}
\mathrm{Inf}_{i}^{p}\left[f\right]=\e\left[\left(f-A_{\left\{ i\right\} }\left[f\right]\right)^{2}\right]=\sum_{S\ni i}\hat{f}\left(S\right)^{2}.\label{eq:influence-Fourier}
\end{equation}
\end{defn}

We shall also need to introduce the noise operator. 
\begin{defn}
Given  $x\in\left\{ 0,1\right\} ^{n}$ the $\left(\rho,p\right)$\emph{-noisy
distribution} of $x$ denoted by $N_{\rho,p}\left(x\right)$ is a
probability distribution on elements $\mathbf{y}\in\left\{ 0,1\right\} ^{n}$,
where we set each coordinate $\mathbf{y}_{i}$ independently to be
$x_{i}$ with probability $\rho$, and to a new $p$-biased element
of $\left\{ 0,1\right\} $ with probability $1-\rho.$ 

The \emph{noise operator }$\mathsf{T}_{\rho,p}$ on the space $L^{2}\left(\left\{ 0,1\right\} ^{n},\mu_{p}\right)$
is the operator that associates to each $f\in L^{2}\left(\left\{ 0,1\right\} ^{n},\mu_{p}\right)$
the function 
\[
\mathsf{T}_{\rho,p}\left[f\right]:=\underset{\mathbf{y}\sim N_{\rho,p}\left(x\right)}{\e}\left[f\left(\mathbf{y}\right)\right].
\]
\end{defn}

We have the following Fourier formula for $\mathsf{T}_{\rho,p}\left[f\right].$ 
\begin{fact}
\label{Fact: T_rho}Let $\rho,p\in\left(0,1\right)$, and let 
\[
f=\sum_{S\subseteq\left[n\right]}\hat{f}\left(S\right)\cps
\]
 be a function in $L^{2}\left(\left\{ 0,1\right\} ^{n},\mu_{p}\right)$.
Then 
\[
\mathsf{T}_{\rho,p}\left[f\right]=\sum_{S\subseteq\left[n\right]}\rho^{\left|S\right|}\hat{f}\left(S\right)\cps.
\]
 
\end{fact}

\subsection{The directed noise operators}

We shall now introduce a directed analogue of the noise operator.
Recall that $D\left(q,p\right)$ is the joint distribution on elements
$\left(\mathbf{x},\mathbf{y}\right)\in\left\{ 0,1\right\} ^{n}\times\left\{ 0,1\right\} ^{n},$
such that 
\[
\mathbf{x}\sim\mu_{q},\mathbf{y}\sim\mu_{p},\,\,\,\forall i:\,\mathbf{y}_{i}\ge\mathbf{x}_{i}.
\]
 We define an operator 
\[
\tra\colon L^{2}\left(\left\{ 0,1\right\} ^{n},\mu_{p}\right)\to L^{2}\left(\left\{ 0,1\right\} ^{n},\mu_{q}\right),
\]
 and its adjoint 
\[
\art\colon L^{2}\left(\left\{ 0,1\right\} ^{n},\mu_{q}\right)\to L^{2}\left(\left\{ 0,1\right\} ^{n},\mu_{p}\right),
\]
 by setting 

\[
\tra\left(f\right)\left(x\right)=\e_{\mathbf{x},\mathbf{y}\sim D\left(q,p\right)}\left[f\left(\mathbf{y}\right)\,|\,\mathbf{x}=x\right],
\]

and 
\[
\art\left(f\right)\left(y\right)=\e_{\mathbf{x},\mathbf{y}\sim D\left(q,p\right)}\left[f\left(\mathbf{x}\right)\,|\,\mathbf{y}=y\right].
\]
These operators were first studied by Ahlberg, Broman, Griffiths,
and Morris \cite{ahlberg2014noise}, and then again by Abdullah and
Venkatasubramania \cite{abdullah2015directed}. The one sided noise
operator has the following Fourier formulas:
\begin{lem}
\label{lem:Fourier expnsion of one sided noise operator} Let $p>q$,
and set $\rho=\sqrt{\frac{q\left(1-p\right)}{p\left(1-q\right)}}.$
Let $f=\sum_{S\subseteq\left[n\right]}\hat{f}\left(S\right)\cqs$
be a function in $L^{2}\left(\left\{ 0,1\right\} ^{n},\mu_{q}\right)$.
Then $\art\left(f\right)$ has the $p$-biased Fourier expansion:
\[
\art\left(f\right)=\sum_{S\subseteq\left[n\right]}\rho^{\left|S\right|}\hat{f}\left(S\right)\cps.
\]
Similarly, if $g=\sum_{S\subseteq\left[n\right]}\hat{g}\left(S\right)\cps$
is a function in $L^{2}\left(\left\{ 0,1\right\} ^{n},\mu_{p}\right)$,
then the function $\tra\left(g\right)$ has the $q$-biased Fourier
expansion 
\[
\tra\left(g\right)=\sum_{S\subseteq\left[n\right]}\rho^{\left|S\right|}\hat{g}\left(S\right)\chi_{S}^{q}.
\]
\end{lem}

\begin{proof}
We shall prove it for the operator $\art,$ as the proof for the other
operator $\tra$ will be similar. By linearity, it is enough to prove
the lemma in the case where $f=\cqs$ for some $S\subseteq\left[n\right].$
Let $y\in\left\{ 0,1\right\} ^{n}.$ Note that 
\begin{align}
\art\left[\cqs\right]\left(y\right) & =\e_{\mathbf{x},\mathbf{y}\sim D\left(q,p\right)|\,\mathbf{y}=y}\left[\cqs\left(\mathbf{x}\right)\right]=\e_{\mathbf{x},\mathbf{y}\sim D\left(q,p\right)|\,\mathbf{y}=y}\left[\prod_{i\in S}\chi_{i}^{q}\left(\mathbf{x}_{i}\right)\right].\label{eq:comp1}
\end{align}
 Claim \ref{claim:comutation silly} below shows that 
\[
\e_{\mathbf{x},\mathbf{y}\sim D\left(q,p\right)|\,\mathbf{y}=y}\left[\chi_{i}^{q}\left(\mathbf{x}_{i}\right)\right]=\rho\chi_{i}^{p}\left(y_{i}\right).
\]
 By the independence of the random variables $\chi_{i}^{q}\left(\mathbf{x}_{i}\right)$
for any $\mathbf{x},\mathbf{y}\sim D\left(q,p\right)|\,\mathbf{y}=y$
we obtain: 
\begin{align}
\e_{\mathbf{x},\mathbf{y}\sim D\left(q,p\right)|\,\mathbf{y}=y}\left[\prod_{i\in S}\chi_{\left\{ i\right\} }^{q}\left(\mathbf{x}_{i}\right)\right] & =\prod_{i\in S}\e_{\mathbf{x},\mathbf{y}\sim D\left(q,p\right)|\,\mathbf{y}=y}\left[\chi_{i}^{q}\left(\mathbf{x}_{i}\right)\right]\label{eq:comp2}\\
 & =\rho^{\left|S\right|}\prod_{i\in S}\chi_{i}^{p}\left(y_{i}\right)=\rho^{\left|S\right|}\chi_{S}\left(y\right).\nonumber 
\end{align}
Combining (\ref{eq:comp1}) with (\ref{eq:comp2}), we complete the
proof.
\end{proof}
\begin{claim}
\label{claim:comutation silly} Let $x,y$ be elements of $\left\{ 0,1\right\} ^{n},$
let $p>q\in\left(0,1\right),$ and let $\rho=\sqrt{\frac{q\left(1-p\right)}{p\left(1-q\right)}}.$
Then 
\begin{equation}
\e_{\mathbf{x},\mathbf{y}\sim D\left(q,p\right)|\,\mathbf{y}=y}\left[\chi_{i}^{q}\left(\mathbf{x}_{i}\right)\right]=\rho\chi_{i}^{p}\left(y_{i}\right),\label{eq:evil expectation 1}
\end{equation}
 and 
\begin{equation}
\e_{\mathbf{x},\mathbf{y}\sim D\left(q,p\right)|\,\mathbf{x}=x}\left[\chi_{i}^{p}\left(\mathbf{y}_{i}\right)\right]=\rho\chi_{i}^{q}\left(x_{i}\right).\label{eq: Evil expectation 2}
\end{equation}
 
\end{claim}

\begin{proof}
Since the functions $\chi_{i}^{q},\chi_{i}^{p}$ depend only on the
$i$th coordinate we may assume that $n=1$, and we shall write $\chi^{p}=\chi_{1}^{p}$
as well as $\chi^{q}=\chi_{1}^{q}$ for brevity. We shall start by
showing (\ref{eq:evil expectation 1}), and the proof of (\ref{eq: Evil expectation 2})
will be similar. Let $h\in L^{2}\left(\left\{ 0,1\right\} ^{n},\mu_{p}\right)$
be the map 
\[
y\mapsto\e_{\mathbf{x},\mathbf{y}\sim D\left(q,p\right)|\,\mathbf{y=}y}\left[\chi^{q}\left(\mathbf{x}\right)\right].
\]
 Note the space $L^{2}\left(\left\{ 0,1\right\} ,\mu_{p}\right)$
is a linear space of dimension $2$. We shall show that $h=\rho\chi^{p}$
by showing that there are two independent linear functionals on the
space $L^{2}\left(\left\{ 0,1\right\} ,\mu_{p}\right)$ that agree
on the functions $h$ and $\rho\chi_{p}$. Namely, the first functional
is the functional of evaluating at $0$, and the second functional
is the expectation according to the $p$-biased distribution. Indeed,
we may use the fact that elements $\boldsymbol{x,y}\sim D\left(q,p\right)$
satisfy $\boldsymbol{x}_{i}\le\boldsymbol{y}_{i}$ with probability
$1$ to obtain: 
\[
\e_{\mathbf{x},\mathbf{y}\sim\left(\left\{ 0,1\right\} ,D\left(q,p\right)\right)|\,\mathbf{y}=0}\left[\chi^{q}\left(\mathbf{x}\right)\right]=\chi^{q}\left(0\right)=\sqrt{\frac{q}{1-q}}=\rho\sqrt{\frac{p}{1-p}}=\rho\chi^{p}\left(0\right).
\]
On the other hand, 
\begin{align*}
\e_{\z\sim\mu_{p}}\left[h\left(\z\right)\right] & =\e_{\mathbf{z}\sim\mu_{p}}\left[\e_{\mathbf{x},\mathbf{y}\sim D\left(q,p\right)|\,\mathbf{y=\mathbf{z}}}\left[\chi^{q}\left(\mathbf{x}\right)\right]\right]\\
 & =\e_{\mathbf{x},\mathbf{y}\sim D\left(q,p\right)}\left[\chi^{q}\left(\mathbf{x}\right)\right]=\e_{\mathbf{x}\sim\mu_{q}}\left[\chi_{q}\left(\mathbf{x}\right)\right]=0\\
 & =\e_{\z\sim\mu_{p}}\left[\rho\chi_{p}\left(\boldsymbol{\boldsymbol{z}}\right)\right].
\end{align*}
 Since the expectation functional and the evaluating by 0 functionals
are independent, and since the space $L^{2}\left(\left\{ 0,1\right\} ,\mu_{p}\right)$
is of dimension 2, we obtain $h=\rho\chi_{p}$. This completes the
proof of (\ref{eq:evil expectation 1}). 

We prove (\ref{eq: Evil expectation 2}) in a similar fashion. Define
$h\in L^{2}\left(\left\{ 0,1\right\} ,\mu_{q}\right)$ by 
\[
x\mapsto\e_{\mathbf{x},\mathbf{y}\sim D\left(q,p\right)|\,\mathbf{x}=x}\left[\chi^{p}\left(\mathbf{y}\right)\right].
\]
 Similarly to the proof of (\ref{eq:evil expectation 1}), it is enough
to prove the identities 
\[
\e_{\mathbf{z}\sim\mu_{q}}\left(h\right)=0,h\left(1\right)=\rho\chi_{q}\left(1\right).
\]
 To prove the former, note that 
\[
\e_{\mathbf{z}\sim\mu_{q}}\left(h\left(\z\right)\right)=\e_{\mathbf{z}\sim\mu_{q}}\left[\e_{\mathbf{x},\mathbf{y}\sim D\left(q,p\right)|\,\mathbf{x}=\mathbf{z}}\left[\chi^{p}\left(\mathbf{y}\right)\right]\right]=\e_{\mathbf{x},\mathbf{y}\sim D\left(q,p\right)}\left[\chi^{p}\left(\mathbf{y}\right)\right]=0.
\]
 To prove the latter, note that 
\[
h\left(1\right)=\e_{\mathbf{x},\mathbf{y}\sim D\left(q,p\right)|\,\mathbf{x}=1}\left[\chi^{p}\left(\mathbf{y}\right)\right]=\chi_{p}\left(1\right)=\rho\chi_{q}\left(1\right).
\]
\end{proof}

\subsection{Fourier regularity}

We shall say that a function $f\colon\left\{ 0,1\right\} ^{n}\to\mathbb{R}$
is $\left(r,\delta,\mu_{p}\right)$-\emph{Fourier regular} if $\left|\hat{f}\left(S\right)\right|<\delta$
for each $0<\left|S\right|\le r$. It is easy to see (see O'Donnell
\cite[Chapter 7]{o2014analysis}) that any $\left(r,\delta,\mu_{p}\right)$-regular
function is $\left(r,\delta,\mu_{p}\right)$-Fourier regular, and
on the converse any $\left(r,\delta,\mu_{p}\right)$-Fourier regular
function is $\left(r,2^{r}\delta,\mu_{p}\right)$-regular. So in a
sense these notions are equivalent. 
\begin{example}
\label{ex:transitive symmetric are Fourier regular} Let $f\colon\left\{ 0,1\right\} ^{n}\to\left[0,1\right]$
be a transitive symmetric function. Then $f$ is $\left(r,\sqrt{\frac{r}{n}},\mu_{p}\right)$-Fourier
regular for any $r$ and $p$. Indeed, let $S\in\binom{\left[n\right]}{r},$
then the fact that $f$ is transitive symmetric implies that there
exist distinct $r$-subsets of $\left[n\right]$, $S_{1},\ldots,S_{\left\lceil \frac{n}{r}\right\rceil }$,
such that $\hat{f}\left(S_{i}\right)=\hat{f}\left(S\right)$ for each
$i$. By Parseval's identity, we have 
\[
1\ge\|f\|_{\mu_{p}}^{2}\ge\sum_{i=1}^{\left\lceil \frac{n}{r}\right\rceil }\hat{f}\left(S_{i}\right)^{2}=\left\lceil \frac{n}{r}\right\rceil \hat{f}\left(S\right)^{2}.
\]
 After rearranging, we obtain that $f$ is $\left(r,2^{r}\sqrt{\frac{r}{n}},\mu_{p}\right)$-Fourier
regular, so it is in fact $\left(r,\delta,\mu_{p}\right)$-Fourier
regular, provided that $n>\frac{4^{r}}{\delta^{2}}$.
\end{example}

\subsection{The noisy influences}

Let $f\in L^{2}\left(\left\{ 0,1\right\} ^{n},\mu_{p}\right)$ be
a function. The noise stability of $f$ is defined by 
\[
\mathrm{Stab}_{\rho,p}\left(f\right):=\left\langle T_{\rho}\left(f\right),f\right\rangle _{\mu_{p}}=\underset{\mathbf{x}\sim\mu_{p},\mathbf{y}\sim N_{\rho,p}\left(x\right)}{\mathbb{E}}\left[f\left(\mathbf{x}\right)f\left(\mathbf{y}\right)\right].
\]
By Fact \ref{Fact: T_rho}, and by Parseval's identity, we have 
\begin{equation}
\mathrm{Stab}_{\rho,p}\left(f\right)=\sum_{S\subseteq\left[n\right]}\rho^{\left|S\right|}\hat{f}\left(S\right)^{2},\label{eq:stab rho}
\end{equation}
 where $\hat{f}\left(S\right)$ are the Fourier coefficients of $f$
with respect to the $p$-biased distribution. 

The $\left(\rho,\mu_{p}\right)$-noisy influences of $f$ are defined
by 

\[
\mathrm{Inf}_{i}^{\left(\rho,p\right)}\left[f\right]:=\mathrm{Stab}_{\rho,p}\left[\left(f-A_{\left\{ i\right\} }\left[f\right]\right)\right].
\]
 By Fact \ref{Fact average} and by (\ref{eq:stab rho}) we have 
\begin{equation}
\mathrm{Inf}_{i}^{\left(\rho,p\right)}\left[f\right]=\sum_{S\ni i}\rho^{\left|S\right|}\hat{f}\left(S\right)^{2}.\label{eq:inf rho}
\end{equation}
 
\begin{defn}
Let $\delta>0,\rho,p\in\left[0,1\right].$ A function $f\colon\left\{ 0,1\right\} ^{n}\to\mathbb{R}$
is said to have $\left(\rho,\delta,\mu_{p}\right)$-small noisy influences
if $\mathrm{Inf}_{i}^{\left(\rho,p\right)}\left[f\right]<\delta$
for every $i\in\left[n\right].$ 
\end{defn}

\subsection{Regularity lemmas we use}

We shall make use of the following regularity lemma presented by Jones
\cite{jones2016noisy}. 
\begin{thm}
\label{thm:Jones regularity} For each $\epsilon>0$ there exists
 $j\in\mathbb{N}$ such that the following holds. Let $p\in\left(\epsilon,1-\epsilon\right)$
and let $f\in L^{2}\left(\left\{ 0,1\right\} ,\mu_{p}\right)$ be
a function. Then there exists a set $J$ of size at most $j$, such
that if we choose  $\mathbf{x}\sim\left(\left\{ 0,1\right\} ^{J},\mu_{p}\right)$,
then the functions $f_{J\to\mathbf{x}}$ has $\left(1-\epsilon,\epsilon,\mu_{p}\right)$-small
noisy influences with probability at least $1-\epsilon.$ 
\end{thm}

\begin{rem}
Jones \cite{jones2016noisy} proved Theorem \ref{thm:Jones regularity}
only for the case where $p=\frac{1}{2}.$ However, as in most of the
results in the area, their proof can be extended verbatim to the $p$-biased
distribution, for any $p$ bounded away from 0 and 1. 
\end{rem}

We also make use of the following regularity lemma of \cite{ellis2016stability}.
\begin{thm}[{ \cite[Theorem 1.7]{ellis2016stability}}]
\label{thm:k-uniform regularity} For each $\delta,\epsilon>0$,
there exists $j\in\mathbb{N}$, such that the following holds. Let
$\delta<\frac{k}{n}<1-\delta,$ and let $\f\subseteq\binom{\left[n\right]}{k}$
be a family. Then there exists a set $J$ of size at most $j$ and
a family $\g\subseteq\p\left(J\right)$, such that:

\begin{enumerate}
\item We have $\mu\left(\f\setminus\left\langle \g\right\rangle \right)<\epsilon.$
\item For each $B\in\g$ the family $\f_{J}^{B}$ is $\left(\left\lceil \frac{1}{\delta}\right\rceil ,\delta\right)$-regular
and $\mu\left(\f_{J}^{B}\right)>\frac{\epsilon}{2}$.
\end{enumerate}
\end{thm}

\subsection{Functions on Gaussian spaces}

Let $\gamma$ be the standard normal probability distribution $N\left(0,1\right)$
on $\mathbb{R}$. Abusing notation, we will also use $\gamma$ to
denote the product normal probability distribution $N\left(0,1\right)^{n}$
on $\mathbb{R}^{n}.$ We shall denote by $L^{2}\left(\mathbb{R}^{n},\gamma\right)$
the space of functions $f\colon\mathbb{R}^{n}\to\mathbb{R}$, such
that $\|f\|_{\gamma}:=\mathbb{E}_{\gamma}\left[f^{2}\right]<\infty$.
This space is equipped with the inner product 
\[
\left\langle f,g\right\rangle =\mathbb{E}_{\gamma}\left[fg\right]=\int_{\mathbb{R}^{n}}f\left(x\right)g\left(x\right)\gamma\left(x\right)dx.
\]

The operator $\mathsf{T}_{\rho}$ on the space $\left(\mathbb{R}^{n},N\left(0,1\right)\right)$,
also known as the \emph{Ornstein-Uhlenbeck} \emph{operator}, is defined
as follows.
\begin{defn}
Let $\rho\in\left(0,1\right)$, and let $x\in\mathbb{R}^{n}$, the
$\rho$\emph{-noisy distribution }of $x$ is the distribution $N_{\rho,\gamma}\left(x\right)$,
where we choose $\mathbf{y}$ by setting each coordinate $\mathbf{y}_{i}$
independently to be $\rho x_{i}+\sqrt{1-\rho^{2}}\mathbf{z}_{i}$,
where $\z$ is a new independent $\gamma$-distributed element of
$\mathbb{R}$. The \emph{noise operator }$\mathsf{T}_{\rho}$ on the
space $L^{2}\left(\mathbb{R}^{n},\gamma\right)$ is the operator that
associates to each $f\in L^{2}\left(\mathbb{R}^{n},\gamma\right)$
the function 
\[
\mathsf{T}_{\rho}\left[f\right]\left(x\right):=\underset{\mathbf{y}\sim N_{\rho,\gamma}\left(x\right)}{\e}\left[f\left(\mathbf{y}\right)\right].
\]
\end{defn}

\begin{rem}
The analogy between the distribution\textbf{ $N_{\rho,p}$}, and $N_{\rho,\gamma}$
stems from the fact if we choose $\mathbf{x}\sim\gamma,$ and $\mathbf{y}\sim N_{\rho,\gamma}\left(\mathbf{x}\right),$
then we have the following properties. 

\begin{itemize}
\item $\mathbf{x},\mathbf{y}\sim\gamma.$ 
\item $\forall i\,:\mathbb{E}\left[\mathbf{x}_{i}\mathbf{y}_{i}\right]=\rho.$
\item The $\mathbb{R}^{2}$-valued random variables$\left(\mathbf{x}_{i},\mathbf{y}_{i}\right)$
are independent of each other.
\end{itemize}
These properties are similarly satisfied when we choose $\mathbf{x\sim}\mu_{p}$
and then choose $\mathbf{y}\sim N_{\rho,p}\left(\mathbf{x}\right).$
\end{rem}

For $\mu\in\left(0,1\right)$, we let $F_{\mu}\colon\mathbb{R}\to\left[0,1\right]$
denote the function $\mathbf{1}_{x<t}$, where $t$ is the only real
number for which $\e_{\boldsymbol{x}\sim\gamma}\left[F_{\mu}\left(\boldsymbol{x}\right)\right]=\mu$.
The following theorem was proved by Borell \cite{borell1985geometric}.
It says that if $f,g\in L^{2}\left(\mathbb{R}^{n},\gamma\right)$
are  functions that take their value in the interval $\left[0,1\right]$
that satisfy $\e\left[f\right]=\mu,\e\left[g\right]=\nu$, then the
maximal possible value of the quantity $\left\langle \mathsf{T}_{\rho}\left[f\right],g\right\rangle $
is obtained when $f=F_{\mu}$ and $g=F_{\nu}.$ 
\begin{thm}[Borell 1985]
 Let $f,g\in L^{2}\left(\mathbb{R}^{n},\gamma\right)$ be two $\left[0,1\right]$-valued
functions. Then 
\[
\left\langle \mathsf{T}_{\rho}\left[f\right],g\right\rangle \le\left\langle \mathsf{T}_{\rho}\left[F_{\e\left[f\right]}\right],F_{\e\left[g\right]}\right\rangle .
\]
\end{thm}

We denote $\left\langle \mathsf{T}_{\rho}\left[F_{\mu}\right],F_{\nu}\right\rangle $
by $\Lambda_{\rho}\left(\mu,\nu\right).$ Note that we trivially have
\[
\Lambda_{\rho}\left(\mu,\nu\right):=\left\langle \mathsf{T}_{\rho}\left[F_{\mu}\right],F_{\nu}\right\rangle \le\left\langle \mathsf{T}_{\rho}\left[F_{\mu}\right],1\right\rangle =\mu.
\]
We will be interested in the case where $\mu$ is bounded away from
$0$ and $\nu,\rho$ are bounded away from 1. For such parameters,
$\Lambda_{\rho}\left(\mu,\nu\right)$ admits a slightly stronger upper
bound.
\begin{lem}
\label{lem:Estimate from Borell's Thm} For each $\epsilon>0$, there
exists $\delta>0$, such that the following holds. Let $\mu\in\left(\epsilon,1\right)$
and let $\rho,\nu\in\left(0,1-\epsilon\right)$. Then 
\[
\Lambda_{\rho}\left(\mu,\nu\right)\le\mu-\delta.
\]
\end{lem}

\begin{proof}
Let $\delta=\left\langle \mathsf{T}_{1-\epsilon}\left[F_{\epsilon}\right],1-F_{1-\epsilon}\right\rangle >0$.
We have 
\begin{align*}
\Lambda_{\rho}\left(\mu,\nu\right) & =\left\langle \mathsf{T}_{\rho}\left[F_{\mu}\right],F_{\nu}\right\rangle =\left\langle \mathsf{T}_{\rho}\left[F_{\mu}\right],1\right\rangle +\left\langle \mathsf{T}_{\rho}\left[F_{\mu}\right],F_{\nu}-1\right\rangle \\
 & =\mu-\left\langle \mathsf{T}_{\rho}\left[F_{\mu}\right],1-F_{\nu}\right\rangle \le\mu-\left\langle \mathsf{T}_{1-\epsilon}\left[F_{\epsilon}\right],1-F_{1-\epsilon}\right\rangle \\
 & =\mu-\delta.
\end{align*}
 
\end{proof}
We shall also use the following estimate on $\Lambda_{\rho}.$ 
\begin{lem}[{\cite[Lemma 2.5]{mossel2017majority}}]
\label{lem:Estimossel} Let $\rho_{1}<\rho_{2}.$ Then 
\[
\left|\Lambda_{\rho_{1}}\left(\mu,\nu\right)-\Lambda_{\rho_{2}}\left(\mu,\nu\right)\right|\le\frac{10\left(\rho_{2}-\rho_{1}\right)}{1-\rho_{2}}.
\]
\end{lem}

We would also like to remark that we have the following Fourier formula
for $\mathsf{T}_{\rho}\left[f\right]$, in the case where $f$ is
a multilinear polynomial:
\begin{fact}
\label{Fact: Fourier formula for gaussians}Let $f=\sum_{S\subseteq\left[n\right]}a_{i}\prod_{i\in S}z_{i}$
be a multilinear polynomial. Then 
\[
\mathsf{T}_{\rho}\left[f\right]=\sum_{S\subseteq\left[n\right]}a_{i}\rho^{\left|S\right|}\prod_{i\in S}z_{i}.
\]
\end{fact}

\subsection{The invariance principle}

The invariance principle is a powerful theorem due to Mossel, O'Donnell,
and Oleszkiewicz \cite{mossel2010noise} that relates the distribution
of a `smooth' function $f\colon\left\{ 0,1\right\} ^{n}\to\mathbb{R}$
with the distribution of functions on Gaussian spaces. To state a
corollary of it that we shall apply, we need to introduce some terminology. 

Let $f\colon\mathbb{R}^{n}\to\mathbb{R}$ be a function. Following
\cite{dinur2009conditional}, we define the function $\mathrm{Chop}\left(f\right)$
by setting 
\[
\mathrm{Chop}\left(f\right)\left(x\right)=\begin{cases}
f\left(x\right) & \mbox{if }f\left(x\right)\in\left[0,1\right]\\
0 & \mbox{if }f\left(x\right)\le0\\
1 & \mbox{if }f\left(x\right)\ge1
\end{cases}.
\]

We shall also need the following definition.
\begin{defn}
Let $f\in L^{2}\left(\left\{ 0,1\right\} ^{n},\mu_{p}\right)$ be
some function with Fourier expansion 
\[
f=\sum_{S\subseteq\left[n\right]}\hat{f}\left(S\right)\cps.
\]
We let the  Gaussian analogue of it be the multilinear polynomial
$\tilde{f}\in L^{2}\left(\mathbb{R}^{n},\gamma\right)$ defined by
\[
\tilde{f}\left(z\right)=\sum_{S\subseteq\left[n\right]}\hat{f}\left(S\right)\prod_{i\in S}z_{i}.
\]
\end{defn}

Roughly speaking, the invariance principle says that if the function
$f$ is sufficiently `smooth', then the distribution of $f\left(\mathbf{x}\right)$,
where $\x\sim\left(\left\{ 0,1\right\} ^{n},\mu_{p}\right)$ is somewhat
similar to the distribution of $\tilde{f}\left(\mathbf{y}\right)$,
where $\mathbf{y}\sim\left(\mathbb{R}^{n},\gamma\right)$ is a Gaussian
random variable. The smoothness requirement that we need is the following.
Let $\delta,\epsilon>0$, we shall say that a function $f\in L^{2}\left(\left\{ 0,1\right\} ^{n},\mu_{p}\right)$
is $\left(\delta,1-\epsilon,\mu_{p}\right)$-smooth if $\mathrm{Inf}_{i}^{p}\left[f\right]<\delta$
for each $i\in\left[n\right],$ and $\left|\hat{f}\left(S\right)\right|\le\left(1-\epsilon\right)^{\left|S\right|}$
for each $S\subseteq\left[n\right].$ 

As a corollary of the invariance principle, one can show (see \cite[Theorem 3.8]{dinur2009conditional}
or \cite[Theorem 3.18]{mossel2010noise}) the following corollary
of it. It says that if $f$ is a `sufficiently smooth' function that
takes its value in the interval $\left[0,1\right]$, then $\tilde{f}$
is concentrated on $\left[0,1\right]$ as well, in the sense that
$\|\tilde{f}-\chop\left(\tilde{f}\right)\|$ is small. 
\begin{cor}[Corollary of the invariance principle]
\label{cor:of invariance} For each $\epsilon,\eta>0$, there exists
 $\delta>0$, such that the following holds. Let $p\in\left(\epsilon,1-\epsilon\right)$,
let $f\colon\left\{ 0,1\right\} ^{n}\to\left[0,1\right]$ be a function,
and suppose that $f$ is $\left(\delta,1-\eta,\mu_{p}\right)$-smooth.
Then 
\[
\|\tilde{f}-\chop\left(\tilde{f}\right)\|<\epsilon.
\]
 
\end{cor}

\section{Counting lemma for the $\rho$-noisy influence regularity lemma}

In this section we prove our version of the majority is stablest theorem
that would serve as a counting lemma in the proof of Theorem \ref{thm:general approximation by junta theorem}.
The proof is a straightforward adaptation of the proof by Mossel,
O'Donnell, and Oleszkiewicz \cite{mossel2010noise} of the Majority
is Stablest Theorem, and its generalizations by Mossel \cite{mossel2010gaussian}. 
\begin{prop}
\label{prop:Counting lemma for the approximation by junta thm} For
each $\epsilon>0,$ there exists $\delta>0,$ such that the following
holds. Let $\rho\in\left(0,1-\epsilon\right),$ and suppose that $p-q>\epsilon.$
Let $f,g\colon\left\{ 0,1\right\} ^{n}\to\left[0,1\right]$ be some
functions, and suppose that 
\[
\max_{i\in\left[n\right]}\min\left\{ \mathrm{Inf}_{i}^{\left(1-\delta,q\right)}\left[f\right],\mathrm{Inf}_{i}^{\left(1-\delta,q\right)}\left[g\right]\right\} <\delta.
\]
Then 
\begin{equation}
\sum_{S\subseteq\left[n\right]}\rho^{\left|S\right|}\hat{f}\left(S\right)\hat{g}\left(S\right)<\Lambda_{\rho}\left(\mu_{q}\left(f\right),\mu_{p}\left(g\right)\right)+\epsilon.\label{eq:conclusion of maj is stablest}
\end{equation}
 
\end{prop}

We divide the proof into three parts. In each of these parts we prove
that if $f,g$ satisfy certain requirement then (\ref{eq:conclusion of maj is stablest})
holds. The hypothesis will be the strongest in the first part, weaker
on the second part, and the weakest on the third part. The parts are
as follows.
\begin{enumerate}
\item We start by showing that (\ref{eq:conclusion of maj is stablest})
holds if $f$ is assumed to be $\left(\delta,1-\epsilon,\mu_{q}\right)$-smooth,
and $g$ is assumed to be $\left(\delta,1-\epsilon,\mu_{p}\right)$-smooth. 
\item We then prove (\ref{eq:conclusion of maj is stablest}) in the case
where $f$ and $g$ are assumed to satisfy 
\[
\max_{i\in\left[n\right]}\max\left\{ \mathrm{Inf}_{i}^{\left(1-\delta,q\right)}\left[f\right],\mathrm{Inf}_{i}^{\left(1-\delta,q\right)}\left[g\right]\right\} <\delta.
\]
 
\item Finally, we shall complete the proof of the proposition by proving
(\ref{eq:conclusion of maj is stablest}) in the case where $f$ and
$g$ are assumed to satisfy 
\[
\max_{i\in\left[n\right]}\min\left\{ \mathrm{Inf}_{i}^{\left(1-\delta,q\right)}\left[f\right],\mathrm{Inf}_{i}^{\left(1-\delta,q\right)}\left[g\right]\right\} <\delta.
\]
\end{enumerate}

\subsection{Proof of the proposition in the case where $f$ is $\left(\delta,1-\epsilon,\mu_{q}\right)$-smooth
and $g$ is $\left(\delta,1-\epsilon,\mu_{p}\right)$-smooth.}

The idea of the proof is to convert the statement on $f$ and $g$
to a corresponding statement about their Gaussian analogues $\tilde{f}$
and $\tilde{g}$, and then to prove the corresponding statement by
applying Borell's Theorem. A difficulty that arises in this approach
is the fact that Borell's Theorem may be applied only on functions
that take their values in the interval $\left[0,1\right],$ while
the functions $\tilde{f}$ and $\tilde{g}$ may take their values
outside of this interval. However, we overcome this difficulty by
noting that Borell's theorem may be applied on the functions $\chop\left(\tilde{f}\right)$
and $\chop\left(\tilde{g}\right),$ and by observing that Corollary
\ref{cor:of invariance} shows that $\tilde{f}$ and $\tilde{g}$
are approximated by the functions $\chop\left(\tilde{f}\right)$ and
$\chop\left(\tilde{g}\right).$ The technical details are below. 
\begin{lem}
\label{lem:Counting Lemma 1}For each $\epsilon>0,$ there exists
$\delta>0$ such that the following holds. Let $q,p\in\left(\epsilon,1-\epsilon\right)$,
let $\rho\in\left(0,1\right),$ let $f=\sum\hat{f}\left(S\right)\chi_{S}^{q}$
be a $\left(\delta,1-\epsilon,\mu_{q}\right)$-smooth function, and
let $g=\sum\hat{g}\left(S\right)\chi_{S}^{p}$ be a $\left(\delta,1-\epsilon,\mu_{p}\right)$-smooth
function. Then 
\[
\sum_{S\subseteq\left[n\right]}\rho^{\left|S\right|}\hat{f}\left(S\right)\hat{g}\left(S\right)<\Lambda_{\rho}\left(\mu_{q}\left(f\right),\mu_{p}\left(g\right)\right)+\epsilon.
\]
 
\end{lem}

\begin{proof}
Let $\epsilon>0$ and suppose that $\delta=\delta\left(\epsilon\right)$
is sufficiently small. Let $\tilde{f}$ be the Gaussian analogue of
$f$ and let $\tilde{g}$ be the Gaussian analogue of $g.$ By Fact
\ref{Fact: Fourier formula for gaussians} we have 
\[
\sum_{S\subseteq\left[n\right]}\rho^{\left|S\right|}\hat{f}\left(S\right)\hat{g}\left(S\right)=\left\langle \t_{\rho}\tilde{f},\tilde{g}\right\rangle .
\]
So our goal is to show that 
\begin{equation}
\left\langle \t_{\rho}\tilde{f},\tilde{g}\right\rangle -\Lambda_{\rho}\left(\mu_{q}\left(f\right),\mu_{p}\left(g\right)\right)<\epsilon,\label{eq:Goal 1}
\end{equation}
 provided that $\delta$ is sufficiently small. Let 
\[
\epsilon_{1}:=\left|\left\langle \t_{\rho}\tilde{f},\tilde{g}\right\rangle -\left\langle \t_{\rho}\left(\mathrm{Chop}\left(\tilde{f}\right)\right),\mathrm{Chop}\left(\tilde{g}\right)\right\rangle \right|,
\]
 and let 
\begin{align*}
\epsilon_{2} & =\left|\Lambda_{\rho}\left(\mathbb{E}\left(\mathrm{Chop}\left(\tilde{f}\right)\right),\mathbb{E}\left(\mathrm{Chop}\left(\tilde{g}\right)\right)\right)-\Lambda_{\rho}\left(\mathbb{E}\left[\tilde{f}\right],\mathbb{E}\left[\tilde{g}\right]\right)\right|\\
 & =\left|\Lambda_{\rho}\left(\mathbb{E}\left(\mathrm{Chop}\left(\tilde{f}\right)\right),\mathbb{E}\left(\mathrm{Chop}\left(\tilde{g}\right)\right)\right)-\Lambda_{\rho}\left(\mu_{q}\left(f\right),\mu_{p}\left(g\right)\right)\right|.
\end{align*}
 Applying Borell\textquoteright s Theorem to the functions $\mathrm{Chop}\left(\tilde{f}\right),\mathrm{Chop}\left(\tilde{g}\right),$
we obtain 
\begin{equation}
\left\langle \t_{\rho}\mathrm{Chop}\left(\tilde{f}\right),\mathrm{Chop}\left(\tilde{g}\right)\right\rangle \le\Lambda_{\rho}\left(\mathbb{E}\left(\mathrm{Chop}\left(\tilde{f}\right)\right),\mathbb{E}\left(\mathrm{Chop}\left(\tilde{g}\right)\right)\right),\label{eq:Borell}
\end{equation}
 and hence 
\[
\left\langle \t_{\rho}\tilde{f},\tilde{g}\right\rangle \le\Lambda_{\rho}\left(\mu_{q}\left(f\right),\mu_{p}\left(g\right)\right)+\epsilon_{1}+\epsilon_{2}.
\]
 So to complete the proof we need to show that $\epsilon_{1}+\epsilon_{2}<\epsilon$
provided that $\delta$ is sufficiently small. 
\begin{claim}
Provided that $\delta$ is sufficiently small we have $\epsilon_{1}<\epsilon/2$.
\end{claim}

\begin{proof}
Note that it follows from Jensen's inequality that the operator $\t_{\rho}$
on the space $L^{2}\left(\mathbb{R}^{n},\gamma\right)$ is a contraction.
Indeed, for each function $h\in L^{2}\left(\mathbb{R}^{n},\gamma\right)$
we have 
\begin{align*}
\left\Vert \t_{\rho}\left(h\right)\right\Vert ^{2} & =\e_{\boldsymbol{x}\sim\left(\mathbb{R}^{n},\gamma\right)}\left[\left(\e_{\boldsymbol{y}\sim N_{\rho}\left(x\right)}\left[h\left(\boldsymbol{y}\right)\right]\right)^{2}\right]\\
 & \le\left(\e_{\boldsymbol{x}\sim\left(\mathbb{R}^{n},\gamma\right)}\left[\e_{\boldsymbol{y}\sim N_{\rho}\left(x\right)}\left[h\left(\boldsymbol{y}\right)\right]\right]\right)^{2}\\
 & =\left(\e_{\boldsymbol{y}\sim\left(\mathbb{R}^{n},\gamma\right)}\left[h\left(\boldsymbol{y}\right)\right]\right)^{2}\\
 & =\|h\|^{2}.
\end{align*}
 Moreover, we note that by Parseval $\|\tilde{g}\|=\sum\hat{g}\left(S\right)^{2}=\sqrt{\e_{\mathbf{y\sim\mu}_{p}}\left[g\left(\boldsymbol{y}\right)^{2}\right]}\le1.$
Therefore, 
\begin{align*}
\epsilon_{1} & =\left|\left\langle \t_{\rho}\tilde{f},\tilde{g}\right\rangle -\left\langle \t_{\rho}\left(\mathrm{Chop}\left(\tilde{f}\right)\right),\mathrm{Chop}\left(\tilde{g}\right)\right\rangle \right|\\
 & \le\left|\left\langle \t_{\rho}\tilde{f},\tilde{g}\right\rangle -\left\langle \t_{\rho}\left(\mathrm{Chop}\left(\tilde{f}\right)\right),\tilde{g}\right\rangle \right|\\
 & +\left|\left\langle \t_{\rho}\left(\mathrm{Chop}\left(\tilde{f}\right)\right),\tilde{g}\right\rangle -\left\langle \t_{\rho}\left(\mathrm{Chop}\left(\tilde{f}\right)\right),\mathrm{Chop}\left(\tilde{g}\right)\right\rangle \right|\\
(\text{By Cauchy-Schwarz}) & \le\left\Vert \t_{\rho}\left(\tilde{f}-\mathrm{Chop}\left(\tilde{f}\right)\right)\right\Vert \left\Vert \tilde{g}\right\Vert +\left\Vert \t_{\rho}\left(\mathrm{Chop}\left(\tilde{f}\right)\right)\right\Vert \left\Vert \tilde{g}-\mathrm{Chop}\left(\tilde{g}\right)\right\Vert \\
(\text{Since }\t_{\rho}\text{ is a contraction)} & \le\left\Vert \tilde{f}-\mathrm{Chop}\left(\tilde{f}\right)\right\Vert \left\Vert \tilde{g}\right\Vert +\left\Vert \mathrm{Chop}\left(\tilde{f}\right)\right\Vert \left\Vert \tilde{g}-\mathrm{Chop}\left(\tilde{g}\right)\right\Vert \\
 & \le\left\Vert \tilde{f}-\mathrm{Chop}\left(\tilde{f}\right)\right\Vert +\left\Vert \tilde{g}-\mathrm{Chop}\left(\tilde{g}\right)\right\Vert .
\end{align*}
We may now apply Corollary \ref{cor:of invariance} with $\epsilon$
replacing $\eta$ and $\frac{\epsilon}{4}$ replacing $\epsilon$,
to obtain that 
\begin{equation}
\left\Vert \tilde{f}-\mathrm{Chop}\left(\tilde{f}\right)\right\Vert +\left\Vert \tilde{g}-\mathrm{Chop}\left(\tilde{g}\right)\right\Vert <\frac{\epsilon}{2},\label{eq:chops are small}
\end{equation}
 provided that $\delta$ is sufficiently small. This completes the
proof of the claim. 
\end{proof}
To finish the proof of the lemma it remains to prove the following
claim.
\begin{claim}
Provided that $\delta$ is sufficiently small, we have $\epsilon_{2}<\epsilon/2.$

Choose $X\sim\left(\mathbb{R},\gamma\right),Y\sim N_{\rho}\left(X\right).$
Then $\Lambda_{\rho}\left(\mathbb{E}\left[\tilde{f}\right],\mathbb{E}\left[\tilde{g}\right]\right)$
is the probability of the event $X<t_{1},Y<t_{2}$ for the proper
values of $t_{1},t_{2}.$ Similarly, $\Lambda_{\rho}\left(\mathbb{E}\left[\chop\left(\tilde{f}\right)\right],\mathbb{E}\left[\chop\left(\tilde{g}\right)\right]\right)$
is the probability of the event $X<t_{3},Y<t_{4}$ for the proper
values of $t_{3},t_{4}.$ These events differ either if $X$ is in
the interval whose endpoints are $t_{1},t_{3}$ or if $Y$ is in the
interval whose endpoints are $t_{2},t_{4}.$ The Probability of the
former event is $\left|\e\left[\chop\left(\tilde{f}\right)\right]-\e\left[\tilde{f}\right]\right|$,
and the probability of the latter event is $|\e\left[\chop\left(\tilde{g}\right)\right]-\e\left[\tilde{g}\right]|.$
Therefore, a union bound implies that:
\begin{align*}
\epsilon_{2} & =\left|\Lambda_{\rho}\left(\mathbb{E}\left[\tilde{f}\right],\mathbb{E}\left[\tilde{g}\right]\right)-\Lambda_{\rho}\left(\mathbb{E}\left[\chop\left(\tilde{f}\right)\right],\mathbb{E}\left[\chop\left(\tilde{g}\right)\right]\right)\right|\\
 & \le\left|\e\left[\chop\left(\tilde{f}\right)\right]-\e\left[\tilde{f}\right]\right|+|\e\left[\chop\left(\tilde{g}\right)\right]-\e\left[\tilde{g}\right]|\\
(\text{By Cauchy-Schwarz and \eqref{eq:chops are small}}) & \le\|\chop\left(\tilde{f}\right)-\tilde{f}\|+\|\chop\left(\tilde{g}\right)-\tilde{g}\|<\frac{\epsilon}{2}.
\end{align*}
 
\end{claim}

\end{proof}

\subsection{The case where $f,g$ have small noisy influences}

We shall now prove a stronger version of Lemma \ref{lem:Counting Lemma 1},
where we impose on $f,g$ the hypothesis 
\[
\max_{i\in\left[n\right]}\max\left\{ \mathrm{Inf}_{i}^{\left(1-\delta,q\right)}\left[f\right],\mathrm{Inf}_{i}^{\left(1-\delta,p\right)}\left[g\right]\right\} <\delta.
\]
 
\begin{lem}
\label{lem:Counting lemma 2}For each $\epsilon>0,$ there exists
$\delta>0$ such that the following holds. Let $q,p\in\left(\epsilon,1-\epsilon\right)$,
let $\rho\in\left(0,1\right)$, let $f=\sum\hat{f}\left(S\right)\chi_{S}^{q}$
be a function and suppose that 
\[
\max_{i\in\left[n\right]}\max\left\{ \mathrm{Inf}_{i}^{\left(1-\delta,q\right)}\left[f\right],\mathrm{Inf}_{i}^{\left(1-\delta,p\right)}\left[g\right]\right\} <\delta.
\]
 Then 
\[
\sum_{S\subseteq\left[n\right]}\rho^{\left|S\right|}\hat{f}\left(S\right)\hat{g}\left(S\right)<\Lambda_{\rho}\left(\mu_{q}\left(f\right),\mu_{p}\left(g\right)\right)+\epsilon.
\]
 
\end{lem}

\begin{proof}
Let $\epsilon>0$, let $\delta_{1}=\delta_{1}\left(\epsilon\right)$
be sufficiently small, and let $\delta=\delta\left(\delta_{1}\right)$
be sufficiently small. Let $f'=\t_{q,1-\delta_{1}}\left(f\right),g'=\t_{p,1-\delta_{1}}\left(g\right),$
$\rho'=\frac{\rho}{\left(1-\delta_{1}\right)^{2}}.$ 

We assert that the functions $f'$ is $\left(\delta,1-\delta_{1},\mu_{q}\right)$-smooth
and the function $g'$ is $\left(\delta,1-\delta_{1},\mu_{p}\right)$-smooth,
provided that $\delta$ is small enough. Indeed, the functions $f'$
is $\left(\delta,1-\delta_{1},\mu_{q}\right)$-smooth since: 
\[
\mathrm{Inf}_{i}\left[f'\right]=\mathrm{Inf}_{i}^{\left(1-\delta_{1},q\right)}\left[f\right]\le\mathrm{Inf}_{i}^{\left(1-\delta,q\right)}\left[f\right]<\delta,
\]
provided that $\delta\le\delta_{1},$ and 
\[
\left|\hat{f'}\left(S\right)\right|=\left|\left(1-\delta_{1}\right)^{\left|S\right|}\hat{f}\left(S\right)\right|\le\left(1-\delta_{1}\right)^{\left|S\right|}.
\]
 The function $g$ is $\left(\delta,1-\delta_{1},\mu_{p}\right)$-smooth
for similar reasons. Provided that $\delta$ is small enough, Lemma
\ref{lem:Counting Lemma 1} implies that 
\[
\left\langle \t_{\rho}f,g\right\rangle =\left\langle \t_{\rho'}f',g'\right\rangle \le\Lambda_{\rho'}\left(\mu_{q}\left(f'\right),\mu_{p}\left(g'\right)\right)+\delta_{1}=\Lambda_{\rho'}\left(\mu_{q}\left(f\right),\mu_{p}\left(g\right)\right)+\delta_{1}.
\]
By Lemma \ref{lem:Estimossel} we have 
\[
\Lambda_{\rho'}\left(\mu_{q}\left(f\right),\mu_{p}\left(g\right)\right)<\Lambda_{\rho}\left(\mu_{q}\left(f\right),\mu_{p}\left(g\right)\right)+\epsilon-\delta_{1},
\]
 provided that $\delta_{1}$ is sufficiently small. Hence 
\[
\left\langle \t_{\rho}f,g\right\rangle \le\Lambda_{\rho}\left(\mu_{q}\left(f\right),\mu_{p}\left(g\right)\right)+\epsilon.
\]
\end{proof}

\subsection{Proof of Proposition \ref{prop:Counting lemma for the approximation by junta thm}}

Finally, we shall replace the hypothesis 
\[
\max_{i\in\left[n\right]}\max\left\{ \mathrm{Inf}_{i}^{\left(1-\delta,q\right)}\left[f\right],\mathrm{Inf}_{i}^{\left(1-\delta,p\right)}\left[g\right]\right\} <\delta
\]
 by the weaker hypothesis 
\[
\max_{i\in\left[n\right]}\min\left\{ \mathrm{Inf}_{i}^{\left(1-\delta,q\right)}\left[f\right],\mathrm{Inf}_{i}^{\left(1-\delta,p\right)}\left[g\right]\right\} <\delta.
\]
 
\begin{proof}
Let $\epsilon>0$, let $\delta_{1}=\delta_{1}\left(\epsilon\right)$
be sufficiently small, and let $\delta=\delta\left(\delta_{1}\right)$
be sufficiently small. Let 
\[
A_{1}=\left\{ i\in\left[n\right]:\,\mathrm{Inf}_{i}^{\left(1-\delta_{1},q\right)}\left[f\right]>\delta_{1}\right\} ,\,\,\,A_{2}=\left\{ i\in\left[n\right]:\,\mathrm{Inf}_{i}^{\left(1-\delta_{1},q\right)}\left[g\right]>\delta_{1}\right\} ,
\]
 let $A=A_{1}\cup A_{2},$ and set $B=\left[n\right]\backslash A.$
Write $f'=\av_{A}\left(f\right),g'=\av_{A}\left(g\right).$ We have
\[
\sum_{S\subseteq\left[n\right]}\rho^{\left|S\right|}\hat{f}\left(S\right)\hat{g}\left(S\right)=\sum_{S\subseteq B}\rho^{\left|S\right|}\hat{f}\left(S\right)\hat{g}\left(S\right)+\sum_{S\cap A\ne\varnothing}\rho^{\left|S\right|}\hat{f}\left(S\right)\hat{g}\left(S\right).
\]
 We shall now bound $\sum_{S\subseteq\left[n\right]}\rho^{\left|S\right|}\hat{f}\left(S\right)\hat{g}\left(S\right)$
by bounding each of the terms in the right hand side. 

\textbf{Upper bounding $\sum_{S\subseteq B}\rho^{\left|S\right|}\hat{f}\left(S\right)\hat{g}\left(S\right)$}

Since $f',g'$ satisfy the hypothesis of Lemma \ref{lem:Counting lemma 2}
(with $\delta_{1}$ replacing $\delta$), we have 
\[
\sum_{S\subseteq B}\rho^{\left|S\right|}\hat{f}\left(S\right)\hat{g}\left(S\right)=\sum_{S\subseteq\left[n\right]}\rho^{\left|S\right|}\hat{f}'\left(S\right)\hat{g'}\left(S\right)\le\Lambda_{\rho}\left(\mu_{q}\left(f'\right),\mu_{p}\left(g'\right)\right)+\frac{\epsilon}{2}=\Lambda_{\rho}\left(\mu_{q}\left(f\right),\mu_{p}\left(g\right)\right)+\frac{\epsilon}{2},
\]
 provided that $\delta_{1}$ is small enough.

\textbf{Upper bounding} $\sum_{S\cap A\ne\varnothing}\rho^{\left|S\right|}\hat{f}\left(S\right)\hat{g}\left(S\right).$

By Cauchy Schwarz, we have 
\begin{align*}
\sum_{S\cap A\ne\varnothing}\rho^{\left|S\right|}\hat{f}\left(S\right)\hat{g}\left(S\right) & \le\sum_{i\in A}\sum_{S\ni i}\rho^{\left|S\right|}\left|\hat{f}\left(S\right)\hat{g}\left(S\right)\right|\le\sum_{i\in A}\sqrt{\sum_{S\ni i}\rho^{\left|S\right|}\hat{f}\left(S\right)^{2}}\sqrt{\sum_{S\ni i}\rho^{\left|S\right|}\hat{g}\left(S\right)^{2}}.\\
 & =\sum_{i\in A}\sqrt{\mathrm{Inf}_{i}^{\left(\rho,q\right)}\left(f\right)}\sqrt{\mathrm{Inf}_{i}^{\left(\rho,p\right)}\left(g\right)}.
\end{align*}
 Now note that 
\[
\max\left\{ \mathrm{Inf}_{i}^{\left(\rho,q\right)}\left(f\right),\mathrm{Inf}_{i}^{\left(\rho,p\right)}\left(g\right)\right\} \le1
\]
 for any $i\in\left[n\right].$ Moreover, we have $\mathrm{Inf}_{i}^{\left(\rho,q\right)}\left(f\right)\le\mathrm{Inf}_{i}^{\left(1-\delta,q\right)}\left(f\right)$
and $\mathrm{Inf}_{i}^{\left(\rho,p\right)}\left(f\right)\le\mathrm{Inf}_{i}^{\left(1-\delta,p\right)}\left(f\right),$
provided that $\delta<1-\rho.$ The hypothesis implies that 
\[
\max_{i\in\left[n\right]}\min\left\{ \mathrm{Inf}_{i}^{\left(1-\delta,q\right)}\left(f\right),\mathrm{Inf}_{i}^{\left(1-\delta,p\right)}\left(f\right)\right\} \le\delta.
\]
 Therefore, 
\[
\sum_{i\in A}\sqrt{\mathrm{Inf}_{i}^{\left(\rho,q\right)}\left(f\right)}\sqrt{\mathrm{Inf}_{i}^{\left(\rho,p\right)}\left(g\right)}\le\left|A\right|\sqrt{\delta}.
\]
 So this completes the proof provided that $\delta\le\frac{\epsilon^{2}}{4\left|A\right|^{2}}$.
We shall now complete the proof by showing that $\left|A\right|=O_{\delta_{1}}\left(1\right).$

\textbf{Upper bounding $\left|A\right|$} 

We show that $\left|A_{1}\right|=O_{\delta_{1}}\left(1\right),$ as
the proof that $\left|A_{2}\right|=O_{\delta_{1}}\left(1\right)$
is similar. Note that the quantity $\sum_{i=1}^{n}\mathrm{Inf}_{i}^{\left(1-\delta_{1},q\right)}\left(f\right)$
is on the one hand bounded from below by $\left|A_{1}\right|\delta_{1},$
and on the other hand we have the following upper bound on it. 
\[
\sum_{i=1}^{n}\mathrm{Inf}_{i}^{\left(q,1-\delta_{1}\right)}\left(f\right)=\sum_{S\subseteq\left[n\right]}\left(1-\delta_{1}\right)^{\left|S\right|}\left|S\right|\hat{f}\left(S\right)^{2}\le\sum_{s=1}^{\infty}s\left(1-\delta_{1}\right)^{s}=O_{\delta_{1}}\left(1\right).
\]
 Hence $\left|A_{1}\right|=O_{\delta_{1}}\left(1\right).$ This completes
the proof of the proposition.
\end{proof}

\section{Proof of the structural result on almost monotone functions}

In this section we prove Theorem \ref{thm:Monotone approximation}.
We restate it for the convenience of the reader.
\begin{thm*}
For each $\epsilon>0$, there exists $j\in\mathbb{N},\delta>0$, such
that the following holds. Let $p,q$ be  numbers in the interval $\left(\epsilon,1-\epsilon\right)$
that satisfy $p-q>\epsilon$ and let $f\colon\left\{ 0,1\right\} ^{n}\to\left\{ 0,1\right\} $
be a $\left(q,p,\delta\right)$-almost monotone function. Then there
exists a monotone $j$-junta $g$, such that 
\[
\Pr_{\x\sim\mu_{q}}\left[f\left(\x\right)>g\left(\x\right)\right]<\epsilon\text{ and }\Pr_{\x\sim\mu_{p}}\left[f\left(\x\right)<g\left(\x\right)\right]<\epsilon.
\]
\end{thm*}
We recall that the proof relies on the regularity method, with the
regularity lemma being Theorem \ref{thm:Jones regularity} of \cite{jones2016noisy},
and with the corresponding counting lemma being Proposition \ref{prop:Counting lemma for the approximation by junta thm}. 

The regularity lemma allows us to decompose $f$ into functions $\left\{ f_{J\to\x}\right\} _{\x\in\left\{ 0,1\right\} ^{J}}$,
such that for most of the parts the function $f_{J\to x}$ has small
noisy influences and a $q$-biased measure that is bounded away from
$0$. We shall then approximate $f$ by the `least' monotone junta
$g\colon\left\{ 0,1\right\} ^{J}\to\left\{ 0,1\right\} $ that takes
the value 1 on all the the $x$, such that the function $f_{J\to x}$
has small noisy influences. Here, by least we mean smallest with respect
to the partial order: $g\le h$ if and only if $g\left(x\right)\le h\left(x\right)$
for each $x$.
\begin{proof}
Let $\delta_{1}=\delta_{1}\left(\epsilon\right)$ be sufficiently
small, let $\delta_{2}=\delta_{2}\left(\delta_{1}\right)$ be sufficiently
small, let $j=j\left(\delta_{2}\right)$ be sufficiently large, and
let $\delta=\delta\left(j,\delta_{1},\epsilon\right)$ be sufficiently
small. By Theorem \ref{thm:Jones regularity}, there exists a set
$J$ of size at most $j$, such that the for a random $\mathbf{x}\sim\left(\left\{ 0,1\right\} ^{J},\mu_{q}\right)$
the function $f_{J\to\mathbf{x}}$ does not have $\left(1-\delta_{2},\delta_{2},\mu_{q}\right)$-small
noisy influences with probability at most $\delta_{2}$.

Let $Q\subseteq\left\{ 0,1\right\} ^{J}$ be the set of `quasirandom
parts' consisting of all $x\in\left\{ 0,1\right\} ^{J},$ such that
$f_{J\to x}$ has $\left(1-\delta_{2},\delta_{2},\mu_{q}\right)$-small
noisy influences. So $\Pr_{\boldsymbol{x}\sim\left\{ 0,1\right\} ^{J}}\left[\x\in Q\right]>1-\delta_{2}.$ 

Let $N\subseteq\left\{ 0,1\right\} ^{J}$ be the set of `negligible
parts' consisting of all $x\in\left\{ 0,1\right\} ^{J},$ such that
$\mu_{q}\left(f_{J\to x}\right)<\epsilon/2.$ Note that 
\begin{align*}
\Pr_{\x\sim\mu_{q}}\left[f\left(\x\right)=1\,|\,\x_{J}\in N\right] & \le\max_{y\in N}\Pr_{\mathbf{x}\sim\mu_{q}}\left[f\left(\mathbf{x}\right)=1\,|\,\mathbf{x}_{J}=y\right]\le\frac{\epsilon}{2}.
\end{align*}

Let $A$ be the up-closure of $Q/N$, i.e. the set of all $x\in\left\{ 0,1\right\} ^{J},$
such that there exists some $y\in Q\backslash N$ that satisfies $\forall i:\,y_{i}\le x_{i}$.
Finally, we let $g\colon\left\{ 0,1\right\} ^{J}\to\left\{ 0,1\right\} $
be the indicator function of $A.$ 

\textbf{Showing that $\Pr_{\x\sim\mu_{q}}\left[f\left(\x\right)>g\left(\x\right)\right]<\epsilon$.}

For each $\mathbf{x}\in\left\{ 0,1\right\} ^{n}$ with $f\left(\x\right)>g\left(\x\right)$
we have $g\left(x\right)=0$, and particularly $x\notin Q\backslash N.$
So we either have $\mathbf{x}\notin Q$ or we have the unlikely event
that $f\left(\mathbf{x}\right)=1$ although $\mathbf{x}_{J}\in N.$
The former event occurs with probability at most $\delta_{2},$ and
the latter event occurs with probability at most $\frac{\epsilon}{2}$
so 
\[
\Pr_{\x\sim\mu_{q}}\left[f\left(\x\right)>g\left(\x\right)\right]<\frac{\epsilon}{2}+\delta_{2}<\epsilon,
\]
provided that $\delta_{2}$ is sufficiently small. 

\textbf{Showing that $\Pr_{\x\sim\mu_{p}}\left[f\left(\x\right)<g\left(\x\right)\right]<\epsilon$.}

Let $y\in A,$ let $x\in Q\backslash N$ be with $\forall i:\,x_{i}\le y_{i}$,
and let $\rho=\sqrt{\frac{q\left(1-p\right)}{p\left(1-q\right)}}$.
Since $x$ is in $Q$, we may apply Proposition \ref{prop:Counting lemma for the approximation by junta thm}
to obtain that 
\begin{equation}
\left\langle \t^{q\to p}f_{J\to x},f_{J\to y}\right\rangle \le\Lambda_{\rho}\left(\mu_{q}\left(f_{J\to x}\right),\mu_{p}\left(f_{J\to y}\right)\right)+\delta_{1},\label{eq:sec4 1}
\end{equation}
 provided that $\delta_{2}$ is sufficiently small. 

This gives us an upper bound on $\left\langle \t^{q\to p}f_{J\to x},f_{J\to y}\right\rangle .$
On the other hand we may use the fact that $f$ is almost monotone
to obtain a lower bound on $\left\langle \t^{q\to p}f_{J\to x},f_{J\to y}\right\rangle $
as follows. Note that we have 
\begin{align}
\delta & \ge\left\langle \t^{q\to p}f,1-f\right\rangle =\Pr_{\mathbf{z,}\mathbf{w}\sim D\left(q,p\right)}\left[f\left(\mathbf{z}\right)=1,f\left(\mathbf{w}\right)=0\right]\nonumber \\
 & \ge\Pr_{\mathbf{z,}\mathbf{w}\sim D\left(q,p\right)}\left[\mathbf{z}_{J}=x,\mathbf{w}_{J}=y\right]\Pr_{\mathbf{z,}\mathbf{w}\sim D\left(q,p\right)}\left[f_{J\to x}\left(\mathbf{z}_{\left[n\right]\backslash J}\right)=1,f_{J\to y}\left(\mathbf{w}_{\left[n\right]\backslash J}\right)=0\right]\label{eq:Theorem 4.1}\\
 & =\Pr_{\mathbf{x},\mathbf{y}\sim\left(\left\{ 0,1\right\} ^{J},D\left(q,p\right)\right)}\left[\mathbf{x}=x,\mathbf{y}=y\right]\left\langle \t^{q\to p}f_{J\to x},1-f_{J\to y}\right\rangle .\nonumber 
\end{align}
Thus, 
\begin{align}
\left\langle \t^{q\to p}f_{J\to x},f_{J\to y}\right\rangle  & =\left\langle \t^{q\to p}f_{J\to x},1\right\rangle -\left\langle \t^{q\to p}f_{J\to x},1-f_{J\to y}\right\rangle \nonumber \\
 & \ge\mu_{q}\left(f_{J\to x}\right)-\frac{\delta}{\Pr_{\mathbf{x},\mathbf{y}\sim\left(\left\{ 0,1\right\} ^{J},D\left(q,p\right)\right)}\left[\mathbf{x}=x,\mathbf{y}=y\right]}\label{eq: sec 4 2}\\
 & \ge\mu_{q}\left(f_{J\to x}\right)-\delta_{1},\nonumber 
\end{align}
 provided that $\delta=\delta\left(\delta_{1},j,\epsilon\right)$
is small enough. Combining (\ref{eq:sec4 1}) and (\ref{eq: sec 4 2})
we obtain 
\[
\Lambda_{\rho}\left(\mu_{q}\left(f_{J\to x}\right),\mu_{p}\left(f_{J\to y}\right)\right)\ge\mu_{q}\left(f_{J\to x}\right)-2\delta_{1}.
\]

By Lemma \ref{lem:Estimate from Borell's Thm} we have $\mu_{p}\left(f_{J\to y}\right)>1-\frac{\epsilon}{2}$
provided that $\delta_{1}$ is small enough (note that $\mu_{q}\left(f_{J\to x}\right)>\epsilon/2,$
since $x\notin N$). 

This shows that any $y$ with $g\left(y\right)=1,f\left(y\right)=0$
satisfies the unlikely event that $f\left(y\right)=0$ while $\mu_{p}\left(f_{J\to y_{J}}\right)>1-\epsilon/2.$
Since a random $\mathbf{y}\sim\mu_{p}$ satisfies this event with
probability at most $\epsilon$, we obtain $\Pr_{\boldsymbol{y}\sim\mu_{p}}\left[f\left(\boldsymbol{y}\right)<g\left(\boldsymbol{y}\right)\right]<\epsilon$.
This completes the proof of the theorem.
\end{proof}
We may repeat the proof of Theorem \ref{thm:Monotone approximation}
to obtain the following lemma that we use in the proof of Theorem
\ref{thm:Robust version of Friedgut-Kalai Theorem}.
\begin{lem}
\label{lem:Section 4} For each $\epsilon>0,j\in\mathbb{N}$ there
exists $\delta>0,$ such that the following holds. Let $f,g\colon\left\{ 0,1\right\} ^{n}\to\left[0,1\right]$
be functions, let $J$ be a set of size at most $j,$ and let $p,q\in\left(\epsilon,1-\epsilon\right)$,
be with $p-q>\epsilon.$ Suppose that $\left\langle \t_{q\to p}f,1-g\right\rangle <\delta,$
and let $x,y\in\left\{ 0,1\right\} ^{J}$ be with $\forall i:\,x_{i}\le y_{i}.$
Suppose additionally that $f_{J\to x}$ has $\left(1-\epsilon,\epsilon,\mu_{q}\right)$-small
noisy influences. Then we either have $\mu_{q}\left(f_{J\to x}\right)<\epsilon$
or we have $\mu_{p}\left(g_{J\to y}\right)>1-\epsilon.$
\end{lem}

\begin{proof}
Let $\delta_{1}=\delta_{1}\left(\epsilon\right)$ be sufficiently
small, and let $\delta=\delta\left(\delta_{1},j\right)$ be sufficiently
small. Similarly to (\ref{eq:Theorem 4.1}) we have 
\begin{align*}
\delta & \ge\left\langle \t^{q\to p}f,1-g\right\rangle \ge\left\langle \t^{q\to p}f_{J\to x},1-g_{J\to y}\right\rangle \Pr_{\mathbf{x},\mathbf{y}\sim\left(\left\{ 0,1\right\} ^{J},D\left(q,p\right)\right)}\left[\mathbf{x}=x,\mathbf{y}=y\right].
\end{align*}
 Similarly to (\ref{eq: sec 4 2}) we have 
\[
\left\langle \t^{q\to p}f_{J\to x},g_{J\to y}\right\rangle \ge\mu_{q}\left(f_{J\to x}\right)-\delta_{1},
\]
provided that $\delta_{1}$ is small enough. Similarly to (\ref{eq:sec4 1}),
we have 
\[
\left\langle \t^{q\to p}f_{J\to x},g_{J\to y}\right\rangle \le\Lambda_{\rho}\left(\mu_{q}\left(f_{J\to x}\right),\mu_{p}\left(g_{J\to y}\right)\right)+\delta_{1}.
\]
Hence, 
\[
\Lambda_{\rho}\left(\mu_{q}\left(f_{J\to x}\right),\mu_{p}\left(g_{J\to y}\right)\right)\ge\mu_{q}\left(f_{J\to x}\right)-2\delta_{1}.
\]
 As in the proof of Theorem (\ref{thm:Monotone approximation}), we
may now apply Lemma \ref{lem:Estimate from Borell's Thm} to complete
the proof.
\end{proof}
We shall now prove Theorem \ref{thm:Robust version of Friedgut-Kalai Theorem}.
We restate it for the convenience of the reader.
\begin{thm*}
For each $\epsilon>0$, there exists $\delta>0$, such that the following
holds. Let $q,p\in\left(\epsilon,1-\epsilon\right)$ and suppose that
$p>q+\epsilon$. Let $f,g\colon\left\{ 0,1\right\} ^{n}\to\left[0,1\right]$,
and suppose that 
\[
\e_{\x,\y\sim D\left(q,p\right)}\left[\left(1-g\left(\y\right)\right)f\left(\x\right)\right]<\delta,
\]
 and that the function $f$ is $\left(\left\lceil \frac{1}{\delta}\right\rceil ,\delta,\mu_{q}\right)$-regular.
Then either $\mu_{q}\left(f\right)<\epsilon$, or $\mu_{p}\left(g\right)>1-\epsilon.$
\end{thm*}
\begin{proof}
Let $\delta_{1}=\delta_{1}\left(\epsilon\right)$ be sufficiently
small, let $j=j\left(\delta_{1},\epsilon\right)$ be sufficiently
large, and let $\delta=\delta\left(j,\delta_{1},\epsilon\right)$
be sufficiently small. By Theorem \ref{thm:Jones regularity}, there
exists a set $J$ of size at most $j$, such that for a random $\mathbf{x}\sim\left(\left\{ 0,1\right\} ^{J},\mu_{q}\right)$
the function $f_{J\to x}$ does not  have $\left(1-\delta_{1},\delta_{1},\mu_{q}\right)$-small
noisy influences with probability at most $\delta_{1}$. Let $Q\subseteq\left\{ 0,1\right\} ^{J}$
be the set of `quasirandom elements' consisting of all $x\in\left\{ 0,1\right\} ^{J},$
such that $f_{J\to x}$ has $\left(1-\delta_{1},\delta_{1},\mu_{q}\right)$-small
noisy influences. Let $A$ be the up-closure of $Q$. Since $A$ is
monotone, we have 
\[
\mu_{p}\left(A\right)\ge\mu_{q}\left(A\right)\ge1-\delta_{2}.
\]
 Moreover, the fact that $f$ is $\left(\left\lceil \frac{1}{\delta}\right\rceil ,\delta,\mu_{q}\right)$-regular
implies that 
\[
\mu_{q}\left(f_{J\to x}\right)\ge\epsilon-\delta>\epsilon/2
\]
 for each $x\in\left\{ 0,1\right\} ^{J},$ provided that $\delta<\min\left\{ \frac{\epsilon}{2},\frac{1}{j}\right\} .$
By Lemma \ref{lem:Section 4} (applied with $\epsilon/2$ rather than
$\epsilon$), we obtain that $\mu_{p}\left(g_{J\to x}\right)>1-\epsilon/2$
for all $x\in A.$ So this implies that 
\[
\mu_{p}\left(g\right)\ge\left(1-\epsilon/2\right)\mu_{p}\left(A\right)\ge\left(1-\epsilon/2\right)\left(1-\delta/2\right)>1-\epsilon.
\]
 This completes the proof of the theorem. 
\end{proof}

\section{\label{sec:Counting matchings} Counting matchings }

In this section we prove Theorem \ref{thm:Counting expanded hypergraphs}
in the case where $\h$ is a matching. 
\begin{thm}
\label{thm:Counting matchings} For each $h\in\mathbb{N},\epsilon>0,$
there exists $\delta>0$, such that the following holds. Let $k_{1},\ldots,k_{h}\le\left(\frac{1}{s}-\epsilon\right)n$,
and let $\f_{1}\subseteq\binom{\left[n\right]}{k_{1}},\ldots,\f_{h}\subseteq\binom{\left[n\right]}{k_{h}}$
be families whose measure is at least $\epsilon$. Suppose that for
each $i\in\left[n\right]$, such that $k_{i}\ge\delta n$ the family
$\f_{i}$ is $\left(\left\lceil \frac{1}{\delta}\right\rceil ,\delta\right)$-regular,
and choose uniformly at random a matching $\left\{ \boldsymbol{A}_{1},\ldots,\boldsymbol{A}_{h}\right\} ,$
such that 
\[
\mathbf{A}_{1}\in\binom{\left[n\right]}{k_{1}},\ldots,\mathbf{A}_{h}\in\binom{\left[n\right]}{k_{h}}.
\]
 Then 
\[
\Pr\left[\mathbf{A}_{1}\in\f_{1},\ldots,\mathbf{A}_{h}\in\f_{h}\right]>\delta.
\]
\end{thm}

We start by stating some constructions that we shall use throughout
the proof.

\subsection{Basic constructions and overview of the proof}

We identify an element $x\in\left\{ 0,1\right\} ^{n}$
with the set of $i\in\left[n\right],$ such that $x_{i}=1$. Thus,
we shall use the notations $\mathcal{F}_{J}^{B},$ $\mathcal{F}_{J}^{1_{B}}$
interchangeably, we write $\mathcal{P}\left(x\right)$ for the family
of all subsets of $\left\{ i:\,x_{i}=1\right\} ,$ we write
$\binom{x}{k}$ for the family of all subsets in $\mathcal{P}\left(x\right)$
whose size is $k$, and we write $\left|x\right|$ for $\#\left\{ i:\,x_{i}=1\right\} .$ 

The first construction that we need associates with each family $\f\subseteq\binom{\left[n\right]}{k}$
a function $f_{\f}\colon\left\{ 0,1\right\} ^{n}\to\left[0,1\right].$
This construction has its origins in the work of Friedgut and Regev
\cite{friedgut2017kneser}. 
\begin{defn}
Let $\f\subseteq\binom{\left[n\right]}{k}$, we associate with $\f$
the function $f_{\f}$ defined by 
\[
f_{\f}\left(x\right)=\begin{cases}
0 & \left|x\right|<k\\
\Pr_{A\sim\binom{x}{k}}\left[A\in\f\right] & \left|x\right|\ge k
\end{cases}.
\]
 
\end{defn}

Another construction we need turns a function $f\colon\left\{ 0,1\right\} ^{n}\to\mathbb{R}$
into a Boolean function $\mathrm{Cut}_{\delta}\left(f\right).$
\begin{defn}
Given a function $f\colon\left\{ 0,1\right\} ^{n}\to\mathbb{R}$,
and a $\delta\in\mathbb{R}$, we define the function $\mathrm{Cut}_{\delta}\left(f\right)$
by setting:
\[
\mathrm{Cut}_{\delta}\left(f\right)\left(\x\right)=\begin{cases}
1 & \mbox{ if }f\left(\x\right)\ge\delta\\
0 & \mbox{ if }f\left(\x\right)<\delta
\end{cases}.
\]
 
\end{defn}

We shall also need to introduce the following distributions.
\begin{defn}
Let $p\in\left(0,1\right),$ and let $k\in\left[n\right].$ 

\begin{itemize}
\item We write $\mu_{p}^{\ge k}$ for the conditional probability distribution
on $\mathbf{x}\sim\left(\left\{ 0,1\right\} ^{n},\mu_{p}\right)$
given that $\left|\mathbf{x}\right|\ge k.$ (The distributions $\mu_{p}^{>k},\mu_{p}^{<k},\mu_{p}^{\le k}$
are defined accordingly.)
\item We write $\left(\mu_{p}^{\ge k},J\to B\right)$ for the conditional
distribution on sets $\mathbf{A}\sim\mu_{p}^{\ge k}$ given that $\mathbf{A}\cap J=B.$
The distributions $\left(\mu_{p},J\to B\right),\left(\binom{\left[n\right]}{k},J\to B\right)$
are defined accordingly.
\end{itemize}
\end{defn}

Another construction we need is the construction of random matchings
\[
\mathbf{B}_{1},\ldots,\mathbf{B}_{h}\in\p\left(\left[n\right]\right),
\]
 such that each of the sets $\mathbf{B}_{i}$ is distributed according
to the $\frac{1}{h}$-biased distribution. 
\begin{defn}
Choose uniformly and independently $\left[0,1\right]$-valued random
variables $X_{1},\ldots,X_{n}$. For each $i\in\left\{ 1,\ldots,h\right\} $
we let $\mathbf{B}_{i}$ be the set of all $j\in\left[n\right],$
such that $X_{j}$ is in the interval $\left[\frac{j-1}{h},\frac{j}{h}\right]$.
We call the sets $\left(\mathbf{B}_{1},\ldots,\mathbf{B}_{h}\right)$
a random \emph{$\frac{1}{h}$-biased matching}. Let $k\le\frac{n}{h}$
we call the conditional distribution on a random $\frac{1}{h}$-biased
matching $\left(\mathbf{B}_{1},\ldots,\mathbf{B}_{h}\right)$ given
that $\left|\mathbf{B}_{i}\right|\ge k$ for each $i$ a \emph{$\left(\frac{1}{h},k\right)$-biased
matching.} Given a $\left(\frac{1}{h},k\right)$-biased matching $\left(\mathbf{B}_{1},\ldots,\mathbf{B}_{h}\right)$,
we obtain that $\mathbf{B}_{1}$ is distributed according to some
distribution that we denote by $\mu_{\frac{1}{h},k}^{\mbox{matching}}.$
\end{defn}

Note that if $\left(\mathbf{B}_{1},\ldots,\mathbf{B}_{h}\right)$
a random \emph{$\frac{1}{h}$}-biased matching, then each $\mathbf{B}_{i}$
is indeed chosen according to the $\frac{1}{h}$-biased distribution,
and moreover the sets $\mathbf{B}_{1},\ldots,\mathbf{B}_{h}$ are
pairwise disjoint with probability 1. 

We will be concerned with the case where $k\le\frac{n}{h}-\Theta\left(n\right)$.
This would yield that $\left|\mathbf{B}_{i}\right|\ge k$ asymptotically
almost surely for all $i$. So intuitively, the distribution of a
$\frac{1}{h}$-biased matching is not very different from the distribution
of a $\left(\frac{1}{h},k\right)$-biased matching. 

The proof of Theorem \ref{thm:Counting matchings} consists of three
steps: 

(In the following $\epsilon_{1}$ is sufficiently small and $\epsilon_{2}=\epsilon_{2}\left(\epsilon_{1}\right)$
is sufficiently small)
\begin{enumerate}
\item We set $q$ to be slightly larger than $\frac{k}{n}$. The first step
is to show that for each of the families $\f_{i}$ of Theorem \ref{thm:Counting matchings},
the function $f_{\f_{i}}$ is $\left(\left\lceil \frac{1}{\epsilon_{1}}\right\rceil ,\epsilon_{1},\mu_{q}\right)$-regular. 
\item The second step is to show that the measure $\mu_{\frac{1}{h},k}^{\mbox{matching}}\left(\mathrm{Cut}_{\epsilon_{2}}\left(f_{\f_{i}}\right)\right)$
is very close to 1.

This step is based on Theorem \ref{thm:Robust version of Friedgut-Kalai Theorem},
and the proof roughly goes as follows. 
\begin{itemize}
\item We show that the term $\e_{\x,\y\sim D\left(q,\frac{1}{h}\right)}\left[f_{\mathcal{F}_{i}}\left(\x\right)\left(1-\mathrm{Cut}_{\epsilon_{2}}\left(f_{\mathcal{F}_{i}}\left(\y\right)\right)\right)\right]$
is always smaller than $\epsilon_{2}$. 
\item We shall apply Theorem \ref{thm:Robust version of Friedgut-Kalai Theorem}
to deduce that $\mu_{\frac{1}{h}}\left(\mathrm{Cut}_{\epsilon_{2}}\left(f_{\mathcal{F}_{i}}\right)\right)$
is large.
\item We shall use the similarity between $\mu_{\frac{1}{s}}$ and $\mu_{\frac{1}{h},k}^{\mathrm{matching}}$
to deduce that $\mu_{\frac{1}{h},k}^{\mbox{matching}}\left(\mathrm{Cut}_{\epsilon_{2}}\left(f_{\mathcal{F}_{i}}\right)\right)$
is large. 
\end{itemize}
\item We then finish the proof by observing that if we choose a $\left(\frac{1}{h},k\right)$-biased
matching $\mathbf{B}_{1},\ldots,\mathbf{B}_{h},$ and then choose
sets $\mathbf{M}_{1}\sim\binom{\mathbf{B}_{1}}{k_{1}},\ldots,\mathbf{M}_{s}\sim\binom{\mathbf{B}_{h}}{k_{h}}.$
Then $\mathbf{M}_{1},\ldots,\mathbf{M}_{h}$ is a uniformly random
matching. By Step 2 and a union bound we would have $\mathrm{Cut}_{\epsilon_{2}}\left(f_{\mathcal{F}_{i}}\right)\left(\mathbf{B}_{i}\right)=1$
with high probability. On the other hand for each $B_{i}$ with $\mathrm{Cut}_{\epsilon_{2}}\left(f_{\mathcal{F}_{i}}\right)\left(B_{i}\right)=1$,
we have $\Pr_{\mathbf{M_{i}\sim}\binom{B_{i}}{k_{i}}}\left[\mathbf{M}_{i}\in\mathcal{F}\right]\ge\epsilon_{2}$,
and for each choice of disjoint $B_{1},\ldots,B_{h}$ these events
are independent. Therefore, the probability that $\mathbf{M}_{i}$
is in $\mathcal{F}$ for each $i$ cannot be much smaller than $\epsilon_{2}^{h}.$ 
\end{enumerate}
We shall start with the proof of the first step. 

\subsection{Showing that if $\mathcal{F}$ is regular, then the function $f_{\protect\f}$
is $\left(\left\lceil \frac{1}{\epsilon}\right\rceil ,\epsilon,\mu_{q}\right)$-regular}

In order to show that the function $f_{\mathcal{F}}$ is regular,
we will need to show that $\mu_{q}\left(f_{\mathcal{F}}\right)$ is
approximately $\mu_{q}\left(\left(f_{\mathcal{F}}\right)_{J\to x}\right)$
for each $\left|J\right|\le\left\lceil \frac{1}{\epsilon}\right\rceil ,$
and each $x\in\left\{ 0,1\right\} ^{J}$. In order to accomplish this
we need to write both of the quantities $\mu_{q}\left(f_{\mathcal{F}}\right),\mu_{q}\left(\left(f_{\mathcal{F}}\right)_{J\to x}\right)$
in terms of $\mathcal{F}$. We shall start by showing that $\mu_{q}\left(f_{\mathcal{F}}\right)$
is approximately equal to $\mu\left(\mathcal{F}\right).$
\begin{lem}
\label{lem:measure of ff}For each $\epsilon>0,$ there exists $n_{0}>0$,
such that the following holds. Let $n>n_{0}$, let $q\in\left(0,1\right),k\le n$
satisfy $q\ge\frac{k}{n}+\epsilon,$ and let $\mathcal{F}\subseteq\binom{\left[n\right]}{k}$
be some family. Then 
\begin{equation}
\mu_{q}\left(f_{\mathcal{F}}\right)\le\mu\left(\mathcal{F}\right)\le\mu_{q}\left(f_{\mathcal{F}}\right)\left(1+\epsilon\right).\label{eq:conclusion chernoff}
\end{equation}
\end{lem}

\begin{proof}
We have 
\begin{align*}
\mu_{q}\left(f_{\mathcal{F}}\right) & =\mathbb{E}_{\mathbf{x}\sim\mu_{q}}\left[f_{\mathcal{F}}\left(\mathbf{x}\right)\right]=\Pr\left[\mathbf{x}\ge k\right]\mathbb{E}_{\mathbf{x}\sim\mu_{q}^{\ge k}}\left[f_{\mathcal{F}}\left(\mathbf{x}\right)\right]+\Pr\left[\mathbf{x}<k\right]\mathbb{E}_{\mathbf{x}\sim\mu_{q}^{<k}}\left[f_{\mathcal{F}}\left(\mathbf{x}\right)\right].\\
 & =\Pr_{\mathbf{x}\sim\mu_{q}}\left[\left|\mathbf{x}\right|\ge k\right]\mathbb{E}_{\mathbf{x}\sim\mu_{q}^{\ge k}}\left[\Pr_{\mathbf{A}\sim\binom{\mathbf{x}}{k}}\left[\mathbf{A}\in\mathcal{F}\right]\right].
\end{align*}
 However, whenever we choose $\mathbf{x}\sim\mu_{q}^{\ge k}$, and
an $\mathbf{A}\sim\binom{\mathbf{x}}{k}$, we obtain a set $\mathbf{A}$
that is distributed uniformly in $\binom{\left[n\right]}{k}$. Thus,
\begin{equation}
\mu_{q}\left(f_{\mathcal{F}}\right)=\Pr_{\mathbf{x}\sim\mu_{q}}\left[\left|\mathbf{x}\right|\ge k\right]\mu\left(\mathcal{F}\right).\label{eq:computation ff}
\end{equation}
 The lemma follows by combining (\ref{eq:computation ff}) with the
fact that $\Pr_{\mathbf{x}\sim\mu_{q}}\left[\left|\mathbf{x}\right|\ge k\right]$
tends to 1 as $n$ tends to infinity.
\end{proof}
We now turn to the task of approximating $\mu_{q}\left(\left(f_{\mathcal{F}}\right)_{J\to x}\right)$
in terms of $\mathcal{F}$. We show that for some $\lambda>0$ the
term $\mu_{q}\left(\left(f_{\mathcal{F}}\right)_{J\to x}\right)$
can be approximated by 
\[
\underset{\mathbf{C}\sim\left(\mathcal{P}\left(x\right),\mu_{\lambda}\right)}{\mathbb{E}}\left[\mu\left(\mathcal{F}_{J\to\mathbf{C}}\right)\right].
\]

\begin{lem}
\label{lem:measure of restricted ff} For each $\epsilon>0$ there
exists $n_{0},$ such that the following holds. Let $n>n_{0},$ let
$k\le n$ and let $q$ be a number in the interval $\left(\frac{k}{n}+\epsilon,1\right)$,
and set $\lambda=\frac{k}{qn}.$ Then 
\[
\mu_{q}\left(\left(f_{\mathcal{F}}\right)_{J\to x}\right)\left(1-\epsilon\right)\le\underset{\mathbf{C}\sim\left(\mathcal{P}\left(x\right),\mu_{\lambda}\right)}{\mathbb{E}}\left[\mu\left(\mathcal{F}_{J\to\mathbf{C}}\right)\right]\le\mu_{q}\left(\left(f_{\mathcal{F}}\right)_{J\to x}\right)\left(1+\epsilon\right).
\]
\end{lem}

\begin{proof}
As in Lemma \ref{lem:measure of ff}, we have
\begin{equation}
\mu_{q}\left(\left(f_{\mathcal{F}}\right)_{J\to x}\right)=\Pr_{\mathbf{y}\sim\left(\mu_{q},J\to x\right)}\left[\left|\mathbf{y}\right|\ge k\right]\underset{\mathbf{y}\sim\left(\mu_{q}^{\ge k},J\to x\right)}{\mathbb{E}}\left[f_{\mathcal{F}}\left(\mathbf{y}\right)\right].\label{eq:restricted ff 1}
\end{equation}
Note that 
\[
\Pr_{\mathbf{y}\sim\left(\mu_{q},J\to x\right)}\left[\left|\mathbf{y}\right|\ge k\right]=1-o\left(1\right),
\]
 where the $o\left(1\right)$ is with respect to $n$ tending to infinity.
So to complete the proof it remains to show that 
\[
\underset{\mathbf{y}\sim\left(\mu_{q}^{\ge k},J\to x\right)}{\mathbb{E}}\left[f_{\mathcal{F}}\left(\mathbf{y}\right)\right]=\left(1+o\left(1\right)\right)\underset{\mathbf{C}\sim\left(\mathcal{P}\left(x\right),\mu_{\lambda}\right)}{\mathbb{E}}\left[\mu\left(\mathcal{F}_{J\to\mathbf{C}}\right)\right].
\]
 Choose $\mathbf{y}\sim\left(\mu_{q}^{\ge k},J\to x\right),$$\mathbf{A\sim}\binom{\mathbf{y}}{k},$
then \textbf{$\mathbf{A}\cap J$ }is equal to some subset $\mathbf{C}$
of $\mathbf{x}.$ Note also that the conditional distribution of $\mathbf{A}$
given that $\mathbf{C=}C$ is the distribution of a uniformly random
element of $\binom{\left[n\right]}{k}$ that intersects $J$ at the
set $C.$ Therefore, 
\begin{align}
\underset{\mathbf{y}\sim\left(\mu_{q}^{\ge k},J\to x\right)}{\mathbb{E}}\left[f_{\mathcal{F}}\left(\mathbf{y}\right)\right] & =\Pr_{\mathbf{y}\sim\left(\mu_{q}^{\ge k},J\to x\right),\mathbf{A}\sim\binom{\mathbf{y}}{k}}\left[A\in\mathcal{F}\right]\nonumber \\
 & \sum_{C\subseteq x}\Pr\left[\mathbf{C}=C\right]\Pr\left[\mathbf{A\in\mathcal{F}\,|\,}\mathbf{C}=C\right]\nonumber \\
 & =\sum_{C\subseteq x}\Pr\left[\mathbf{C}=C\right]\mu\left(\mathcal{F}_{J}^{C}\right).\label{eq:restricted ff2}
\end{align}
So to complete the proof it remains to show that 
\begin{equation}
\Pr\left[\mathbf{C}=C\right]=\lambda^{\left|C\right|}\left(1-\lambda\right)^{\left|x\right|\backslash\left|C\right|}\left(1+o\left(1\right)\right).\label{eq:restricted ff3}
\end{equation}
 Indeed, with high probability $\left|\mathbf{y}\right|=qn\left(1+o\left(1\right)\right),$
and the conditional probability that $\mathbf{C}=C$ given that $\left|\mathbf{y}\right|=s$
is 
\[
\frac{\left|\left\{ S\in\binom{\mathbf{y}}{k}\mbox{ that satisfy }\left|S\cap x\right|=C\right\} \right|}{\left|\binom{\mathbf{y}}{k}\right|}=\frac{\binom{s-\left|x\right|}{k-\left|C\right|}}{\binom{s}{k}}=\left(\frac{k}{s}\right)^{\left|C\right|}\left(1-\frac{k}{s}\right)^{\left|x\right|\backslash\left|C\right|}\left(1+o\left(1\right)\right).
\]
Let $\mathbf{s}=\left|\mathbf{y}\right|.$ Thus,
\begin{align*}
\Pr\left[\mathbf{C}=C\right] & =\mathbb{E}_{\mathbf{s}}\left[\Pr\left[\mathbf{C}=C\,|\,\mathbf{s}\right]\right]=\mathbb{E}_{\mathbf{s}}\left[\left(\frac{k}{\mathbf{s}}\right)^{\left|C\right|}\left(1-\frac{k}{\mathbf{s}}\right)^{\left|x\right|\backslash\left|C\right|}\left(1+o\left(1\right)\right)\right]\\
 & =\lambda^{\left|C\right|}\left(1-\lambda\right)^{\left|x\right|\backslash\left|C\right|}\left(1+o\left(1\right)\right),
\end{align*}
 where the last equality follows from the fact that $\frac{k}{\mathbf{s}}=\lambda\left(1+o\left(1\right)\right)$
with high probability. This completes the proof of the lemma. 
\end{proof}
We are now ready to show that if $\f\subseteq\binom{\left[n\right]}{k}$
is a $\left(\left\lceil \frac{1}{\epsilon}\right\rceil ,\epsilon\right)$-regular
family, and if we choose $q$ that is bounded from below away from $\frac{k}{n}$,
then the function $f_{\f}$ is $\left(\left\lceil \frac{1}{2\epsilon}\right\rceil ,2\epsilon,\mu_{q}\right)$-regular,
provided that $n$ is sufficiently large. 
\begin{lem}
For each $\epsilon<0$, there exists  $n_{0},$ such that the following
holds. Let $n>n_{0}$, let $\f\subseteq\binom{\left[n\right]}{k}$
be a $\left(\left\lceil \frac{1}{2\epsilon}\right\rceil ,2\epsilon\right)$-regular
family, and let $q\ge\frac{k}{n}+\epsilon.$ Then the function $f_{\f}$
is $\left(\left\lceil \frac{1}{\epsilon}\right\rceil ,\epsilon,\mu_{q}\right)$-regular.
\end{lem}

\begin{proof}
Fix $\epsilon>0,$ let $n_{0}$ be sufficiently large, and let $\mathcal{F}\subseteq\binom{\left[n\right]}{k},$
be as in the hypothesis of the lemma. Let $B\subseteq J\subseteq\left[n\right]$
be  sets, such that $\left|J\right|\le\left\lceil \frac{1}{\epsilon}\right\rceil $.
By Lemma \ref{lem:measure of ff} 
\begin{equation}
\left|\mu\left(\mathcal{F}\right)-\mu_{q}\left(f_{\mathcal{F}}\right)\right|<\frac{\epsilon}{4},\label{eq:regularity k large1}
\end{equation}
 provided that $n_{0}$ is large enough. By Lemma \ref{lem:measure of restricted ff}
\begin{equation}
\left|\mu_{q}\left(\left(f_{\mathcal{F}}\right)_{J\to B}\right)-\mathbb{E}_{\mathbf{C}\sim\left(\p\left(x\right),\mu_{\frac{k}{qn}}\right)}\mu\left(\mathcal{F}_{J}^{C}\right)\right|<\frac{\epsilon}{4},\label{eq:regularity k large2}
\end{equation}
 provided that $n_{0}$ is large enough.

By hypothesis 
\begin{align}
\left|\mu\left(\mathcal{F}\right)-\mathbb{E}_{\mathbf{C}\sim\left(\p\left(x\right),\mu_{\frac{k}{qn}}\right)}\mu\left(\mathcal{F}_{J}^{C}\right)\right| & =\left|\mathbb{E}_{\mathbf{C}\sim\left(\p\left(x\right),\mu_{\frac{k}{qn}}\right)}\left(\mu\left(\mathcal{F}_{J}^{C}\right)-\mu\left(\f\right)\right)\right|\label{eq:regularity k large3}\\
 & \le\mathbb{E}_{\mathbf{C}\sim\left(\p\left(x\right),\mu_{\frac{k}{qn}}\right)}\frac{\epsilon}{2}=\frac{\epsilon}{2}.\nonumber 
\end{align}
 Thus, 
\begin{align*}
\left|\mu_{q}\left(\left(f_{\mathcal{F}}\right)_{J\to B}\right)-\mu_{q}\left(f_{\mathcal{F}}\right)\right| & \le\left|\mu_{q}\left(\left(f_{\mathcal{F}}\right)_{J\to B}\right)-\mathbb{E}_{\mathbf{C}\sim\left(\p\left(x\right),\mu_{\frac{k}{qn}}\right)}\mu\left(\mathcal{F}_{J}^{C}\right)\right|\\
 & +\left|\mathbb{E}_{\mathbf{C}\sim\left(\p\left(x\right),\mu_{\frac{k}{qn}}\right)}\mu\left(\mathcal{F}_{J}^{C}\right)-\mu\left(\mathcal{F}\right)\right|+\left|\mu\left(\mathcal{F}\right)-\mu_{q}\left(f_{\mathcal{F}}\right)\right|\\
 & <\frac{\epsilon}{4}+\frac{\epsilon}{4}+\frac{\epsilon}{2}=\epsilon.
\end{align*}
 This completes the proof that $f_{\f}$ is $\left(\left\lceil \frac{1}{\epsilon}\right\rceil ,\epsilon,\mu_{q}\right)$-regular.
\end{proof}

\subsection{Showing that if $\frac{k}{n}$ is small, then $f_{\mathcal{F}}$
is \textmd{$\left(\left\lceil \frac{1}{\epsilon}\right\rceil ,\epsilon,\mu_{q}\right)$-regular}}
\begin{lem}
\label{lem:k small}For each $\epsilon>0,$ there exists $\delta>0$
such that the following holds. Let $\frac{k}{n}<\delta$, let $q\ge\epsilon$,
and let $\f\subseteq\binom{\left[n\right]}{k}$ be some family. Then
the function $f_{\f}$ is $\left(\left\lceil \frac{1}{\epsilon}\right\rceil ,\epsilon,\mu_{q}\right)$-regular.
\end{lem}

\begin{proof}
Let $J$ be of size at most $\left\lceil \frac{1}{\epsilon}\right\rceil ,$
and let $B\subseteq J.$ We have 
\begin{equation}
\left|\mu_{q}\left(\left(f_{\mathcal{F}}\right)_{J\to B}\right)-\mu_{q}\left(f_{\mathcal{F}}\right)\right|\le\left|\mu\left(\mathcal{F}_{J}^{\varnothing}\right)-\mu_{q}\left(f_{\mathcal{F}}\right)\right|+\left|\mu\left(\mathcal{F}_{J}^{\varnothing}\right)-\mu_{q}\left(\left(f_{\mathcal{F}}\right)_{J\to B}\right)\right|.\label{eq:k small 1}
\end{equation}
We shall complete the proof by giving an upper bound of $\frac{\epsilon}{2}$
for each of the summands in the right hand side of (\ref{eq:k small 1}).

\textbf{Showing that }$\left|\mu\left(\mathcal{F}_{J}^{\varnothing}\right)-\mu_{q}\left(f_{\mathcal{F}}\right)\right|\le\frac{\epsilon}{2}.$

By decreasing $\delta$ if necessary we may assume that $n$ is as
large as we wish. Therefore Lemma \ref{lem:measure of ff} implies
that 
\[
\left|\mu_{q}\left(f_{\mathcal{F}}\right)-\mu\left(\mathcal{F}\right)\right|<\frac{\epsilon}{4},
\]
 provided that $\delta$ is small enough. Now note that 
\begin{align}
\mu\left(\mathcal{F}\right) & =\sum_{B\subseteq J}\Pr_{\mathbf{A}\sim\binom{\left[n\right]}{k}}\left[\mathbf{A}\cap J=B\right]\mu\left(\mathcal{F}_{J}^{B}\right)\label{eq:k divided by n is negligible}\\
 & =\left(1-O\left(\frac{k}{n}\right)\right)\mu\left(\mathcal{F}_{J}^{\varnothing}\right)+\sum_{\varnothing\ne B\subseteq J}O\left(\left(\frac{k}{n}\right)^{\left|B\right|}\right)\mu\left(\mathcal{F}_{J}^{B}\right).\nonumber 
\end{align}
 So provided that $\delta$ is small enough, we have 
\[
\left|\mu_{q}\left(f_{\mathcal{F}}\right)-\mu\left(\mathcal{F}_{J}^{\varnothing}\right)\right|\le\left|\mu_{q}\left(f_{\mathcal{F}}\right)-\mu\left(\mathcal{F}\right)\right|+\left|\mu\left(\mathcal{F}\right)-\mu\left(\mathcal{F}_{J}^{\varnothing}\right)\right|\le\frac{\epsilon}{4}+O_{\epsilon}\left(\frac{k}{n}\right)<\epsilon/2.
\]
 \textbf{Showing that} $\left|\mu\left(\mathcal{F}_{J}^{\varnothing}\right)-\mu_{q}\left(\left(f_{\mathcal{F}}\right)_{J\to B}\right)\right|<\frac{\epsilon}{2}.$ 

By Lemma \ref{lem:k small} 
\[
\left|\mu_{q}\left(\left(f_{\mathcal{F}}\right)_{J\to B}\right)-\sum_{C\subseteq B}\left(\frac{k}{qn}\right)^{\left|C\right|}\left(1-\frac{k}{qn}\right)^{\left|B\right|\backslash\left|C\right|}\mu\left(\mathcal{F}_{J}^{C}\right)\right|\le\frac{\epsilon}{4},
\]
provided that $n$ is large enough. Now note that similarly to (\ref{eq:k divided by n is negligible})
we have 
\[
\sum_{C\subseteq B}\left(\frac{k}{qn}\right)^{\left|C\right|}\left(1-\frac{k}{qn}\right)^{\left|B\right|\backslash\left|C\right|}\mu\left(\mathcal{F}_{J}^{C}\right)=\mu\left(\mathcal{F}_{J}^{\varnothing}\right)+O_{\epsilon}\left(\frac{k}{n}\right).
\]
 Thus, 
\[
\left|\mu_{q}\left(\left(f_{\mathcal{F}}\right)_{J\to B}\right)-\mu\left(\mathcal{F}_{J}^{\varnothing}\right)\right|\le\frac{\epsilon}{4}+O_{\epsilon}\left(\frac{k}{n}\right)<\frac{\epsilon}{2},
\]
 provided that $\delta$ is small enough. This completes the proof
of the lemma.
\end{proof}

\subsection{Showing that if $\mathcal{F}$ is $\left(\left\lceil \frac{1}{\delta}\right\rceil ,\delta,\mu_{q}\right)$-regular,
then $\mu_{\frac{1}{h}}\left(\mathrm{Cut}_{\delta}\left(f_{\protect\f}\right)\right)$
is large}

In the last two sections we gave two criteria on a family $\mathcal{F}$
that imply that the function $f_{\mathcal{F}}$ is $\left(\left\lceil \frac{1}{\delta}\right\rceil ,\delta,\mu_{q}\right)$-regular
for a small $\delta$. We shall now show that both criteria imply
that $\mu_{\frac{1}{h},k}^{\mbox{matching}}\left(\mathrm{Cut}_{\delta}\left(f_{\mathcal{F}}\right)\right)$
is large. As mentioned, we will first show that the pair $\left\langle \t^{\frac{1}{h}\to q}f_{\mathcal{F}},\mathrm{Cut}_{\delta}\left(f_{\mathcal{F}}\right)\right\rangle <\epsilon$.
We will then deduce from Theorem \ref{thm:Robust version of Friedgut-Kalai Theorem}
that $\mu_{\frac{1}{h}}\left(\mathrm{Cut}_{\delta}\left(f_{\mathcal{F}}\right)\right)$
is large, and this will allow us to finish the proof by using the
fact that the distributions $\mu_{\frac{1}{h},k}^{\mbox{matching}},\mu_{\frac{1}{h}}$
are very close to each other.
\begin{lem}
\label{lem:ff cut delta ff are stable}Let $\delta>0$, let $p>q>\frac{k}{n}$.
For any $\f\subseteq\binom{\left[n\right]}{k}$ we have 
\[
\mathbb{E}_{\mathbf{x,y}\sim D\left(q,p\right)}\left[f_{\mathcal{F}}\left(\mathbf{x}\right)\left(1-\mathrm{Cut}_{\delta}\left(f_{\mathcal{F}}\right)\left(\mathbf{y}\right)\right)\right]\le\delta.
\]
\end{lem}

\begin{proof}
We show the stronger statement that for each value $y$ of $\mathbf{y}$,
we obtain that if we choose conditionally $\mathbf{x},\mathbf{y}\sim D\left(q,p\right)$
given that $\mathbf{y}=y,$ then 
\begin{equation}
\mathbb{E}_{\mathbf{x}}\left[f_{\mathcal{F}}\left(\mathbf{x}\right)\left(1-\mathrm{Cut}_{\delta}\left(f_{\mathcal{F}}\right)\left(y\right)\right)\right]\le\delta.\label{eq:stronger lemma stability}
\end{equation}
 This clearly holds if $\mathrm{Cut}_{\delta}\left(f_{\mathcal{F}}\right)\left(y\right)=1$.
So suppose that $\mathrm{Cut}_{\delta}\left(f_{\mathcal{F}}\right)\left(y\right)=0.$
We also suppose that $\left|y\right|\ge k$, for otherwise we would
have $\left|\mathbf{x}\right|<k$, and hence $f_{\mathcal{F}}\left(\mathbf{x}\right)=0.$
Now note that 
\begin{align*}
\mathbb{E}_{\mathbf{x}}\left[f_{\mathcal{F}}\left(\mathbf{x}\right)\right] & \le\mathbb{E}_{\mathbf{x}}\left[\Pr_{A\sim\binom{\mathbf{x}}{k}}\left[A\in\mathcal{F}\right]\,|\,\left|\mathbf{x}\right|\ge k\right]=\Pr_{\mathbf{A}\sim\binom{y}{k}}\left[\mathbf{A}\in\mathcal{F}\right]=f_{\mathcal{F}}\left(y\right)\le\delta.
\end{align*}
This completes the proof of the lemma.
\end{proof}
We shall now complete the second step of showing that $\mu_{\frac{1}{h}}\left(\mathrm{Cut}_{\delta}\left(f_{\mathcal{F}}\right)\right)$
is large. 
\begin{lem}
\label{lem:mu one s large}For each $\epsilon>0,$ there exists $\delta>0$,
such that the following holds. Let $p\ge\frac{k}{n}+\epsilon$, and
let $\f\subseteq\binom{\left[n\right]}{k}$ be some family whose measure
is at least $\epsilon$. Suppose that we either have $k\le\delta n$
or the family $\f$ is $\left(\left\lceil \frac{1}{\delta}\right\rceil ,\delta\right)$-regular.
Then 
\[
\mu_{p}\left(\mathrm{Cut}_{\delta}\left(f_{\mathcal{F}}\right)\right)>1-\epsilon.
\]
 
\end{lem}

\begin{proof}
Let $q=\frac{k}{n}+\frac{\epsilon}{2},$ and note that by Lemma \ref{lem:ff cut delta ff are stable}
we have 
\[
\mathbb{E}_{\mathbf{x,y}\sim D\left(q,p\right)}\left[f_{\mathcal{F}}\left(\mathbf{x}\right)\left(1-\mathrm{Cut}_{\delta}\left(f_{\mathcal{F}}\right)\left(\mathbf{y}\right)\right)\right]\le\delta.
\]
 By decreasing $\delta$ if necessary, we may assume that $n$ is
sufficiently large for Lemma \ref{lem:measure of ff} to imply that
$\mu_{q}\left(f_{\mathcal{F}}\right)\ge\frac{\epsilon}{2}$. By Theorem
\ref{thm:Robust version of Friedgut-Kalai Theorem} (applied with
$\frac{\epsilon}{2}$ rather than $\epsilon$) we have $\mu_{p}\left(f_{\mathcal{F}}\right)\ge1-\frac{\epsilon}{2}>1-\epsilon$,
provided that $\delta$ is small enough. This completes the proof
of the lemma. 
\end{proof}
We now complete the final step of deducing from the fact that $\mu_{\frac{1}{h}}$
and $\mu_{\frac{1}{h},k}^{\mbox{matching}}$ are `close', and from
the fact that $\mu_{\frac{1}{h}}\left(\mathrm{Cut}_{\delta}\left(f_{\mathcal{F}}\right)\right)$
is large, that $\mu_{\frac{1}{h},k}^{\mbox{matching}}\left(\mathrm{Cut}_{\delta}\left(f_{\mathcal{F}}\right)\right)$
is in fact also large.
\begin{cor}
\label{cor:mu matching large}For each $\epsilon>0,h$ there exists
$\delta>0$, such that the following holds. Let $\frac{k}{n}\le\frac{1}{h}-\epsilon$,
and let $\f\subseteq\binom{\left[n\right]}{k}$ be some family whose
measure is at least $\epsilon$. Suppose that we either have $k\le\delta n$
or the family $\f$ is $\left(\left\lceil \frac{1}{\delta}\right\rceil ,\delta\right)$-regular.
Then 
\[
\mu_{\frac{1}{h},k}^{\mathrm{mathching}}\left(\mathrm{Cut}_{\delta}\left(f_{\mathcal{F}}\right)\right)\ge1-\epsilon.
\]
\end{cor}

\begin{proof}
By Lemma \ref{lem:mu one s large}, we have $\mu_{\frac{1}{h}}\left(\mathrm{Cut}_{\delta}\left(f_{\mathcal{F}}\right)\right)>1-\frac{\epsilon}{2}$
provided that $\delta$ is small enough. Also note that we may assume
that $n$ is sufficiently large by decreasing $\delta$ if necessary. 

We shall now define a coupling between $\frac{1}{h}$-biased matching,
and $\left(\frac{1}{h},k\right)$-biased matchings as follows. We
choose a $\frac{1}{h}$-biased matching $\mathbf{M}_{1},\mathbf{M}_{2},\ldots,\mathbf{M}_{h}$,
and we then let $\mathbf{M}_{1}',\ldots,\mathbf{M}_{s}'$ to be equal
to the original $\frac{1}{s}$-biased matching if the (likely) event
$\forall i:\,\left|\mathbf{M}_{i}\right|\ge k$ occurred, and we let
it be equal to a new $\left(\frac{1}{h},k\right)$-biased matching
otherwise. Note that 
\[
\Pr\left[\mathbf{M}_{1}\in\mathrm{Cut}_{\delta}\left(f_{\mathcal{F}}\right)\right]=\mu_{\frac{1}{h}}\left(\mathrm{Cut}_{\delta}\left(f_{\mathcal{F}}\right)\right),
\]
 and that 
\[
\Pr\left[\mathbf{M}'_{1}\in\mathrm{Cut}_{\delta}\left(f_{\mathcal{F}}\right)\right]=\mu_{\frac{1}{h},k}^{\mbox{matching}}\left(\mathrm{Cut}_{\delta}\left(f_{\mathcal{F}}\right)\right).
\]
 Thus, 
\[
\mu_{\frac{1}{h},k}^{\mbox{matching}}\left(\mathrm{Cut}_{\delta}\left(f_{\mathcal{F}}\right)\right)\ge\mu_{\frac{1}{h}}\left(\mathrm{Cut}_{\delta}\left(f_{\mathcal{F}}\right)\right)-\Pr\left[\mathbf{M}_{1}\ne\mathbf{M}_{1}'\right]\ge1-\epsilon,
\]
provided that $n$ is sufficiently large to imply $\Pr\left[\mathbf{M}_{1}\ne\mathbf{M}_{1}'\right]<\frac{\epsilon}{2}.$
\end{proof}

\subsection{Proof of Theorem \ref{thm:Counting matchings}}
\begin{proof}[Proof of Theorem \ref{thm:Counting matchings}]
 Let $\mathcal{F}_{1}\subseteq\binom{\left[n\right]}{k_{1}},\ldots,\mathcal{F}_{h}\subseteq\binom{\left[n\right]}{k_{h}}$
be some families that satisfy the hypothesis of the theorem, let $\delta'=\delta'\left(\epsilon\right)$
be sufficiently small, and let $\delta=\frac{\left(\delta'\right)^{h}}{3}$.
By Corollary \ref{cor:mu matching large}, 
\[
\mu_{\frac{1}{h},k}^{\mbox{matching}}\left(\mathrm{Cut}_{\delta'}\left(f_{\mathcal{F}_{i}}\right)\right)>1-\frac{1}{2h}
\]
 for each $i$, provided that $\delta'$ is small enough. A union
bound implies that if we choose a $\left(\frac{1}{h},k\right)$-biased
matching $\mathbf{A}_{1},\ldots,\mathbf{A}_{h},$ then the event $\forall i\,\mathrm{Cut}_{\delta'}\left(f_{\mathcal{F}_{i}}\right)\left(\mathbf{A}_{i}\right)=1$
happens with probability greater than $\frac{1}{2}.$ Now choose independently
a matching $\mathbf{M}_{i}\sim\binom{\mathbf{A}_{i}}{k_{i}}.$ For
each choice of values $A_{1},\ldots,A_{h}$ of the sets $\boldsymbol{A}_{1},\ldots,\boldsymbol{A}_{h}$,
we obtain that the events $\left\{ \mathbf{M}_{i}\in\mathcal{F}_{i}|\boldsymbol{A}_{i}=A_{i}\right\} _{i=1}^{h}$
are independent. Therefore 
\[
\Pr\left[\forall i:\,\boldsymbol{M}_{i}\in\f_{i}\right]\ge\Pr\left[\forall i\,\mathrm{Cut}_{\delta}\left(f_{\mathcal{F}_{i}}\right)\left(\mathbf{A}_{i}\right)=1\right]\delta'^{h}\ge\frac{\delta'^{h}}{2}=\delta.
\]
 This completes the proof of the theorem, since the hypergraph $\left\{ \mathbf{M}_{1},\ldots,\mathbf{M}_{h}\right\} $
is a uniformly random matching.
\end{proof}

\section{\label{sec:Counting-general-expanded hypergraphs} Counting expanded
hypergraphs}

We shall now generalize Theorem \ref{thm:Counting matchings} to general
expanded hypergraphs. 
\begin{thm}
\label{thm:Counting expanded hypergraphs} For each $c,h\in\mathbb{N},\epsilon>0$
there exists $\delta>0$, such that the following holds. Let $\frac{1}{\delta}\le k_{1},\ldots,k_{h}\le\left(\frac{1}{h}-\epsilon\right)n$
be some numbers, and let $A_{i}\in\binom{\left[n\right]}{k_{i}}$
be sets, such that the hypergraph $\h=\left\{ A_{1},\ldots,A_{h}\right\} $
has center of size at most $c$. Let $\mathcal{F}_{1}\subseteq\binom{\left[n\right]}{k_{1}},\ldots,\mathcal{F}_{h}\subseteq\binom{\left[n\right]}{k_{h}}$
be some families of measure at least $\epsilon.$ Suppose that for
each $i$, we either have $\frac{k_{i}}{n}\le\delta$ or the family
$\f_{i}$ is $\left(\left\lceil \frac{1}{\delta}\right\rceil ,\delta\right)$-regular.
Choose uniformly at random a copy\textbf{ $\left(\mathbf{A}_{1},\ldots,\mathbf{A}_{h}\right)\in\binom{\left[n\right]}{k_{1}}\times\cdots\times\binom{\left[n\right]}{k_{h}}$
}of $\mathcal{H}$. Then 
\[
\Pr\left[\mathbf{A}_{1}\in\f_{1},\ldots,\mathbf{A}_{h}\in\f_{h}\right]>\delta.
\]
\end{thm}

Note that $\mathcal{H\subseteq P}\left(V\right)$ can be written in
the form 
\[
\left\{ E_{1}\cup D_{1},\ldots,E_{h}\cup D_{h}\right\} ,
\]
 where $C:=E_{1}\cup\cdots\cup E_{h}$ is the center of $\mathcal{H}$,
and where the sets $C,D_{1},\ldots,D_{h}$ are pairwise disjoint.
If $\pi:V\to\left[n\right]$ is a random injection, then $\pi\left(E_{1}\cup D_{1}\right),\ldots,\pi\left(E_{h}\cup D_{h}\right)$
is a uniformly random copy of $\mathcal{H}$. Write $\mathbf{E}_{i}=\pi\left(E_{i}\right),$
$\mathbf{D}_{i}=\pi\left(D_{i}\right),$ and $\mathbf{C}=\pi\left(C\right).$
Our basic observation is that the following events are equal.
\begin{enumerate}
\item The families $\mathcal{F}_{1},\ldots,\mathcal{F}_{h}$ cross-contain
the random copy of $\mathcal{H}$ 
\[
\left(\pi\left(E_{1}\cup D_{1}\right),\ldots,\pi\left(E_{h}\cup D_{h}\right)\right)=\left(\mathbf{E}_{1}\cup\mathbf{D}_{1},\ldots,\mathbf{E}_{h}\cup\mathbf{D}_{h}\right).
\]
 
\item The families $\left(\mathcal{F}_{i}\right)_{\mathbf{C}}^{\mathbf{E}_{i}}$
cross-contain the uniformly random matching $\mathbf{D}_{1},\ldots,\mathbf{D}_{h}$.
\end{enumerate}
Therefore it is natural to try to apply Theorem \ref{thm:Counting matchings}
on the families $\left(\mathcal{F}_{i}\right)_{\mathbf{C}}^{\mathbf{E}_{i}}.$
As it turns out, the only hypothesis of Theorem \ref{thm:Counting matchings}
that the families $\left(\mathcal{F}_{i}\right)_{\mathbf{C}}^{\mathbf{E}_{i}}$
do not obviously satisfy is the hypothesis that $\mu\left(\left(\mathcal{F}_{i}\right)_{\mathbf{C}}^{\mathbf{E}_{i}}\right)>\epsilon.$
The following Fairness Proposition by Keller and the author \cite[Proposition 5.1]{keller2017junta}
allows us to deduce that the families $\left(\mathcal{F}_{i}\right)_{\mathbf{C}}^{\mathbf{E}_{i}}$
are of measure greater than $\epsilon$ with high probability. To
state their result we need to introduce the notion of fairness. Roughly
speaking, a set $J$ is `fair' for $\mathcal{F}$ if the measures
of each of the families $\mathcal{F}_{J}^{B}$ are not significantly
smaller than the measure of $\mathcal{F}.$
\begin{defn}
A set $J$ is $\mathcal{\epsilon}$-fair for $\mathcal{F}$ if 
\[
\mu\left(\mathcal{F}_{J}^{B}\right)\ge\left(1-\epsilon\right)\mu\left(\mathcal{F}\right)
\]
 for any $B\subseteq J.$ 
\end{defn}

The Fairness Proposition tells us that for any $\mathcal{F}\subseteq\binom{\left[n\right]}{k}$
and any $\mathbf{J}\sim\binom{\left[n\right]}{s}$ is $\epsilon$-fair
for $\mathcal{F}$ with high probability, provided that $k$ is not
too close to either $0$ or $n$.
\begin{prop}
\label{prop:fairness}For each $\epsilon,s>0,$ there exists $m>0$
such that the following holds. Let $m<k<n-m$, let $\mathcal{F}\subseteq\binom{\left[n\right]}{k}$
be some family of measure at least $\epsilon$, and let $\mathbf{J}\sim\binom{\left[n\right]}{s}$.
Then
\[
\Pr\left[\mathbf{J}\mbox{ is }\epsilon\text{-fair for }\f\right]\ge1-\epsilon.
\]
 
\end{prop}

\begin{proof}[Proof of Theorem \ref{thm:Counting expanded hypergraphs} ]
 Let $\boldsymbol{E}_{i},\boldsymbol{C},\boldsymbol{D}_{i}$ be as
above. Our goal is to show that the families $\left(\mathcal{F}_{i}\right)_{\mathbf{C}}^{\mathbf{E}_{i}}$
cross-contain the uniformly random matching $\mathbf{D}_{1},\ldots,\mathbf{D}_{h}$
with probability $\ge\delta$. Noting that the size of $\mathbf{C}$
is fixed, the following observations are easy to verify provided that
$\delta$ is sufficiently small:

\begin{itemize}
\item Proposition \ref{prop:fairness}, implies that the set $\mathbf{C}$
is $\frac{1}{2}$-fair with probability at least $\frac{1}{2}$. For
any such $\mathbf{C}$ the measure of the family $\left(\mathcal{F}_{i}\right)_{\mathbf{C}}^{\mathbf{E}_{i}}\subseteq\binom{\left[n\right]\backslash\mathbf{C}}{k_{i}-\left|E_{i}\right|}$
is at least $\frac{\epsilon}{2}$.
\item For any $i$ such that $\frac{k_{i}}{n}<\delta,$ we have $\frac{k_{i}-\left|E_{i}\right|}{n-\left|\mathbf{C}\right|}<2\delta.$ 
\item If $\mathcal{F}_{i}$ is $\left(\left\lceil \frac{1}{\delta}\right\rceil ,\delta\right)$-regular,
then $\left(\mathcal{F}_{i}\right)_{\mathbf{C}}^{\mathbf{E}_{i}}$
is $\left(\left\lceil \frac{1}{2\delta}\right\rceil ,2\delta\right)$-regular. 
\end{itemize}
We shall also assume that the $\delta$ of this lemma is small enough
for Theorem \ref{thm:Counting matchings} to hold with $2\delta$
replacing $\delta,$ and $\frac{\epsilon}{2}$ replacing $\epsilon.$
These observations allow us to apply Theorem \ref{thm:Counting matchings},
and to deduce that for each set $C'$ that is $\frac{1}{2}$-fair
for $\mathcal{F}$, and for each set $E'_{i}\in\binom{C'}{\left|E_{i}\right|}$
we have 
\[
\Pr\left[\forall i:\,\mathbf{D}_{i}\in\mathcal{F}_{C'}^{E_{i}}\right]>2\delta.
\]
 Therefore, 
\begin{align*}
\Pr\left[\forall i:\,\mathbf{D}_{i}\in\mathcal{F}_{C'}^{E_{i}}\right] & \ge\Pr\left[\mathbf{C}\mbox{ is }\frac{1}{2}\text{-fair}\right]\Pr\left[\forall i:\,\mathbf{D}_{i}\in\mathcal{F}_{\mathbf{C}}^{\mathbf{E}_{i}}\,|\,\mathbf{C}\text{ is }\frac{1}{2}\text{-fair}\right]>\delta.
\end{align*}
 This completes the proof of the theorem.
\end{proof}

\section{Removal lemma for expanded hypergraphs}

In this section we prove Theorem \ref{thm:Genneral removal Lemma},
Proposition \ref{Prop:converse to removal lemma}, and Theorem \ref{thm:Removal for matchings}. 

Let $\mathcal{H}=\left\{ A_{1},\ldots,A_{h}\right\} \subseteq\mathcal{P}\left(V\right).$
Any hypergraph of the form $\left\{ A_{1}\cap S,\ldots,A_{h}\cap S\right\} $
is called a \emph{trace} of $\mathcal{H}.$ We shall need the following
lemma.
\begin{lem}
\label{lem:junta is h d free meaning} For each $h,c,s,j\in\mathbb{N},\epsilon>0$
there exists $\delta>0$, such that the following holds. Let $\h$
be a hypergraph with $h$ edges whose center is of size $c$, and
let $\mathcal{G}\subseteq\mathcal{P}\left(J\right)$. Let $\frac{1}{\delta}\le k\le\left(\frac{1}{h}-\epsilon\right)n.$
Then the following are equivalent.

\begin{enumerate}
\item The junta $\left\langle \mathcal{G}\right\rangle $ is $\left(\mathcal{H},s\right)$-free.
\item There exists no copy of a trace of $\mathcal{H}$ in $\mathcal{G}$,
whose center is of size at most $s.$ 
\end{enumerate}
\end{lem}

\begin{proof}
We start by showing that if (2) does not hold, then (1) does not holds.
By hypothesis, there exist a trace $\left\{ C_{1},\ldots,C_{h}\right\} $
of $\mathcal{H}$ in $\mathcal{G}$ whose center is of size at most
$s.$ Let $B_{1}\in\binom{\left[n\right]\backslash J}{k-\left|C_{1}\right|},\ldots,B_{h}\in\binom{\left[n\right]\backslash J}{k-\left|C_{h}\right|}$
be some pairwise disjoint sets (such sets exist provided that $\delta$
is large enough). Then the hypergraph $\left\{ C_{1}\cup B_{1},\ldots,C_{h}\cup B_{h}\right\} $
is contained in $\mathcal{\left\langle \mathcal{G}\right\rangle }$,
it is the resolution of the hypergraph $\mathcal{H}$ and its center
is of size at most $s$. Therefore, the family $\mathcal{\left\langle G\right\rangle }$
is not $\left(\mathcal{H},s\right)$-free and so (1) does not hold. 

We now show that if (1) does not hold, then (2) does not hold. Let
$\left\{ A_{1},\ldots,A_{h}\right\} \subseteq\left\langle \mathcal{G}\right\rangle $
be a resolution of $\mathcal{H}$ whose center is of size at most
$s.$ The hypergraph $\left\{ A_{1}\cap J,\ldots,A_{h}\cap J\right\} $
is contained in $\mathcal{G},$ its center is of size at most $s$,
and in order to complete the proof we need to show that it is a trace
of $\mathcal{H}.$ 

For each $i=1,\ldots,h$ let $D_{i}\subseteq\left[n\right]\backslash J$
be a sufficiently large set that is contained in $A_{i}$ and does
not intersect any other edge of $\mathcal{H}$. The fact that $\left\{ A_{1},\ldots,A_{h}\right\} $
is a resolution of $\mathcal{H}$ implies that there exist sets $E_{1},\ldots,E_{h}\subseteq\left[n\right]\backslash J,$
such that 
\[
\mathcal{H}':=\left\{ \left(A_{1}\backslash D_{1}\right)\cup E_{1},\ldots,\left(A_{h}\backslash D_{h}\right)\cup E_{h}\right\} 
\]
 is a copy of $\mathcal{H}.$ Now note that if we intersect each of
the edges of $\mathcal{H}'$ with $J$, we obtain the original hypergraph
$\left\{ A_{1}\cap J,\ldots,A_{h}\cap J\right\} .$ Therefore, $A_{1}\cap J,\ldots,A_{h}\cap J$
is indeed a trace of $\mathcal{H}.$ This completes the proof of the
lemma.
\end{proof}
We will also need the following lemma.
\begin{lem}
\label{lem:removal expanded calculation} Let $\epsilon>0,s\le c,h\in\mathbb{N}$
be some constants. Let $\epsilon n\le k\le\left(\frac{1}{h}-\epsilon\right)n,$
let $\mathcal{H}$ be a $k$-uniform hypergraph whose center is of
size $c$. Let $\left\{ \mathbf{A}_{1},\ldots,\mathbf{A}_{h}\right\} $
be a uniformly random copy of $\mathcal{H}$, let $J\subseteq\left[n\right]$
be some set of size at most $\frac{1}{\epsilon},$ and let $\left\{ B_{1},\ldots,B_{h}\right\} \subseteq\mathcal{P}\left(J\right)$
be a trace of $\mathcal{H}$ of center of size $s.$ Then the probability
that $\mathbf{A}_{i}\cap J=B_{i}$ for each $i$ is $\Theta\left(\left(\frac{1}{n}\right)^{s}\right).$
\end{lem}

\begin{proof}
Let $\mathbf{C}$ be the center of $\left\{ \mathbf{A}_{1},\ldots,\mathbf{A}_{h}\right\} $,
let $\mathbf{C}_{i}=\mathbf{C}\cap\mathbf{A}_{i},$ and let $\mathbf{D}_{i}=\mathbf{A}_{i}\backslash\mathbf{C}$
for each $i$. Similarly, let $C$ be the center of $\left\{ B_{1},\ldots,B_{h}\right\} ,$
write $C_{i}=B_{i}\cap C,$ and $D_{i}=B_{i}\backslash C$. We define
the following three events:

\begin{itemize}
\item We let $E_{0}$ be the event that $\mathbf{C}\cap J=C$
\item We let $E_{1}^{i}$ be the event that $\mathbf{C}_{i}\cap J=C_{i}$. 
\item We let $E_{2}^{i}$ be the event that \textbf{$\mathbf{D}_{i}\cap J=D_{i}$}.
\end{itemize}
Note that the event that $\mathbf{A}_{i}\cap J=B_{i}$ for each $i$
is the intersection of all the above events. The lemma will follow
from the fact that $E_{0}$ occurs with probability $\frac{\binom{n-j}{c-d}}{\binom{n}{c}}=\Theta\left(\frac{1}{n^{s}}\right),$
once we show that the other events occur with conditional probability
$\Theta\left(1\right)$ given that $E_{0}$ holds. 

Since $\left|\mathbf{C}\right|=c$ is constant, the event $\bigcap E_{1}^{i}$
occur with conditional probability $\Theta\left(1\right)$ given $E_{0}$. 

We now complete the proof by showing that $\Pr\left[\bigcap_{i=1}^{h}E_{2}^{i}|E_{1}^{1},\ldots,E_{1}^{h},E_{0}\right]$
is $\Theta\left(1\right).$ 

Note that $\mathbf{D}_{1},\ldots,\mathbf{D}_{h}$ may be chosen by
first taking a set $\mathbf{D}_{1}\sim\binom{\left[n\right]\backslash C_{1}}{k-\left|\mathbf{C}_{1}\right|},$
then taking a set $\mathbf{D}_{2}\sim\binom{\left[n\right]\backslash\left(\mathbf{C}\cup\mathbf{D}_{1}\right)}{k-\left|\mathbf{C}_{2}\right|},$
and so on until we choose a set $\mathbf{D}_{h}\sim\binom{\left[n\right]\backslash\left(\mathbf{C}\cup\bigcup_{i=1}^{h-1}\mathbf{D}_{i}\right)}{k-\left|\mathbf{C}_{h}\right|}.$
Since for each $i$ the term $\frac{k-\left|\mathbf{C}_{i}\right|}{n-\left|\mathbf{C}\right|-\sum_{i=1}^{h}\left|\mathbf{D}_{h}\right|}$
is bounded away from 0 to 1, we have 
\[
\Pr\left[\forall i:\,\mathbf{D}_{i}\cap J\mathbf{=}D_{i}\right]=\Theta\left(1\right).
\]
This completes the proof. 
\end{proof}
We now prove the following restatement of Theorem \ref{thm:Genneral removal Lemma}.
\begin{thm*}
For each $h,c,\epsilon,s\in\mathbb{N},$ there exists $\delta,j>0$,
such that the following holds. Let $\h$ be a hypergraph with $h$
edges whose center is of size $c$. Let 
\[
\frac{1}{\delta}\le k\le\left(\frac{1}{h}-\epsilon\right)n,
\]
 and let $\mathcal{F}$ be a $\frac{\delta}{n^{s}}$-almost $\mathcal{H}$-free
family. Then $\mathcal{F}$ is $\epsilon$-essentially contained in
an $\left(\mathcal{H},s\right)$-free $j$-junta.
\end{thm*}
\begin{proof}[Proof of Theorem \ref{thm:Genneral removal Lemma}]
 Let $\delta_{1}=\delta_{1}\left(h,c,\epsilon,s\right)$ be sufficiently
small, let $j=j\left(\delta_{1}\right)$ be sufficiently large, and
let $\delta=\delta\left(j\right)$ be sufficiently small. By Theorem
\ref{thm:Counting expanded hypergraphs} we may take $\mathcal{J}$
to be $\varnothing$ if $\frac{k}{n}<\delta_{1},$ provided that $\delta_{1}$
is sufficiently small. So suppose that $\frac{k}{n}>\delta_{1}.$
By Theorem \ref{thm:k-uniform regularity}, there exists a set $J$
and a family $\mathcal{\g}\subseteq\p\left(J\right),$ such that $\f$
is $\epsilon$-essentially contained in $\j:=\left\langle \g\right\rangle $,
and such that for each $B\in\g$ the family $\mathcal{F}_{J}^{B}$
is $\left(\left\lceil \frac{1}{\delta_{1}}\right\rceil ,\delta_{1}\right)$-regular
and $\mu\left(\f_{J}^{B}\right)\ge\frac{\epsilon}{2}$. 

\textbf{Showing that $\mathcal{J}$ is $\left(\mathcal{H},s\right)$-free.}

By Lemma \ref{lem:junta is h d free meaning} it is enough to show
that $\mathcal{\mathcal{G}}$ does not contain a trace of $\mathcal{H}$
in $\mathcal{G}$ whose center is of size at most $s.$ Suppose on
the contrary that $\left\{ C_{1},\ldots,C_{h}\right\} $ is a trace
of $\mathcal{H}$ in $\mathcal{G}$ whose center is of size at most
$s$. Then there exist sets $B_{1},\ldots,B_{h}\subseteq\left[n\right]\backslash J,$
such that $\mathcal{H}'=\left\{ C_{1}\cup B_{1},\ldots,C_{h}\cup B_{h}\right\} $
is a copy of $\mathcal{H}.$ Let $\left\{ \mathbf{H}_{1},\ldots,\mathbf{H}_{h}\right\} $
be a random copy of $\mathcal{H}$ on $\left[n\right].$ By Lemma
\ref{lem:removal expanded calculation} we have 
\[
\Pr\left[\forall i:\,\mathbf{H}_{i}\cap J=C_{i}\right]=\Omega_{\delta_{1},j}\left(\frac{1}{n^{s}}\right),
\]
 and we have 
\[
\Pr\left[\left\{ \mathbf{H}_{1},\ldots,\mathbf{H}_{h}\right\} \subseteq\mathcal{F}|\,\forall i:\,\mathbf{H}_{i}\cap J=C_{i}\right]=\Pr\left[\forall i:\,\mathbf{H}_{i}\backslash J\in\mathcal{F}_{J}^{C_{i}}\right].
\]

Now note the families $\f_{J}^{C_{i}}$ are $\left(\left\lceil \frac{1}{\delta_{1}}\right\rceil ,\delta_{1}\right)$-regular
and have measure greater than $\frac{\epsilon}{2}$. Provided that
$\delta_{1}$ is small enough, we may apply Theorem \ref{thm:Counting expanded hypergraphs}
with $\epsilon/2$ instead of $\epsilon$, the hypergraph $\left(\boldsymbol{H}_{1}\backslash J,\ldots,\boldsymbol{H}_{h}\backslash J\right)$,
and $\delta_{1}$ instead of $\delta$, to obtain 
\[
\Pr\left[\forall i:\,\mathbf{H}_{i}\backslash J\in\mathcal{F}_{J}^{C_{i}}\right]\ge\delta_{1}.
\]
Putting everything together, we obtain 
\begin{align*}
\Pr\left[\left\{ \mathbf{H}_{1},\ldots,\mathbf{H}_{h}\right\} \subseteq\mathcal{F}\right] & \ge\Pr\left[\left\{ \mathbf{H}_{1},\ldots,\mathbf{H}_{h}\right\} \subseteq\mathcal{F}|\,\forall i:\,\mathbf{H}_{i}\cap J=C_{i}\right]\cdot\Pr\left[\forall i:\,\mathbf{H}_{i}\cap J=C_{i}\right]\\
 & =\Omega_{\delta_{1},\epsilon}\left(\frac{1}{n^{s}}\right).
\end{align*}
 Provided that $\delta$ is small enough, this contradicts the hypothesis.
\end{proof}
Note that in the proof of Theorem \ref{thm:Genneral removal Lemma},
the hypothesis $k>\frac{1}{\delta}$ is not needed in the case where
$\mathcal{H}$ is a matching as we may apply Theorem \ref{thm:Counting matchings}
rather than Theorem \ref{thm:Counting expanded hypergraphs}. 

We shall now prove Proposition \ref{Prop:converse to removal lemma}.
We restate it for the convenience of the reader.
\begin{prop*}
For each constants $h,c,j,s\in\mathbb{N},$ there exists a constant
$C>0$, such that the following holds. Let $\h$ be a hypergraph with
$h$ edges whose center is of size $c$. Let $\epsilon n\le k\le n\left(\frac{1}{h}-\epsilon\right),$
and let $\j$ be some $\left(\mathcal{H},s\right)$-free $j$-junta.
Then $\mathcal{J}$ is $\frac{C}{n^{s+1}}$-almost $\mathcal{H}$-free.
\end{prop*}
\begin{proof}
Let $\left\{ \mathbf{A}_{1},\ldots,\mathbf{A}_{h}\right\} $ be a
random copy of $\mathcal{H},$ and let $J$ be a set of size at most
$j$, such that $\mathcal{J}$ depends on $J$. Let $\mathbf{C}_{1}=\mathbf{A}_{1}\cap J,\ldots,\mathbf{C}_{h}=\mathbf{A}_{h}\cap J.$
Since $\mathcal{J}$ is $\left(\mathcal{H},s\right)$-free, we obtain
by Lemma \ref{lem:junta is h d free meaning} that for any copy $\left\{ A_{1},\ldots,A_{h}\right\} $
of $\mathcal{H}$ in $\mathcal{J}$ the center of $\left\{ A_{1}\cap J,\ldots,A_{h}\cap J\right\} $
is of size at least $s+1$. Now for any hypergraph $C_{1},\ldots,C_{h}$
of center of size at least $s+1$ we have $\Pr\left[\forall i:\,\mathbf{C}_{i}=C_{i}\right]=O\left(\frac{1}{n^{s+1}}\right)$
by Lemma \ref{lem:removal expanded calculation}. Since there are
only a constant number of subsets $\left\{ C_{1},\ldots,C_{h}\right\} \subseteq\mathcal{P}\left(J\right),$
we obtain that$\left\{ \mathbf{A}_{1},\ldots,\mathbf{A}_{h}\right\} $
is a copy of $\mathcal{H}$ with probability at most $O\left(\frac{1}{n^{s+1}}\right).$
This completes the proof of the proposition.
\end{proof}
Finally we shall prove Theorem \ref{thm:Removal for matchings}, which
we restate for the convenience of the reader. 
\begin{thm*}
For each $h,d\in\mathbb{N},\epsilon>0$
there exist $C,\delta>0$ such that if $C\le k\le\left(\frac{1}{h}-\epsilon\right)n$,
and $\mathcal{H}$ is a $k$-uniform $\left(h,d\right)$-expanded
hypergraph, then the following statements hold.
\begin{enumerate}
\item If the family $\mathcal{F}$ is $\delta$-almost $\mathcal{H}$-free,
then $\f$ is $\epsilon$-essentially contained in an $\mathcal{M}_{h}$-free
family. 
\item Conversely, if the family $\mathcal{F}$ is $\delta$-essentially
contained in an $\mathcal{M}_{h}$-free family, then $\f$ is $\epsilon$-almost
$\h$-free. 
\end{enumerate}
\end{thm*}
\begin{proof}[Proof of Theorem \ref{thm:Removal for matchings}]
 (1) $\implies$ (2) follows by applying Theorem \ref{thm:Genneral removal Lemma}
with $s=0,$ noting that a family is $\left(\mathcal{H},0\right)$-free
if and only if it is free of a matching. We now show the converse
implication. 

Suppose that (2) holds. By Theorem (\ref{thm:Genneral removal Lemma})
$\f$ is $\frac{\epsilon}{h+1}$-essentially contained in an $\mathcal{M}_{h}$-free
junta, provided that $\delta$ is small enough. Let $\left\{ \mathbf{A}_{1},\ldots,\mathbf{A}_{h}\right\} $
be a random copy of $\mathcal{H}.$ Note that the event $\left\{ \mathbf{A}_{1},\ldots,\mathbf{A}_{h}\right\} \subseteq\mathcal{F}$
can occur only if for some $i$ we have $\mathbf{A}_{i}\in\mathcal{J}\backslash\mathcal{F},$
or if $\left\{ \mathbf{A}_{1},\ldots,\mathbf{A}_{h}\right\} \subseteq\mathcal{J}.$
So a union bound implies that it is enough to show that each of these
events occurs with probability $<\frac{\epsilon}{h+1}.$ 

By Proposition \ref{Prop:converse to removal lemma} a random copy
of $\mathcal{H}$ lies in $\mathcal{J}$ with probability $O\left(\frac{1}{n}\right)<\frac{\epsilon}{h+1},$
provided that $C$ is sufficiently large to imply the needed lower
bound on $n$. Moreover, each $\mathbf{A}_{i}$ is uniformly distributed
in $\binom{\left[n\right]}{k}$. Therefore, for each $i$ the probability
that $\mathbf{A}_{i}$ is in $\mathcal{F}$ but not in $\mathcal{J}$
is at most $\frac{\epsilon}{h+1}$. This completes the proof of the
theorem. 
\end{proof}

\section*{Acknowledgments} 

I would like to thank Nathan Keller for providing many helpful comments
and suggestions, which tremendously improved the exposition of the
paper.

\bibliographystyle{amsplain}

\begin{thebibliography}{10}
  \bibitem{abdullah2015directed} Amirali Abdullah, and Suresh
    Venkatasubramanian. \newblock A directed isoperimetric inequality
    with application to bregman near   neighbor lower
    bounds. \newblock In {\em Proceedings of the Forty-Seventh Annual
      ACM on Symposium on   Theory of Computing}, pages 509--518. ACM,
    2015.

    
\bibitem{ahlberg2014noise} Daniel Ahlberg, Erik Broman, Simon Griffiths, and Robert Morris. \newblock Noise sensitivity in continuum percolation. \newblock {\em Israel Journal of Mathematics}, 201(2):847--899, 2014.

\bibitem{ahlswede1996complete} Rudolf Ahlswede and Levon~H. Khachatrian. \newblock The complete nontrivial-intersection theorem for systems of finite   sets. \newblock {\em journal of combinatorial theory, Series A}, 76(1):121--138,   1996.

  \bibitem{balogh2015independent} J{\'o}zsef Balogh, Robert Morris,
    and Wojciech Samotij. \newblock Independent sets in
    hypergraphs. \newblock {\em Journal of the American Mathematical
      Society}, 28(3):669--709,   2015.

    \bibitem{beckner1975inequalities} William Beckner. \newblock
      Inequalities in fourier analysis. \newblock {\em Annals of
        Mathematics}, pages 159--182, 1975.
      
\bibitem{ben1990collective} Michael Ben-Or and Nathan
  Linial. \newblock Collective coin flipping. \newblock {\em
    randomness and computation}, 5:91--115, 1990.
  
\bibitem{bollobas1987threshold} B{\'e}la Bollob{\'a}s and Arthur~G Thomason. \newblock Threshold functions. \newblock {\em Combinatorica}, 7(1):35--38, 1987.

  \bibitem{bonami1970etude} Aline Bonami. \newblock {\'E}tude des coefficients de fourier des fonctions de $l^p(g)$. \newblock In {\em Annales de l'institut Fourier}, volume~20, pages 335--402,   1970.
\bibitem{borell1985geometric} Christer Borell. \newblock Geometric bounds on the ornstein-uhlenbeck velocity process. \newblock {\em Probability Theory and Related Fields}, 70(1):1--13, 1985.


  \bibitem{BourgainKalai19} Jean Bourgain and Gil Kalai. \newblock
    Influences of variables and threshold intervals under group
    symmetries. \newblock {\em Geometric and Functional Analysis},
    7(3):438--461, 1997.
    
\bibitem{conlon2013graph} David Conlon and Jacob Fox. \newblock Graph
  removal lemmas. \newblock {\em Surveys in combinatorics}, 1(2):3,
  2013.

 \bibitem{conlon2016combinatorial} David Conlon and Timothy
   Gowers. \newblock Combinatorial theorems in sparse random
   sets. \newblock {\em Annals of Mathematics}, 184:367--454, 2016.
   
   \bibitem{Das2016Removal} Das Shagnik and Tuan Tran. \newblock Removal
  and Stability for {E}rd{\H{o}}s--{K}o--{R}ado. \newblock {\em SIAM
    Journal on Discrete Mathematics}, 30(2): 1102-1114, 2016.

\bibitem{deza1978intersection} Mikhail Deza, Paul Erd{\H{o}}s, and
  Peter Frankl. \newblock Intersection properties of systems of finite
  sets. \newblock {\em Proc. London Math. Soc.}, 36(2):369--384, 1978.
  
 \bibitem{dinur2009intersecting} Irit Dinur and Ehud
   Friedgut. \newblock Intersecting families are essentially contained
   in juntas. \newblock {\em Combinatorics, Probability and
     Computing}, 18(1-2):107--122,   2009.
   
\bibitem{dinur2009conditional} Irit Dinur, Elchanan Mossel, and Oded Regev. \newblock Conditional hardness for approximate coloring. \newblock {\em SIAM Journal on Computing}, 39(3):843--873, 2009.

\bibitem{ellis2016stability} David Ellis, Nathan Keller, and Noam Lifshitz. \newblock Stability versions of {E}rd{\H{o}}s--{K}o--{R}ado type theorems, via   isoperimetry. \newblock {\em arXiv preprint arXiv:1604.02160}, 2016.

\bibitem{ellis2016stabilityfor} David Ellis, Nathan Keller, and Noam Lifshitz. \newblock Stability for the complete intersection theorem, and the forbidden   intersection problem of {E}rd{\H{o}}s and {S}{\'{o}}s. \newblock {\em arXiv preprint arXiv:1604.06135}, 2017.
  
\bibitem{erdHos1965problem} Paul Erd{\H{o}}s.
  \newblock A problem on
  independent $r$-tuples.
   \newblock {\em Ann. Univ. Sci. Budapest}, 8:93--95, 1965.

   \bibitem{erdHos1975problems} Paul Erd{\H{o}}s. \newblock Problems
     and results in graph theory and combinatorial analysis. \newblock
     {\em Proc. British Combinatorial Conj., 5th}, pages 169--192,
     1975.

     \bibitem{erdos1961intersection} Paul Erd{\H{o}}s, Chao Ko, and Richard Rado. \newblock Intersection theorems for systems of finite sets. \newblock {\em The Quarterly Journal of Mathematics}, 12(1):313--320, 1961.

     \bibitem{filmus2016ahlswede} Yuval Filmus. \newblock Ahlswede-Khachatrian theorems: Weighted, infinite, and Hamming. \newblock {\em arXiv preprint arXiv:1610.00756}, 2016.

\bibitem{fox2011new} Jacob Fox. \newblock A new proof of the graph removal lemma. \newblock {\em Annals of Mathematics}, pages 561--579, 2011.

\bibitem{frankl18erdos} Peter Frankl. \newblock The {E}rd{\H{o}}s--{K}o--{R}ado theorem is true for $n= ckt.$   combinatorics (proc. fifth hungarian colloq., keszthey, 1976), vol. i,   365--375. \newblock In {\em Colloq. math. Soc. J{\'a}nos Bolyai}, volume~18.

\bibitem{frankl1977families} Peter Frankl. \newblock On families of finite sets no two of which intersect in a singleton. \newblock {\em Bulletin of the Australian Mathematical Society},   17(1):125--134, 1977.

\bibitem{frankl1987erdos} Peter Frankl. \newblock Erd{\"o}s-ko-rado theorem with conditions on the maximal degree. \newblock {\em Journal of Combinatorial Theory, Series A}, 46(2):252--263,   1987.

\bibitem{frankl2013improved} Peter Frankl. \newblock Improved bounds for erd{\H{o}}s matching conjecture. \newblock {\em Journal of Combinatorial Theory, Series A}, 120(5):1068--1072,   2013.

\bibitem{frankl1985forbidding} Peter Frankl and Zolt{\'a}n F{\"u}redi. \newblock Forbidding just one intersection. \newblock {\em Journal of Combinatorial Theory, Series A}, 39(2):160--176,   1985.

\bibitem{frankl1987exact} Peter Frankl and Zolt{\'a}n F{\"u}redi. \newblock Exact solution of some {T}ur{\'a}n-type problems. \newblock {\em Journal of Combinatorial Theory, Series A}, 45(2):226--262,   1987.

\bibitem{frankl1987forbidden} Peter Frankl and Vojt{\v{e}}ch R{\"o}dl. \newblock Forbidden intersections. \newblock {\em Transactions of the American Mathematical Society},   300(1):259--286, 1987.

\bibitem{friedgut1998boolean} Ehud Friedgut. \newblock Boolean functions with low average sensitivity depend on few   coordinates. \newblock {\em Combinatorica}, 18(1):27--35, 1998.

\bibitem{friedgut1999sharp} Ehud Friedgut. \newblock Sharp thresholds of graph properties, and the $k$-sat problem (with   an appendix by {J}ean {B}ourgain). \newblock {\em Journal of the American Mathematical Society}, 12(4):1017--1054,   1999.

\bibitem{friedgut1996every} Ehud Friedgut and Gil Kalai. \newblock Every monotone graph property has a sharp threshold. \newblock {\em Proceedings of the American mathematical Society},   124(10):2993--3002, 1996.

\bibitem{friedgut2017kneser} Ehud Friedgut and Oded Regev. \newblock Kneser graphs are like swiss cheese. \newblock {\em Discrete Analysis}, 2:0--18, 2018.

\bibitem{furedi2014exact} Zolt{\'a}n F{\"u}redi, Tao Jiang, and Robert Seiver. \newblock Exact solution of the hypergraph {T}ur{\'a}n problem for k-uniform   linear paths. \newblock {\em Combinatorica}, 34(3):299--322, 2014.

\bibitem{gowers2006quasirandomness} Timothy Gowers. \newblock Quasirandomness, counting and regularity for 3-uniform hypergraphs. \newblock {\em Combinatorics, Probability and Computing}, 15(1-2):143--184,   2006.

\bibitem{gowers2007hypergraph} Timothy Gowers. \newblock Hypergraph regularity and the multidimensional Szemer{\'e}di theorem. \newblock {\em Annals of Mathematics}, pages 897--946, 2007.

\bibitem{gross1975logarithmic} L.~Gross. \newblock Logarithmic sobolev inequalities. \newblock {\em American J. Math.}, 97:1061--1083, 1975.

\bibitem{hatami2012structure} Hamed Hatami. \newblock A structure theorem for boolean functions with small total   influences. \newblock {\em Annals of Mathematics}, 176(1):509--533, 2012.

\bibitem{jones2016noisy} Chris Jones. \newblock A noisy-influence regularity lemma for boolean functions. \newblock {\em arXiv preprint arXiv:1610.06950}, 2016.

\bibitem{keevash2011hypergraph} Peter Keevash. \newblock Hypergraph {T}ur{\'a}n problems. \newblock {\em Surveys in combinatorics}, 392:83--140, 2011.

\bibitem{keevash2006set} Peter Keevash, Dhruv Mubayi, and
  Richard~M. Wilson. \newblock Set systems with no singleton
  intersection.
  \newblock {\em SIAM Journal on Discrete Mathematics}, 20(4):1031--1041, 2006.

\bibitem{keller2017junta} Nathan Keller and Noam Lifshitz. \newblock The junta method for hypergraphs and chv{\'{a}}tal's simplex   conjecture. \newblock {\em arXiv preprint arXiv:1707.02643}, 2017.

\bibitem{kostochka2015turan} Alexandr Kostochka, Dhruv Mubayi, and Jacques Verstra{\"e}te. \newblock Tur{\'a}n problems and shadows {I}: paths and cycles. \newblock {\em Journal of Combinatorial Theory, Series A}, 129:57--79, 2015.

\bibitem{Kostochka2015} Alexandr Kostochka, Dhruv Mubayi, and Jacques Verstra{\"e}te. \newblock Tur{\'a}n problems and shadows {III}: expansions of graphs. \newblock {\em SIAM Journal on Discrete Mathematics}, 29(2):868--876, 2015.

\bibitem{kostochka2017turan} Alexandr Kostochka, Dhruv Mubayi, and Jacques Verstra{\"e}te. \newblock Tur{\'a}n problems and shadows {II}: trees. \newblock {\em Journal of Combinatorial Theory, Series B}, 122:457--478, 2017.

\bibitem{mantel1907problem} Willem Mantel. \newblock Problem 28. \newblock {\em Wiskundige Opgaven}, 10(60-61):320, 1907.

\bibitem{mossel2010gaussian} Elchanan Mossel. \newblock Gaussian bounds for noise correlation of functions. \newblock {\em Geometric and Functional Analysis}, 19(6):1713--1756, 2010.

\bibitem{mossel2017majority} Elchanan Mossel. \newblock Majority is asymptotically the most stable resilient function. \newblock {\em arXiv preprint arXiv:1704.04745}, 2017.

\bibitem{mossel2010noise} Elchanan Mossel, Ryan O'Donnell, and Krzysztof Oleszkiewicz. \newblock Noise stability of functions with low influences: Invariance and   optimality. \newblock {\em Annals of Mathematics}, pages 295--341, 2010.

\bibitem{mubayi2016survey} Dhruv Mubayi and Jacques Verstra{\"e}te. \newblock A survey of {T}ur{\'a}n problems for expansions. \newblock In {\em Recent Trends in Combinatorics}, pages 117--143. Springer,   2016.

\bibitem{nagle2006counting} Brendan Nagle, Vojt{\v{e}}ch R{\"o}dl, and Mathias Schacht. \newblock The counting lemma for regular $k$-uniform hypergraphs. \newblock {\em Random Structures \& Algorithms}, 28(2):113--179, 2006.

\bibitem{o2014analysis} Ryan O'Donnell. \newblock {\em Analysis of boolean functions}. \newblock Cambridge University Press, 2014.

\bibitem{rodl2004regularity} Vojt{\v{e}}ch R{\"o}dl and Jozef Skokan. \newblock Regularity lemma for $k$-uniform hypergraphs. \newblock {\em Random Structures \& Algorithms}, 25(1):1--42, 2004.

\bibitem{russo1982approximate} Lucio Russo. \newblock An approximate zero-one law. \newblock {\em Probability Theory and Related Fields}, 61(1):129--139, 1982.

\bibitem{ruzsa1978triple} Imre~Z. Ruzsa and Endre Szemer\'{e}di. \newblock Triple systems with no six points carrying three triangles. \newblock {\em Colloq. Math. Soc. J\'{a}nos Bolyai}, 18:939--945, 1978.

\bibitem{samotij2014stability} Wojciech Samotij. \newblock Stability results for random discrete structures. \newblock {\em Random Structures \& Algorithms}, 44(3):269--289, 2014.

\bibitem{saxton2015hypergraph} David Saxton and Andrew Thomason. \newblock Hypergraph containers. \newblock {\em Inventiones mathematicae}, 201(3):925--992, 2015.

\bibitem{schacht2016extremal} Mathias Schacht. \newblock Extremal results for random discrete structures. \newblock {\em Annals of Mathematics}, 184(2):333--365, 2016.

\bibitem{simonovits1968method} Mikl{\'o}s Simonovits. \newblock A method for solving extremal problems in graph theory, stability   problems. \newblock In {\em Theory of Graphs (Proc. Colloq., Tihany, 1966)}, pages   279--319, 1968.

\bibitem{tao2006variant} Terence Tao. \newblock A variant of the hypergraph removal lemma. \newblock {\em Journal of combinatorial theory, Series A}, 113(7):1257--1280,   2006.

\bibitem{turan1941basic} P\'{a}l Tur\'{a}n. \newblock On an extremal problem in graph theory (in {H}ungarian). \newblock {\em Mat. Fiz. Lapok}, 48:436--452, 1941.

\end{thebibliography}


\begin{dajauthors}
  \begin{authorinfo}[noam]
  Noam Lifshitz\\
  Einstein institute of Mathematics\\
  Jerusalem, Israel\\
  noamlifshitz\imageat{}gmail\imagedot{}com  
\end{authorinfo}

\end{dajauthors}

\end{document}